\documentclass[12pt]{amsart}

\usepackage{graphicx}
\usepackage{color}
\usepackage{amssymb}
\usepackage{enumitem}
\usepackage[driverfallback=hypertex]{hyperref}
\usepackage[utf8]{inputenc} 
\input xy\huge
\xyoption{all}
\usepackage{comment}

\newcommand{\M}{\mathcal{M}}

\newcommand{\R}{\mathbb{R}}
\newcommand{\C}{\mathbb{C}}
\newcommand{\T}{\mathbb{T}}

\newcommand{\Z}{\mathbb{Z}}
\newcommand{\N}{\mathbb{N}}
\newcommand{\K}{\mathbb{K}}
\newcommand{\cL}{\mathcal{L}}
\newcommand{\Pp}{\mathbb{P}}
\newcommand{\Ll}{\mathcal{L}}
\newcommand{\F}{\mathcal{F}}

\newcommand{\CO}{\mathcal{CO}}

\newcommand{\restrict[3]}{\left. {#2}\right|_{#3}}

\DeclareMathOperator{\id}{Id}

\DeclareMathOperator{\re}{Re}
\DeclareMathOperator{\im}{Im}
\DeclareMathOperator{\val}{val}

\DeclareMathOperator{\crit}{Crit}

\DeclareMathOperator{\Tt}{T^*\T}
\setlist[enumerate,1]{label={(\alph*)}}
\setlist[enumerate,2]{label={(\roman*)}}

\newtheorem{tm}{Theorem}[section]
\newtheorem{tmc}{Theorem\slash Construction}[section]

\newtheorem{lm}[tm]{Lemma}
\newtheorem{cy}[tm]{Corollary}

\theoremstyle{definition}
\newtheorem{df}[tm]{Definition}

\theoremstyle{remark}
\newtheorem{rem}[tm]{Remark}
\newtheorem{ex}[tm]{Example}
\newtheorem*{exs}{Examples}
\title[The wrapped Fukaya category for semi-toric Calabi-Yau]{The wrapped Fukaya category for semi-toric SYZ fibrations}
\author{Yoel Groman}

\begin{document}

\maketitle
\begin{abstract}
We introduce the wrapped Donaldson-Fukaya category of a (generalized) semi-toric SYZ fibration with Lagrangian section satisfying a tameness condition at infinity. Examples include the Gross fibration on the complement of an anti-canonical divisor in a toric Calabi-Yau 3-fold. We compute the wrapped Floer cohomology of a Lagrangian section and find that it is the algebra of functions on the Hori-Vafa mirror. The latter result is the key step in proving homological mirror symmetry for this case. The techniques developed here allow the construction in general of the wrapped Fukaya category on an open Calabi-Yau manifold carrying an SYZ fibration with nice behavior at infinity. We discuss the relation of this to the algebraic vs analytic aspects of mirror symmetry.
\end{abstract}
\section{Introduction}
Toric Calabi-Yau varieties have long been a testing ground for ideas in mirror symmetry and have generated a large literature, both in physics and mathematics. Most importantly for us, if one removes a smooth anti-canonical divisor $D$ from a toric Calabi-Yau $X$, one obtains a manifold $M=X\setminus D$ carrying a special Lagrangian torus fibration $\Ll:M\to\R^n$ \cite{Gross}, referred to later as the Gross fibration. The Gross fibration has been used in \cite{AurouxSYZ,CSN,AAK16} to construct an SYZ mirror to $M$. This setting is one of a handful in which such an explicit construction is known. However, a full understanding of homological mirror symmetry requires the wrapped Fukaya category\cite{AAEKO}. To the author's knowledge, the wrapped Fukaya category has to date only been considered in connection with manifolds that are Liouville \cite{AbouSeidel} or at least conical at infinity \cite{RitterSmith17}. On the other hand, the manifolds $M=X\setminus D$ are generally not conical at infinity. The present paper fills this gap. We introduce the wrapped Fukaya category associated with a semi-toric SYZ fibration. We then carry out the main step towards its full computation in dimension $n\leq 3$ by computing the wrapped Floer cohomology of a Lagrangian section. This in turn will be used in a forthcoming paper to prove homological mirror symmetry and to intrinsically calculate the mirror map. A secondary aim of this paper is to illustrate the use of the techniques of \cite{Groman15} in connection with SYZ mirror symmetry. While we restrict attention to the semi-toric case, the construction should be relevant for any well behaved SYZ fibration, a notion introduced herein.

\subsection{Semi-toric SYZ fibrations}
\begin{df}\label{dfSemiSYZ}
A semi-toric SYZ fibration is a pair $(M,\Ll=(\mu,H))$ as follows. $M=(M,\omega)$ is a $2n$-dimensional symplectic manifold satisfying $c_1(M)=0$ and carrying a Hamiltonian $\T^{n-1}$ action with moment map $\mu:M\to \R^{n-1}$. $H:M\to\R$ is a smooth function which commutes with $\mu$. We assume that the singular orbits are non-degenerate of codimension $2$ and of positive type\footnote{See Definition \ref{dfSTSYZ} below. In the literature\cite{pelayo13}, such integrable systems have been referred to as generalized semi-toric. This is due to the fact that in the SYZ setting, the moment map for the torus action is generally not proper, even though the fibration as whole is.}. We say that $(M,\Ll)$ is simple if the set of points with non-trivial isotropy is connected.
\end{df}
\begin{ex}
The Gross fibration is a simple semi-toric SYZ fibration. So is the cotangent bundle of a torus. On the other hand, toric Calabi-Yau varieties with their toric fibration over the Delzant polyhedron do \textit{not} fall under Definition \ref{dfSemiSYZ}. As an aside, by Lemma 9.11 in \cite{Groman15} the Fukaya category of a toric Calabi-Yau variety is trivial. Thus, for the purposes of mirror symmetry, semi-toric SYZ fibrations naturally follow cotangent bundles of tori in order of complexity.
\end{ex}
As motivation for considering semi-toric SYZ fibrations observe that an understanding of mirror symmetry requires the study of Floer theory of torus fibrations with singular fibres. In dimension three, a fairly general class\cite{Gross01} is with singularities occurring along a trivalent-valent graph such that, denoting the ranks of the middle homology groups by $(i,j)$, over the edges the singularity is of type $(2,2)$ and over the vertices the singularity is of one of two types, the positive, or type $(1,2)$, and the negative, or type $(2,1)$. In general, these singularities give rise to wall crossing phenomena which significantly complicate the picture due to scattering of walls. Semi-toric SYZ fibrations are those for which there are only positive singularities, and moreover, the local $T^{n-1}$ symmetries of the vertices are coherent. This leads to an intermediate level of difficulty where there are wall crossing phenomena, but no scattering of walls.
\begin{rem}
One can also consider the fibrations with  only negative singularities and with global $S^1$ symmetry. These are precisely the mirrors of semi-toric SYZ fibrations. Their wrapped Fukaya category have been studied in \cite{KPU16,KPU162}\footnote{The author apologizes for any omissions}. Note that the latter are affine varieties, and in particular have no non-trivial closed Gromov-Witten invariants.
\end{rem}
\subsection{The wrapped Donaldson-Fukaya category}
In our formulation, the input for the wrapped Fukaya category is a semi-toric SYZ fibration $(M,\Ll)$ together with a Lagrangian section $\sigma$ and a Landau-Ginzburg potential $(\pi,J)$ which we now turn to define.

\begin{df}\label{dfLgPotential}
A Landau-Ginzburg potential for $(M,\Ll)$ is a pair $(\pi,J)$ where $J$ is an almost complex structure tamed by $\omega$ and $\pi: M\to\C^*\subset\Pp^1$ is a $J$-holomorphic function such that $\crit(\Ll)\subset\pi^{-1}(1)$. A Lagrangian $L$ is said to be tautologically unobstructed with respect to $J$ if $L$ does not meet $\pi^{-1}(1)$ and the intersection number of $\pi^{-1}(1)$ with any element of $\pi_2(M,L)$ is $0$.  
\end{df}

In \S\ref{SecWrap} we introduce the notion of a tame semi-toric SYZ fibration which means that the fibration is well behaved at infinity. \textit{Henceforth, we restrict attention to the case $n\leq3$.}
\begin{tmc}\label{tmc}
To any tame semi-toric SYZ fibration $(M,\Ll)$ there is naturally associated the following data
\begin{enumerate}
 \item A semi-group $\mathcal{H}_\Ll\subset C^{\infty}(M)$.
 \item A graded unital BV algebra
    \[
    SH^*(M;\Ll):=\varinjlim_{H\in\mathcal{H}_\Ll} HF^*(M,H).
    \]
 \item
    A connected set $\Pi$ of Landau-Ginzburg potentials
$(\pi,J)$.
\end{enumerate}
To any $(\pi,J) \in \Pi$ and $(\pi,J)$-tautologically unobstructed Lagrangian section $\sigma$ we can associate
\begin{enumerate}
\item  A collection of admissible Lagrangians $A_{\pi,J}$, consisting of $(\pi,J)$-tautologically unobstructed Lagrangians that outside a compact set arise from $\sigma$ by a conormal construction.
\item For any $L_0,L_1 \in A_{\pi,J}$, a cofinal subset $\mathcal{H}_{L_0,L_1,\pi,J} \subset \mathcal{H}_{\mathcal L}$ and a graded vector space
\[
HW^*(L_0,L_1)  = \lim_{H \in \mathcal{H}_{L_0,L_1,\pi,J} } HF^*(L_0,L_1;H).
\]
\item For any $L_0,L_1,L_2 \in A_{\pi,J}$, a composition map
\[
HW^*(L_0,L_1) \otimes HW(L_1,L_2) \to HW^*(L_0,L_2)
\]
satisfying associativity and, for $L_i=L$, unitality.
\end{enumerate}
The category thus defined for a choice of $\sigma,\pi,J,$ is referred to as the wrapped Donaldson-Fukaya category of $(M,\Ll,\sigma,\pi,J)$.

\end{tmc}
The proof of Theorem \ref{tmc} is carried out in  \S\ref{SecWrap} and \S\ref{SecWrapped} and wrapped up at the end of \S\ref{SecWrapped}.
%
\begin{rem}
The role of the Landau-Ginzburg potential in Theorem \ref{tmc} is to control disk bubbling. This allows us to implement a self-contained transversality package. The price is that we have to make a choice of LG potential. For a generic fixed admissible Lagrangian there are exactly two connected components of such. A fully invariant formulation requires considering objects which are pairs of Lagrangians and bounding cochains and applying the machinery of \cite{FOOO}. In that formulation, the choice of connected component corresponds to a certain choice of local coordinates on the SYZ mirror.

By contrast, the role of $\sigma$ appears to be required for singling out the algebraic structure on the mirror. It appears that it should be possible to eliminate the choice of $\sigma$ if one is willing to work in the rigid analytic category. We do not pursue this here.
\end{rem}


\begin{rem}
The $A_\infty$ refinement to the wrapped Fukaya category will be discussed in a forthcoming paper in connection with homological mirror symmetry.
\end{rem}

\subsection{Statement of the main results}
Without further notice we fix the data $M, \Ll,\pi,J,$ and $\sigma,$ and omit them from the notation.

Associated to a semi-toric SYZ fibration there is a generator of $H^1(M;\Z)$ which is, roughly, the class of a global angle coordinate defined away from the singular point of $\Ll$. See Theorem \ref{tmBasicTop} below. Thus, there is an extra grading on $SH^*(M)$ and $HW^*(L)$. We denote the degree of an element with respect to this grading by $|\cdot|$. Denote by $\Lambda_{nov}$ the universal Novikov field over $\C$. Namely,
\[
\Lambda_{nov}=\left\{\sum_{i=0}^{\infty}a_it^{\lambda_i}:a_i\in \C,\lambda_i\in\R,\lim_{i\to\infty}\lambda_i=\infty\right\}.
\]
By definition, $M$ carries a $\T^{n-1}$ action generated by the $\mu$ component of $\cL$. Henceforth we denote this torus by $G$. 
\begin{tm}\label{tmRoughform}
Henceforth, we restrict attention to the case $n\leq3$. 
\begin{enumerate}
\item There are natural isomorphisms
    \[
    SH^{0,0}(M)=HW^{0,0}(\sigma,\sigma)=\Lambda_{nov}[H_1(G;\Z)]\simeq\Lambda_{nov}[u_1^{\pm1},\dots,u_{n-1}^{\pm1}].
    \]
    and,
    \[
    SH^0(M)=HW^0(\sigma,\sigma).
    \]
\item\label{tmRoughformb}
    There are elements $x,y,g\in HW^0(\sigma,\sigma)$ with $|x|=1$, $|y|=-1$ and $|g|=0$ such that
    $HW^{0}(\sigma,\sigma)=\Lambda_{nov}[H_1(G;\Z),x,y]/xy-g$.
\item\label{tmRoughformc} $HW^i(\sigma,\sigma)=0$ for $i\neq 0$.
\end{enumerate}
\end{tm}

We now give a more explicit description of the Laurent polynomial $g$. For definiteness take $n=3$. The image of the lower dimensional orbits under $\mu$ gives rise to the toric diagram $\mu(\crit(\Ll))$ in the plane. It is a rational balanced graph and possesses a dual graph denoted by $\Delta^*$. For $\alpha\in\Delta^*$ denote by $f(\alpha)$ the distance to the origin of the line dual to $\alpha$ in the toric diagram. This is well defined up to on overall constant in $\R^2$
.

\begin{tm}[Cf. \cite{AAK16} Theorem 8.4]\label{tmLaurent0}
We have
\[
g=\sum_{\alpha\in \Delta^*}(1+c_{\alpha})t^{f(\alpha)}u^\alpha,
\]
where $c_{\alpha}\in\Lambda_{nov}$ such that $\val{c}_{\alpha}>0$. More precisely,
\[
c_\alpha=\sum_{A\in H_2(M;\Z)|\omega(A)> 0}n_At^{\omega(A)},
\]
for some integers $n_A$.
\end{tm}

\begin{rem}
The case $n=2$ of the above theorems is carried out in \cite{Pasc1}
\end{rem}
\begin{rem}
The numbers $n_A$ count certain configurations arising from the analytic expansions of the closed open map. These expansions in turn arise from a particular choice of a pair of charts coming from the LG potential. In a forthcoming paper we give a more intrinsic formulation in terms of Gromov-Witten invariants.
\end{rem}

\begin{rem}
Note that the choice of LG potential potentially only affects the definition of $HW(\sigma,\sigma)$, not that of $SH^*(M)$. Thus, a-posteriori, Theorem~\ref{tmRoughform} tells us that $HW(\sigma,\sigma)$ is independent of the choice of $(\pi,J)$ up to isomorphism. If one takes on board the theory of \cite{FOOO}, it should follow that different choices are canonically identified since a section is simply connected.
\end{rem}
A particular consequence of Theorem~\ref{tmLaurent0} is that $HW^0(\sigma,\sigma)$ is a smooth algebra over the Novikov field. This follows because $\Delta^*$ is the dual partition of the Newton polygon for a smooth tropical curve. See Corollary~\ref{corMuSmooth}. This fact is the key step in the proof of the following theorem which will be given in a forthcoming paper.

\begin{tm}[\textbf{Homological mirror symmetry}]\label{tmHMS}
There is an $A_\infty$ refinement $\mathcal{W}(M,\Ll)$ of the Donaldson-Fukaya category and we have an equivalence of categories between
$D^\pi\mathcal{W}(M,\Ll)$ and finitely generated modules over $HW(\sigma,\sigma)$.
\end{tm}
\subsection{Comparison to previous work}
This paper is closely related to and builds on \cite{Auroux07} and the later works \cite{AAK16,CSN}. Among other things, these works construct an SYZ mirror for $M$ using family Floer homology. The main contribution of the current work is to construct and compute the wrapped Fukaya category of $M$. We mention that in the case $n=2$, the wrapped Fukaya category is studied in \cite{Pasc1}. One way to see the importance of the wrapped Fukaya category is to observe that in Theorem \ref{tmHMS} the wrapped Floer homology of a Lagrangian section gives the full mirror, not just the mirror to the regular fibers. On the other hand, to obtain the same in the family Floer approach, it is necessary to study family Floer theory for immersed Lagrangians\footnote{This has been carried out in dimension $n=2$ in \cite{ERT18}. The higher dimensional cases fall under a work in progress by M. Abouzaid and Z. Sylvan.}. We also mention that  explaining enumerative mirror symmetry, beyond numerical verification, involves the wrapped Fukaya category \cite{GPS}.

We point out that there is currently no widely accepted definition of the wrapped Fukaya category away from the conical setting. One of the contributions of the present paper is a proposed definition and construction which is natural from the standpoint of SYZ mirror symmetry. This is explained in the next subsection. From a more technical standpoint, for wrapped Floer theory to be well defined in the non-exact setting we need to develop the Lipschitz geometry associated with an SYZ fibration after an appropriate completion. This is partly responsible for the length of this paper. An overview of this is given at the beginning of section \S\ref{SecWrap}.

At last, section \S\ref{SecComp}, which is the main section, introduces a new approach to wall crossing analysis which replaces the super-potential of \cite{Auroux07,AAK16,CSN} by certain elements of $SH^0(M;\Ll)$. Following \cite{GKH1}, the latter are thought of as global functions on the mirror $M^\vee$. The evaluation of these functions on points of the mirror arises from the closed-open map $SH^0(M;\Ll)\to HW^0(L,L;\Ll)$. A similar idea has been employed previously by Pascaleff\cite{Pasc2}. The essence of this idea is spelled out for the case $n=2$ in subsection \S\ref{SubSecWorkedExample} below which can be understood for the most part without reading the rest of the paper. The upshot is that while the wall crossing could be analyzed in terms Maslov $0$ disc bubbling in the open-closed map, it can also be analyzed purely in the realm of closed strings in terms the appearance of non-trivial continuation trajectories in Hamiltonian Floer cohomology upon crossing a wall (which occurs by a homotopy of $(J,\pi)$). It is for this reason that our work, while informed by the philosophy of \cite{FOOO}, is independent of its foundations.

\subsection{Algebraic vs Analytic aspects of mirror symmetry}
The principle of GAGA states that on a projective manifold, geometric and analytic geometry are equivalent. On a quasi-projective manifold this is no longer true and the algebraic structure, when such exists, is an additional structure. In fact, there are examples of relevance in mirror symmetry where there exist inequivalent algebraic structures which are analytically isomorphic.
In \cite{Tu14,Abouzaid14,Abouzaid2017} the SYZ mirror and the mirror functor are constructed  in the rigid analytic category.  By contrast, Theorem \ref{tmHMS} is algebraic. We comment on how this arises Floer theoretically.

We show in \S\ref{SecWrap} that if $(M,\Ll)$ is sufficiently well behaved, there is a well defined notion of a Lipschitz function with respect to the integral affine structure on the base of $\Ll$. This is the semi-group $\mathcal{H}_{\Ll}$ alluded to in Theorem \ref{tmc}. As motivation in connection with homological mirror symmetry, consider the following toy example. Take $M=\Tt=\R\times S^1$ with the standard coordinates $(s,t)$. Let $\Ll:M\to\R$ be defined by $(s,t)\mapsto s$. Let $\sigma$ be the section $\{t=0\}$. Since there are no critical points the choice of $(\pi,J)$ is superfluous. Then $\mathcal{H}_{\Ll}$ is the set of Hamiltonians that are proper, bounded below and satisfy that $|\partial_sH|$ is bounded. The growth condition on Lagrangians is given by a similar Lipschitz condition in the standard coordinates. In this case (and, more generally, for $\Tt^n$) our definition reproduces the standard wrapped Fukaya category of the exact case. Observe now that $SH^*(\Tt)=\C[x,x^{-1}]$. Moreover, the filtering by the bound on $|\partial_sH|$ corresponds to the filtering of Laurent polynomials by the order of pole at $\{0,\infty\}$. Our definition of the wrapped Fukaya category for the semi-toric case reproduces the algebraic structure on the mirror in a similar way. The question for more complicated SYZ fibrations awaits further exploration.

To put this into the general context note that it was  observed in \cite{Groman15} there are two different flavors of symplectic cohomology on an open manifold. On the one hand, denoting by $C(M)\subset C^{\infty}(M)$ the additive group of smooth exhaustion functions, let $\mathcal{H}\subset C(M)$ be a semi-group admitting a cofinal sequence $\{H_i\}$. One then obtains an algebra by considering
\[
SH^*(M;\mathcal{H})=\varinjlim_iHF^*(H_i).
\]
On the other hand one can associate a ring $SH^*(M|K)$ with the structure of a Banach space to a pre-compact set $K\subset M$. This is formally the Floer cohomology of the function which is $0$ on $K$ and $\infty$ on $M\setminus K$. One then obtains a ring associated with the entire space by considering
\[
SH^*_c(M)=\varprojlim_{K_i}SH^*(M|K_i),
\]
for an exhaustion of $M$ by the precompact sets $\{K_i\}$. Note that the latter construction is a purely symplectic invariant while the first construction depends on the semi-group. There are a number of examples showing that this dependence is non-trivial. We expect that the choice of semi-group is relevant to determining an algebraic structure on the mirror. It is believed that the second construction when applied to invariant sets of an SYZ fibrations gives rise to the sheaf of analytic functions. Such a sheaf is studied in \cite{Umut}, and, in connection with mirror symmetry, in a forthcoming work joint to the author and U. Varolgunes. The open string version of this latter construction is explored in forthcoming work, joint with M. Abouzaid and U. Varolgunes.

\subsection{Acknowledgements}
The author would like to thank M. Abouzaid for many useful discussions. The author is grateful to the referee whose comments and suggestions have improved this work. 

Various stages of the work were carried out at the math departments of Hebrew University of Jerusalem, the ETH of Zurich, and Columbia University. They were supported by the ERC Starting Grant 337560, ISF Grant 1747/13, Swiss National
Science Foundation (grant number 200021$\_$156000) and the Simons Foundation/SFARI ($\#$385571,M.A.). The author is grateful to the Azrieli foundation for granting him an Azrieli fellowship.
\section{Geometric Setup}\label{SecSetup}
\subsection{Semi-toric SYZ fibrations}
Let $(M,\omega)$ be a symplectic manifold of dimension $2n=6$. Let $B=\R^3$. Let $\Ll:M\to B$ be a proper surjective map whose fibers are Lagrangian. By the Arnold-Liouville Theorem, the generic fibers are Lagrangian tori. Moreover the local functions on a neighborhood of  a point $b\in B_{reg}$ whose induced flow  at $b$ is $1$-periodic give rise to a canonical integral lattice $\Lambda^*\subset T^*B_{reg}$. $\Lambda^*$ is a locally constant sheaf on $B_{reg}$ naturally isomorphic to the sheaf $R^1\Ll_*(\Ll^{-1}(B_{reg}))$ whose restriction to any simply connected $U\subset B_{reg}$ is $H_1(\Ll^{-1}(U);\Z)$. The isomorphism on the stalk at $b\in B_{reg}$ is given by $\gamma\mapsto dI_{\gamma}$ where $\gamma\in H_1(L_b;\Z)$ and $I_{\gamma}:U\subset B_{reg}\to\R$ maps $x\in U$ to the integral of $\omega$ over the cylinder traced by moving $\gamma$ from $b$ to $x$. We will use this identification freely. Dually, we have an integral lattice $\Lambda\subset TB_{reg}$, giving rise to an integral affine structure on $B_{reg}$.

For the following definition we remind  the terminology of \cite{Gross01} for singular fibers in s-dimensional torus fibrations. A singular fiber $F$ is said to be of type $(i,j)$ if the  $rk H^1(F;\Z)=i$ and $rk H^2(F;\Z)=j$.  
\begin{df}\label{dfSTSYZ}
$\Ll$ is called a \textit{semi-toric SYZ fibration with a section} if the skeleton $\Delta:=B\setminus B_{reg}$ is a trivalent graph, and, in the terminology of \cite{Gross01}, the following are satisfied
\begin{enumerate}
\item The singular fibers over the edges of $\Delta$ are of type $(2,2)$ and the singular 1-dimensional orbits are non-degenerate in the sense of \cite{MirandaZung}.
\item The singular fibers over the vertices are of type $(1,2)$.
\item The set  of critical points of $\Ll$ is a union of submanifolds whose tangent spaces span $T_pM$ at each fixed point $p$ of $\Ll$.
\item There is a global codimension $1$ sub-sheaf $\Gamma$ of $\Lambda^*$ with trivial monodromy.
\item There is a Lagrangian section $\sigma$ which is disjoint of the critical points of $\Ll$. Such a section is referred to as \textit{admissible}.
\end{enumerate}
$(M,\Ll)$ is said to be \textit{simple} if $\Delta$ is connected.
\end{df}

\textit{Henceforth $(M,\Ll)$ is a simple semi-toric SYZ fibration with an admissible section $\sigma$.}

To justify the terminology note that the assumption about the monodromy implies there is a pair of Hamiltonians $(H_{F_1},H_{F_2})$ which factor through $\Ll$ and generate  a $\T^2$ action. Namely, pick two generators $\gamma_1,\gamma_2$ of $\Gamma$ over $b_0$, and define $H_{F_i}=I_{\gamma_i}\circ\Ll$.
Henceforth, we  denote the $2$-torus by $G$. Then $\Gamma$ is naturally identified with the integral lattice in the Lie algebra $\mathfrak{g}$. The action of $G$ has a moment map $\mu:M\to\mathfrak{g}^*$ which in the choice of basis for $G$ above is just $\mu=(H_{F_1},H_{F_2})$. By construction $\mu$  factors through $\Ll$. We will denote this by $\mu=f_{\mu}\circ\Ll$. Thus $f_{\mu}(\Delta)$ is a graph in $\mathfrak{g}^*$.

\begin{ex}\label{exNormalVertex}
Consider $M=\C^3\setminus\{z_1z_1z_3=1\}$ and let
\[
H_B=-\ln\|1-z_1z_2z_3\|,\quad H_{F_i}=|z|^2_i-|z|^2_3,\quad i=1,2.
\]
The Hamiltonian $H_{F_i}$ and $H_B$ commute with each other. The fibers of $\mathcal{L}=(H_B,H_{F_1},H_{F_2})$ are Lagrangian tori on which the holomorphic volume form
\[
\Omega:=\frac{dz_1\wedge dz_2\wedge dz_3}{1-z_1z_2z_3}
\]
has constant phase. The $H_{F_i}$ induce commuting Hamiltonian circle actions. The skeleton is given by the union of rays
\[
\Delta=\{H_{F_1}=0,H_{F_2}\geq 0\}\cup \{H_{F_2}=0,H_{F_1}\geq 0\}\cup \{H_{F_1}=H_{F_2}\leq 0\}
\]
\end{ex}
In fact, locally, near the singular points lying over vertices of $\Delta$ there are $\T^2$ equivariant neighborhoods which are symplectomorphic to corresponding neighborhoods of the singularity in the last example. See \cite{Castano04}. Further examples are given by the complements of anti-canonical divisors in toric Calabi-Yau. These are described below in \S\ref{subsecTCY}.

\textit{Henceforth, we use the notation $L_b:=\Ll^{-1}(b)$ for $b\in B$.}
\begin{tm}\label{tmBasicTop}
\begin{enumerate}
\item \label{tmBasicTop3} There is a generator of $H^1(M;\Z)$ whose restriction to $H^1(L_b;\Z)$ coincides with a lift of a global section of $\Lambda/\Gamma^*$.
\item $c_1(M)=0$.
\item The Maslov class of the torus fibers vanishes.
\item Pick an almost complex structure $J$ which is compatible with $\omega$. Then $\Ll$ determines a canonical homotopy class of sections $\Omega$ of $\Lambda_{\C}^3T^*M$ by the requirement that $arg (\Omega|_{T\Ll^{-1}(b)})$ be exact for any $b\in B_{reg}$.
\end{enumerate}
\end{tm}
\begin{proof}
\begin{enumerate}
\item
Dual to the two global action coordinates there exists a global angle coordinate defined on the complement of the singular fibers. Indeed, on $B_{reg}$ we have a globally defined vector field $\frac{\partial}{\partial x_3}$ defined as being in the kernel of the globally defined $1$-forms $dx_1,dx_2$ and normalized so that $dx_3(\frac{\partial}{\partial x_3})=1$. The latter condition is well defined, since $dx_3$ is well defined up to addition of a linear combination of $dx_1$ and $dx_2$. Choosing a Lagrangian section we can lift to a vector field $V$ on $\Ll^{-1}(B_{reg})$. We thus obtain a $1$-form $\theta_3=\iota_V\omega$. This $1$-form depends of course on the choice of lift, but its restriction to any fiber of $\Ll$  does not. The one form $\theta_3$ extends at least continuously to $M\setminus \crit(\Ll)$. Indeed, $dx_3$ diverges logarithmically \cite{Castano04} near a singular value, so $\frac{\partial}{\partial x_3}$ approaches $0$ on any sequence converging to a non-singular point of $\Ll$. But, due to the singularity at $\crit(\Ll)$, $\theta_3$ does not extend continuously to all of $M$. However the integral of $\theta_3$ vanishes on any loop contained in a small enough neighborhood of $\crit(\Ll)$. Thus, using Mayer-Vietoris, the cohomology class of $\theta_3$ extends across $M$. The one form $\theta_3$ satisfies the required property.
\item
We show below that the underlying symplectic manifold of a simple semi-toric SYZ fibration with section contains a deformation retract which is symplectomorphic to an open subset of a toric Calabi-Yau.
\item We first recall the definition of the Maslov class $\mu:\pi_2(M,L)\to\mathbb{Z}$ for a Lagrangian sub-manifold in a symplectic manifold $M$. For a map $u:(D,\partial D)\to (M,L)$ it is defined by 
$$\mu(u)=\mu(u^*TM,\partial u^*TL)$$
 where the right hand side is the Maslov index of the Lagrangian loop defined by the sub-bundle $u^*TL\subset u^*TM$ along $u(\partial D)$. At first sight this requires fixing a choice of trivialization $u^*TM$ along $u(\partial D)$. But up to homotopy there is a unique one which extends to $D$. For a reference see \cite[Appendix C]{MS2}.  In particular, when $\partial u$ is contractible we have $\mu(u)=c_1(u^*TM)$ where without loss of generality $u$ is taken to be a map from the sphere.
 
Proceeding with the situation at hand, call  a class in $\pi_2(M,L_b)$ \textit{basic} if it is represented by a disk which is foliated by vanishing cycles of some edge of $\Delta$. Then $\pi_2(M,L)$ is generated over $\pi_2(M)$ by the basic discs. Having established the vanishing of $c_1(M)$ it suffices to verify the vanishing of the Maslov class on the basic disks. Given a basic disc $D_{e,b}$ in $L_b$ corresponding to an edge $e$, we may move $b$ arbitrarily close to $e$ without changing the Maslov index. By the equivariant local normal form theorem we have an equivariant identification of a neighborhood of a point $x$ on $e$ with $\R^4\times \R^2$ together with a linear $S^1$ action on $\R^4$ with weights $(1,-1)$.  This identification takes $x$ to $0$ and the boundary of the disc to a an $S^1$ orbit. The claim is now clear.
\item
Pick a section $\Omega_0$ and consder the $1$-form $\alpha=d\arg\Omega_0|_{TL_b}$. Then class of  $\alpha$ maps to the Maslov class under the coboundary map $\delta:H^1(L;\Z)\to H^2(M,L;\Z)$.  Since $\delta(\alpha)=0$ by the previous item, we have that $\alpha$ lifts to a class $\beta\in H^1(M;\Z)$. $\beta$ can be represented as the class $d\theta$ for some map $\theta:M\to \R/\Z$. The section $\Omega=e^{-i\theta}\Omega_0$ satisfies what is required. 
\end{enumerate}
\end{proof}

\subsection{Monodromy and the dual graph}
In this subsection we will show that $f_\mu(\Delta)$ is a smooth tropical curve and we will interpret its dual graph, abusively denoted $\Delta^*$, in terms of monodromy. This will be important in the computation of symplectic cohomology.

The monodromy $\rho_e$ around any edge $e$ of $\Delta$ is unipotent. See, e.g., eq. \eqref{eqLocalSemiGlobal} below. By our assumption it follows further that $\rho_e-\id$ is a primitive lattice element of $\Lambda\otimes\Gamma$. In other words, choosing an integral basis $\eta_1,\eta_2$ for $\Gamma_b\subset\Lambda_b^*$ and extending it to an  integral basis of $\Lambda_b^*$ the monodromy around an edge is given by a matrix of the form
\begin{equation}\label{eqStForm}
\left(\begin{array}{ccc}
        1 & 0 & a \\
        0 & 1 & b \\
        0 & 0 & 1
      \end{array}\right),
\end{equation}
where the vector $a\eta_1+b\eta_2$ depends on the edge $e$ and $gcd(a,b)=1$.

\begin{lm}
For an edge $e\in\Delta$ and an $x\in\Ll^{-1}(\Delta)\cap\crit(\Ll)$, let $G_x\subset G$ be the stabilizer of $x$. Let $\gamma\in\pi_1(B_{reg},b_0)$ be a loop around $e$. Then under the identification $\Gamma=\mathfrak{g}_{\Z}$, we have that $\rho(\gamma)-Id$ surjects unto $\mathfrak{g}_{x,\Z}$, the integral lattice in $\mathfrak{g}_x$. In particular, $G_x$ depends only on $e$.
\end{lm}
\begin{proof}
For $g\in\Lambda^*$ we have that $(\rho(\gamma)-Id)g\in\Gamma$ and so it vanishes when transported to $e$. Thus the group element corresponding to $\rho(\gamma)-Id$ stabilizes $x$. Moreover, the stabilizer group is $1$-dimensional by the slice theorem. By our assumptions, the image of $(\rho(\gamma)-Id)$ contains a primitive element. The claim follows.
\end{proof}
\begin{lm}
The tangent space to any point $p$ in the edge $f_{\mu}(e)$ is the annihilator $\mathfrak{g}^o$ of the stabilizer $\mathfrak{g}_e$. In particular, $f_{\mu}(\Delta)$ has rational slopes.

\end{lm}
\begin{proof}
This is standard.
\end{proof}

The monodromy representation gives rise to a dual graph $\Delta^*\subset\Gamma\subset\mathfrak{g}$ as follows. First define $\Delta^*$ as an abstract graph with a vertex corresponding to each face of $f_\mu(\Delta)$ and an edge connecting vertices corresponding to adjacent faces. We now show that we can canonically embed $\Delta^*$ up to a translation into $\mathfrak{g}$ so that the vertices land in $\mathfrak{g}_{\Z}=\Gamma$. More precisely, under this embedding, the edge connecting faces $\alpha,\beta$ will be a primitive generator of the annihilator of the edge in $\Delta$ separating the two faces. For this associate an element of $\Gamma$ with each directed edge in $\Delta^*$ as follows. Let $W\subset B$ be a hypersurface which contains $\Delta$ and maps diffeomorphically to $\R^2$ under $f_\mu$. Pick a base-point $b_0\in B_{reg}$ and an element $g$ of $\Lambda^*_{b_0}$ whose projection to $\Lambda^*_{b_0}/\Gamma$ is a primitive generator. For $\alpha\in\Delta^*$ denote by $F_{\alpha}$ the corresponding face. To the directed edge $\alpha\beta$ associate the loop $\gamma_{\alpha\beta}$ based at $b_0$ going through $F_{\alpha}$ and back through $F_{\beta}$. Let
\[
v_{\alpha\beta}:=(\rho(\gamma_{\alpha\beta})-\id)g\in\Gamma.
\]
We have
\begin{equation}\label{eqSymmVert}
v_{\alpha\beta}=-v_{\beta\alpha}.
\end{equation}
Also, if $\alpha,\beta,\gamma$ are vertices corresponding two a triple of faces meeting at a vertex in $\Delta$, we have the relation
\begin{equation}\label{eqCocycle}
v_{\alpha\beta}+v_{\beta\gamma}+v_{\gamma\alpha}=0.
\end{equation}
But trivalence of $\Delta$ implies that $\Delta^*$ is a triangulation. Thus for any loop in $\Delta^*$, the sum of the vectors $v_{\alpha\beta}$ along the loop vanishes. We can now embed $\Delta^*$ into $\mathfrak{g}$ as follows. Pick a spanning tree $T$ in $\Delta^*$ and pick a vertex $\alpha_0$  to be the root. Map $\alpha_0$ to $0$, and map each vertex $\beta_0$ the sum of the vectors $v_{\alpha\beta}$ along the directed path in $T$ connecting $\alpha_0$ to $\beta_0$. The relations above show this is independent of the choice of $T$.

The choices made in constructing $\Delta^*$ have the following effect. Changing the $\alpha_0$ amounts to a translation of $\Delta^*$ and thus will be of no concern to us. Changing either the generator of $\Lambda^*_{b_0}/\Gamma$ or the connected component of $B\setminus W$ containing $b_0$ has the effect of multiplying $\Delta^*$ by $-1$. The following Lemma singles out a choice and summarizes the above discussion.
\begin{figure}
\includegraphics[scale=0.8]{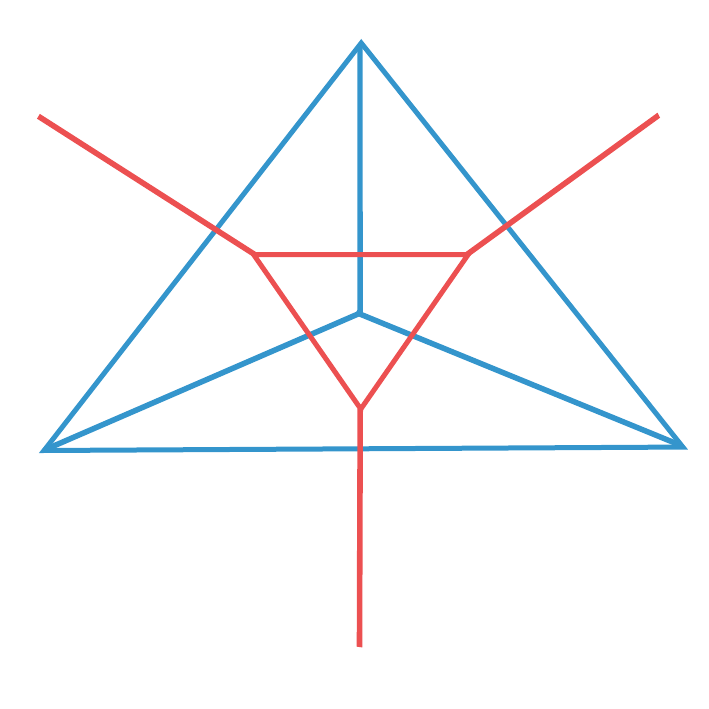}\label{FigDual}
\caption{The diagram $\Delta$ (red), and its dual $\Delta^*$ (blue). }
\centering
\end{figure}
\begin{lm}\label{lmDualGraph}
For any $\alpha,\beta\in\Delta^*$ the monodromy of $g_0$ for the loop based at $b_0$ around the dual edge which passes through the face $\beta$ and then through $\alpha$ is, given by $g_0\mapsto g_0+(\beta-\alpha)$. Moreover, $g_0$ can be chosen to satisfy the following. For any $\alpha\in\Delta^*$ corresponding to the face $F$, the vectors to $\alpha$ from its adjacent vertices generate the dual cone to $F$ over $\Z$. The latter can be characterized as the set of elements $\gamma\in\Gamma$ such that either $\gamma=0$ or for any sequence $w_n\in F$ we have $\langle \gamma,w_n\rangle \to\infty$  if the distance of $w_n$ to the closed set  $\partial F$ goes to infinity. Here $\langle\cdot,\cdot\rangle$ denotes the natural  pairing $\mathfrak{g}\times \mathfrak{g}^*\to\R$ and the distance is the one induced by some norm on $\mathfrak{g}^*\simeq\R^2$ with respect to some inner product. 
\end{lm}

\begin{rem}
The significance of the characterization of the dual cone to a face given in the Lemma will be made clear during  the discussion at the end of \S\ref{subsecLaurent} leading to the proof of Theoren~\ref{tmLaurent0}]. See in particular Lemma \ref{lmPositivCone}.
\end{rem}
\begin{proof}
Pick a basis of $\mathfrak{g}$ to identify it with $\mathfrak{g}^*$. Then the edges of $\Delta^*$ are orthogonal to the edges of $\Delta$. The claim is that after possibly multiplying by $-1$, the directed edges emanating from any $v\in\Delta^*$ corresponding to a face $F$, point into $F$. We have already observed that $\Delta^*$ is a triangulation of a polygon. For each triangle the vectors are either all point inwards or all outwards by \eqref{eqCocycle}. If they all point inwards for one vertex, they do for all vertices by \eqref{eqSymmVert}.

To see the characterization of the dual cone, note that the distance from $\partial F$ to a sequence $w_n$ of points of $F$ goes to infinity precisely when the distance to each component of $\partial F$ goes to infinity. It is thus clear that the covectors whose value on all such sequences goes to infinity are precisely the non-negative non-trivial combinations of the inward pointing normals to the components of $\partial F$ 
\end{proof}

For each vertex $v\in\Delta$ let $\Delta^*_v\subset V$ be the triangle in formed by the vertices corresponding to the face containing $V$.
\begin{lm}
$\Delta^*_v$ has area $1/2$.
\end{lm}
\begin{proof}
This follows since the sides of $\Delta^*_v$ are all primitive. Thus, they can be mapped to the standard $2$-simplex by an element of $SL(2;\Z)$.
\end{proof}


\begin{cy}\label{corMuSmooth}
$f_{\mu}(\Delta)$ is a smooth tropical curve.
\end{cy}



We have shown that any simple semi-toric SYZ fibration gives rise to a smooth tropical curve. The converse construction is described in the next subsection.

\subsection{Toric Calabi-Yau manifolds and the Gross fibration}\label{subsecTCY}
\begin{df}
A \textit{toric Calabi-Yau manifold} is a symplectic manifold constructed as follows. Let $k\in\mathbb{N}$ and let $G'\simeq\T^n$ be a closed subtorus of $\T^{n+k-1}$ embedded in $SU(n+k)$ acting on $\C^{n+k}$ as the group of diagonal matrices. Denote by $\mu_{G'}:\C^{n+k}\to \mathfrak{g}^{'*}$ the moment map and let $c\in\mathfrak{g}^{'*}$ be a regular value of $\mu_{G'}$. Let  $X:=\mu_{G'}^{-1}(c)/G'$. X carries an action $G''=\T^{n+k}/G'$.
\end{df}
The assumption $G'\hookrightarrow SU(n)$ implies that the canonical bundle of $X$ is trivial. Indeed, there is an induced nowhere vanishing complex volume form on $X$ defined as follows. Let $\Omega_0=dz_1\wedge\dots \wedge dz^{n+k}$ be the standard holomorphic volume form on $\C$. Pick a basis $g_1,...,g_n$ of $\mathfrak{g}'$. Let
\[
\overline{\Omega}:=\iota_{g_{1},...,g_{n}}\Omega_0.
\]
Then $\overline{\Omega}$ descends to a holomorphic volume form $\Omega$ on the quotient $X$.

The embedding $G'\hookrightarrow SU(n)$ implies that $X$ is non-compact. Indeed, let $v=\sum e_i\in\R^{n+k}$. Then identifying the Lie algebra of the standard torus with $R^{n+k}$ we have $\langle u,v\rangle=0$ for any component $u$ of $\mu_{G'}$. Thus, identifying $\R^{n+k}$ with it dual, $\mu_{\T^{n+k}}(\mu_{G'}^{-1}(c))$ contains an affine ray parallel to the one generated by $v$.

Let $G\subset G''=\T^{n+k}/G'$ be the $n+k-1$-dimensional torus which preserves the complex volume $\Omega$. The \textit{toric diagram} of $X$ is the image under $\mu_G$ of the $1$-dimensional orbits of $G=\T^{n+k}/G'$.

$X$ is also a complex variety since it can be realized as a GIT quotient. Denote its complex structure by $J$. On $X$ there is a globally defined function
\begin{equation}\label{eqStLG}
f:[z_1,\ldots,z_{n+k}]\mapsto z_1\cdot\ldots \cdot z_{n+k},
\end{equation}
which is well defined since $G'$ preserves complex volume.  Picking any regular value $c$ of the function $f$, we obtain an SYZ fibration by considering $M:=X\setminus f^{-}(c)$ and $H_B=\ln\|c-f\|$ \cite{Gross}. We refer to this as the Gross fibration.
\begin{tm}\label{tmStGross}
Let $X$ be toric Calabi-Yau threefold. Then $M=X\setminus D$ with the gross fibration $\Ll:=(\mu_G,H_B)$ is a semi-toric SYZ fibration with a section.
\end{tm}
\begin{proof}
We start by studying the singularities of $\Ll$. These correspond to the fixed points of the torus action $\mu$ which occur when $\mu$ is on the toric diagram (the image of the $1$-skeleton of the moment polyhedron under the projection which forgets the scaling action) and $H_B=\ln c$. This gives $\Delta$ which is a trivalent graph, projecting to the toric diagram when forgetting $H_B$. Topologically, the singular fiber over a generic point of the graph is a of type $(2,2)$. Namely, at an edge $e$ it is the product of an $S^1$, generated by an $S^1$ action which is transverse to $e$, with a nodal torus generated by $H_B$ and the circle action which is fixed on $e$. We need to show that it is non-degenerate. This means that, denoting by $V\subset T_xM$ the fixed subspace at sum singular point $x$, the sub-algebra of $sp(V)$ generated by $d^2H_B,d^2H_F$ is Cartan. This is equivalent \cite{BoFo} to them being linearly independent and that there exists a linear combination $\lambda d^2H_B+\mu d^2H_F$ with four distinct eigenvalues. This property is stable under symplectic quotient so it suffices to verify this condition for the pair of function $H_1(z_1,\dots,z_n)=|z_1|^2-|z_2|^2$ and $H_2(z_1,\dots,z_n)=\re(z_1z_2\dots z_n)$. The property is invariant under the action of the torus, so we can verify this at a point for which $\im(z_2\dots z_n)=0$. The claim is now verified by a straightforward computation.
\end{proof}
\subsection{The local normal form}
We conclude our discussion of basic geometric setup with a local description near the critical points $\crit(\Ll)$. We will use this to show in Theorem \ref{tmLocForm} that while (semi-)globally there is an infinite dimensional space of symplectically inequivalent germs of toric SYZ fibrations with the same skeleton $f_\mu(\Delta)$, an equivariant (with respect to $(H_{F_1},H_{F_2})$) neighborhood of $\crit(\Ll)$ is determined up to equivariant symplectomorphism by the skeleton. This will be important in the construction of LG potentials.
\begin{tm}[\textbf{Equivariant local normal form}]\label{tmLocForm}
Suppose $(M_i,\Ll_i)$ are simple semi-toric SYZ fibrations with discriminant loci $\Delta_i$, for $i=1,2$ and suppose that $f_{\mu_1}(\Delta_1)$ can be mapped to $f_{\mu_2}(\Delta_2)$ by an element $A$ of $SL(2,\Z)$. Then there are neighborhoods $V_i$ of $f_{\mu_i}(\Delta_i)$ in $\mathfrak{g}$ and $U_i$ of $\Ll_i^{-1}(\Delta_i)\cap \crit(\Ll_i)$ in $M_i$ projecting onto $V_i$ and a commutative diagram
\[
\xymatrix{
U_1\ar[d]\ar[r]& U_2\ar[d] \\
V_1 \ar[r] &V_2,
}
\]
where the upper horizontal map is an equivariant symplectomorphism and the lower one is a the linear map $A$.
\end{tm}
\begin{proof}
First observe that there exists a local equivariant symplectomorphism as in the theorem near each fixed point $x$. Indeed, germs of such neighborhoods are classified up to equivariant isomorphism by the weights of the torus action $G_x$ on $T_xM$. But these weights are determined by the requirement on the critical points in the definition of a semi-toric SYZ fibration. We extend this to a diffeomorphism in a neighborhood of the one dimensional orbits as follows. First, we define the map on the union of one-dimensional orbits equivariantly and in a way that preserves the moment map.   The condition on $f_\mu(\Delta)$ guarantees that this is possible. This guarantees that the $\omega_i$ one restricted to the tangent directions are intertwined. We now extend to a diffeomorphism $\phi$ that further intertwines the symplectic forms along these $1$-dimensional orbits in the normal directions. This can done locally by the slice theorem and any two such local diffeomorphism can be patched together equivariantly without changing anything along the $1$-dimensional orbits. The claim now follows by the equivariant Darboux theorem.
\end{proof}
The above local normal form involves only the $2$-torus action. We will also be interested in another normal form which also involves the local behavior of the third Hamiltonian $H_B$ near a $1$-dimensional orbit. We refer the reader to \cite{Castano04} for the proof.
\begin{tm}\label{tmLocNorFo}\cite{Castano04,MirandaZung}
Let $\mathcal{O}$ be a $1$-dimensional orbit of $\Ll$. Let $V=D^4\times (0,1)\times S^1$ with canonical coordinates $(x_1,y_1,x_2,y_2,r,\theta)$. Let
\begin{equation}\label{eqLocEdgeNormForm}
 q_1=x_1y_1+x_2y_2,\quad q_2=x_1y_2-x_2y_1\quad,q_3=r.
\end{equation}
 Then there is a neighborhood $U$ of $\mathcal{O}$, and a symplectic embedding
 \[
 \psi_{\mathcal{O}}:U\to V,
 \]
 which sends fibers of $\Ll$ to fibers of the map $\Ll_0=(q_1,q_2,q_3)$.
\end{tm}
We call $U$ a normal neighborhood of $\mathcal{O}$ and the maps $(\psi,\Ll_0)$ a normal form near $\mathcal{O}$. 
Examining the proof in \cite{MirandaZung} we strengthen this as follows.
\begin{lm}\label{lmEdgeLocalForm}
For each edge $e$ of $\Delta$ there a neighborhood $U_e$ of $\crit(\Ll)|_e,$  a symplectic ball $D^4\subset \R^4,$ and a symplectomorphism $\psi_e$ such that $\psi_e$ symplectically embeds $U_e$ in $D^4\times f_\mu(e)\times S^1$ intertwining $\Ll|_{U_e}$ with $\Ll_0$. Moreover, $\psi_e$ can be made to agree on any closed interval $e'\subset e$ with any normalizing $\psi$ already defined on a neighborhood of the endpoints of $e'$.
\end{lm}
\begin{proof}
See the proof of Lemma 4.2 in \cite{MirandaZung}.
\end{proof}
\begin{rem}
Observe that the local normal form is not unique. However \cite{Ngoc03}, its germ is determined up to a sign change, a map on the base which is tangent to the identity at the origin, i.e, it differs from the identity by map with derivatives at the origin vanishing to all orders, and a Hamiltonian isotopy  flowing in the fibers.
\end{rem}

\begin{tm}\label{tmNormalTCY}
For any simple semi-toric SYZ fibration $M$ there is a unique toric Calabi-Yau threefold $X_M$ up to equivariant symplectomorphism such that $f_\mu(\Delta_M)$ is the toric diagram of $X_M$. Moreover, there exists a map $\psi$ embedding a neighborhood of $Crit(\Ll_M)$ into a neighborhood of $Crit(\Ll_{X_M})$ locally intertwining $\Ll_M$ and $\Ll_{X_M}$.
\end{tm}
\begin{proof}
The first half is a standard construction in toric geometry. 
The second half is an immediate consequence of Theorem \ref{tmLocForm}.
\end{proof}

\section{Tame SYZ fibrations}\label{SecWrap}
\subsection{Overview}\label{SubsecTschRem}
In \cite{Auroux07,AAK16,CSN} Floer theory for $M$ is developed using bordered curves which are strictly holomorphic  with respect to the $J$ induced from the GIT construction of $X$. This holomorphicity implies a maximum principle which plays  a key role both in establishing the necessary $C^0$ estimates and in carrying out the wall crossing analysis.  In our approach, however, we need to consider the inhomogeneous Floer equation for which we do not have an obvious maximum principle. Beyond this technical problem, to construct Floer theory in a non-ad-hoc manner in the non-exact setting, we apply the approach of \cite{Groman15} to Floer theory on an open manifold.

The latter approach requires at its minimum that the underlying manifold $M=X\setminus D$ be geometrically bounded, and, in particular complete. For this, it is necessary that we inflate the symplectic form near $D$. It turns out, according to Theorem \ref{tmUniqueComp}, that there is a unique natural way, up to equivariant symplectomorphism, to do this.

Once that is done, the approach of \cite{Groman15} requires further that we specify a growth condition at infinity with respect to which wrapping is done.
It turns out that to get a growth condition which could induce an ``order of pole'' filtration as described in the introduction, we need that the action coordinates together with a Lagrangian section induce a uniform structure at infinity. The growth condition is that our Hamiltonians are Lipschitz in the uniform structure. This gives rise to finite dimensional increasing filtrations in wrapped Floer cohomology.

The existence of such a uniform structure translates into a simple requirement on the behavior of the semi-global invariant at infinity. An $\Ll$ satisfying this is called tame. 
Given this Lipschitz structure, there exist compatible almost complex structures on $M$ which are geometrically bounded. In appendix \ref{AppLipFloer} we spell out the details of how to obtain $C^0$ estimates for the wrapped Fukaya category with Lipschitz Floer data. Thus, a-priori, the input for the wrapped Fukaya category in general is a geometrically bounded symplectic manifold with a tame Lagrangian torus fibration. 

So far we have talked about the construction of the wrapped Fukaya category. To go beyond this  and actually compute the wrapped Fukaya category, or merely to control disk bubbling, we need Landau-Ginzburg potentials $(\pi,J)$ with $J$ that is compatible with uniform structure induced by $\Ll$ and $\sigma$. This is significantly harder. It is not known to the author whether there is a way to inflate the symplectic structure near $D$ so that the GIT complex structure remains geometrically bounded. For example, if one inflates by pulling back a form supported near $\pi(D)$, one easily constructs a sequence of loops in a fixed homology class with length going to $0$. Thus, we are forced to construct a Landau-Ginzburg potential ``by hand''. This point is responsible for the technical difficulty of Appendix \ref{AppA}.

\subsection{Uniform isoperimetry}
\begin{df}\label{dfGoBound}
An $\omega$-compatible almost complex structure $J$ is said to be \textit{uniformly isoperimetric} if the associated metric $g_J$ is complete and there are constants $\delta$ and $c$ such that $inj_{g_J}(M)<\delta$ and any loop $\gamma$ satisfying $\ell(\gamma)<\delta$ can be filled by a disc $D$ satisfying $Area(D)<c\ell(\gamma)^2.$ If $J$ is merely $\omega$-tame, we add the requirement that, for some $c>1$,
\[
\frac1c<\left|\frac{\omega(Jx,Jy)}{\omega(x,y)}\right|<c.
\]
We refer to almost complex structures satisfying the latter condition as uniformly $\omega$-tame.
\end{df}
\begin{rem}
The criterion is satisfied for $J$ compatible if $g_J$ is
\textit{geometrically bounded}, namely, it has sectional curvature bounded from above and radius of injectivity bounded from below.
\end{rem}
\begin{rem}\label{remObs}
It is clear from the definition that if $g_{J_0}$ is uniformly isoperimetric and $g_{J_1}$ is metrically equivalent to $g_{J_0}$ then $g_{J_1}$ is also uniformly isoperimetric. Metric equivalence here means there are constants $c ,C$ such that for any non-zero tangent vector $v\in TM$ we have
$$ c<\frac{g_{J_1}(v,v)}{g_{J_2}(v,v)}<C$$.
 \end{rem}

\begin{lm}\label{lmEquivCont}
Let $J_0,J_1$ be $\omega$-tame almost complex structures and write $Z=J_0J_1^{-1}$. Call $J_0$ and $J_1$ equivalent if $Z$ and $Z^{-1}$ both have bounded norm with respect to $g_{J_0}$ (and, therefore, $g_{J_1}$). 
Then any equivalence class 
is contractible.
\end{lm}
\begin{rem}
Observe that equivalence in the sense of the last theorem entails equivalence of the associated metrics. Moreover, if both almost complex structures are compatible, it is the same as equivalence of the associated metrics.
\end{rem}
\begin{proof}
Applying the Cayley transform $\psi:Z\mapsto (Id-Z)(Id+Z)^{-1}$ then $\psi$ maps the set of positive definite matrices to the ball $\|W\|<1$. Now, the space of tame almost complex structures $J'$ maps diffeomorphically to a convex subset of the unit ball under the composition of $\psi$ with the map sending $J'$ to $Z=J_0J^{'-1}$. See \cite[Prop. 2.5.13, Proof 3]{MS}. Moreover $J_0$ is equivalent to $J_1$ if and only if there is an $\epsilon>0$ for which $\|\psi(Z)\|<1-\epsilon$. The claim follows.

\end{proof}
As we will see below, geometric boundedness in the sense of Definition \ref{dfGoBound} is the key to $C^0$-estimates in Floer theory. Combining the observation in remark \ref{remObs} with Lemma \ref{lmEquivCont} we see that in order to define Floer theoretic invariants associated to $\Ll$, we need to be able to specify a canonical metric up to equivalence. In the rest of the section we will show that if $(M,\Ll)$ satisfy a certain tameness condition then any Lagrangian section gives rise to an equivalence class of metrics. To formulate this precisely we need to first study the invariants of an SYZ fibration.
\subsection{The semi-global invariant}

\begin{df}
Semi-toric SYZ fibrations with section $(M_i,\Ll_i:M_i\to B_i,\sigma_i)$,  $i=0,1$, are said to \textit{have the same germ} if there exist neighborhoods $V_i$ of $\Delta_i$, a diffeomorphism $\phi:V_0\to V_1$ and a symplectomorphism $\psi:\Ll^{-1}_0(V_0)\to\Ll^{-1}_1(V_1)$, commuting with the $\Ll_i$ and the $\sigma_i$. This is an equivalence relation. We refer to the equivalence class of $(M,\Ll,\sigma)$ as its \textit{germ}.
\end{df}
\begin{rem}
Observe that the information about the section is redundant. Namely, any two admissible sections in a small enough neighborhood of $\Delta$ are easily seen to be Hamiltonian isotopic.
\end{rem}

\begin{tm}\label{tmGermCalsReg}
The fibrations $(\Ll_i,B_i,\sigma_i)$ are symplectomorphic if and only if they have the same germ and there is an integral affine isomorphism $B_{0,reg}\to B_{1,reg}$.
\end{tm}
\begin{proof}
We use the section to define a symplectomorphism over $B_{reg}$. Namely, the section together with any simply connected local affine chart on $V\subset B_{reg}$ define an identification of $\Ll^{-1}_i(V)$ with $V\times \T^n$. This induces a symplectomorphism over $B_{reg}$. Note this symplectomorphism is the unique one intertwining the sections.  Since the $\Ll_i$ have the same germ, there is also symplectomorphism over a neighborhood of $\Delta$. Since the overlap between $B_{reg}$ and a neighborhood of $\Delta$ is a subset of $B_{reg}$ and all the symplectomorphisms involved intertwine sections, these symplectomorphims agree on the overlap.
\end{proof}

To make the germ a little more concrete,
we describe the behavior of $\Ll$ and its period lattice near the critical points and singular values of $\Ll$. This is a summary of results of \cite{Castano04,Castano09} (which build upon results on the focus-focus singularity also appearing in \cite{Ngoc03}). Let $e$ be an edge of $\Delta$. In a small neighborhood $V_e\subset B$ of $e$ there is a (multi-valued) basis $\{\eta_1,\eta_2,\eta_3\}$ of the period lattice $\Lambda^*_\Z$ restricted to $V_e\setminus e$ in which the monodromy is given as in \eqref{eqStForm}. Using normal coordinates we identify $V_e$ with the product $D\times e$ where $D$ is a $2$-disc and there is a neighborhood $U_e$ of $\crit(\Ll)|_e$ projecting to $V_e$ for which the there is a local normal form as in Lemma \ref{tmLocNorFo}. The functions $q_1,q_2,q_3$ describing $\Ll$ in the local normal coordinates extend uniquely to $\Ll^{-1}(V_e).$ (In particular, they give rise to a Hamiltonian flow on $U_e$). We consider their values as defining local coordinates $b_1,b_2,b_3$ on $V_e$. Semi-global invariants for $\Ll$ near $e$ are obtained by expressing the period lattice in these coordinates. Namely,
\begin{tm}
There is a smooth function $h:D\times e\to\R$ such that writing
\begin{equation}\label{eqEta0}
        \eta_0=-\ln|b_1+ib_2|db_1+\arg(b_1+ib_2)db_2,
\end{equation}
the period lattice is given by
\begin{equation}\label{eqLocalSemiGlobal}
\eta_1=\eta_0+dh,\quad \eta_2=2\pi db_2,\quad \eta_3=dr.
\end{equation}
\end{tm}

Conversely,
\begin{tm}\label{tmExFibSem}
For every smooth function $h$ on $V_e$ there is a Lagrangian torus fibration over $V_e$ for which the periods are given by equation \eqref{eqLocalSemiGlobal}.
\end{tm}

The normal coordinates are only unique up to certain changes of coordinates. These can be factored as a discrete choice of a symplectomorphism which is non-isotopic to the identity (i.e, the sign inversions leaving \eqref{eqLocEdgeNormForm} invariant) following one that differs from the identity by a map that vanishes with all its derivatives along the critical points. The discrete choice affects the definition of $\eta_1$. Accordingly, it turns out that the germ of $h$ along $e$, modulo functions that vanish to all orders along $e$, and considered as a function on $V_e=D\times e$ is  independent of any choices up to a global constant. Thus,
\begin{tm}
The germs of a pair of Lagrangian torus fibrations  $\Ll_1,\Ll_2$ over $V_e$ are symplectomorphic if and only if writing
\[
\eta_1=\eta_0+dh_i,\quad i=1,2,
\]
we have that the derivatives to all orders of $h_1-h_2$ vanish along $e$. 
\end{tm}
The data of of the germ of $h$ along each edge $e$ modulo flat functions is the semi-global invariant referred to in the title of this subsection.

We refer the reader to \cite{Castano04,Castano09} for the proofs of the above theorems.
Functions whose derivatives vanish to all orders at a point $p\in\R^n$ are called \emph{flat}. 
\begin{df}
Let $(M,\Ll)$ be a semitoric SYZ fibrations. We denote by $\Ll_{\Delta}$ the data of the smooth tropical curve $f_\mu(\Delta)$ together with for each edge $e$ the germ modulo flat functions $H_{e,\Ll}:V_e\to\R$.
\end{df}
In the following denote by $\Delta_\R$ the set $f_{\mu}^{-1}( f_\mu(\Delta))$. $\Delta_\R$ is diffeomorphic to the product of $\Delta$ with $\R$.
\begin{df}
We say that the affine structure on $B_{reg}$ is \textit{complete} if every sequence of points of $B_{reg}$ contained in an affine chart with bounded coordinates has a convergent subsequence in $B$. We say that it is \textit{nice} if there is a choice of global action coordinates on $f_{\mu}^{-1}( f_\mu(\Delta))$ defining an embedding into $\R^3$.
\end{df}
\begin{rem}
We comment on the niceness condition. Note that the components of $f_{\mu}$ are integral affine. The assumption of niceness is that there is a choice of a transversal integral affine coordinate $b_n$ which is injective on lines satisfying $f_{\mu}=const$. An example where an affine coordinate might not be injective is given by considering the universal cover of the complement of a point in $\R^2$ with its standard affine coordinates. One can use this to construct a counterexample to the second half of Theorem \ref{tmBijecCompDel} when the niceness is dropped. It appears that in the presence of the completeness condition, the niceness condition is redundant. We do not try to prove this.

The main consequence of these assumptions is as follows. Let $\Ll$ be complete and nice and let $B_0=B\setminus\Delta_\R$ then $B_0$ is a union of integral affine polyhedra each of which is a product of $\R$ with a (not necessarily closed) polygon. Moreover, these are embedded densely in $\R^3$ with its affine structure.

\end{rem}
\begin{tm}\label{tmBijecCompDel}
There is a bijection between germs of semi-toric SYZ fibrations and complete semi-toric SYZ fibrations which are nice. Moreover given a nice semi-toric SYZ fibration over on open set in $\R^3$ it can be uniquely completed so that it remains nice.
\end{tm}
\begin{proof}
Let $\Ll$ be a germ of a fibration, let $\tau:f_\mu(\Delta)\to\R$ be the $0th$ order term of $\Ll_{\Delta}$. Then $\Delta\subset B=\R^3$ can be thought of as the graph of $\tau$. For each edge of $f_\mu(\Delta)$ let
\[
P_e:=\{(x,t)\subset e\times\R|t\leq\tau(x)\}\subset\R^3.
\]
Let $U=\R^3\setminus\cup_{e}P_e$ where the union runs over the closed edges of $\Delta$. We glue $U$ together across $P_e$ according to the monodromy matrix associated in the dual graph $\Delta^*$ to the dual edge. This defines $B_{reg}$ with its integral affine structure. There is a unique Lagrangian torus fibration $\Ll_{reg}$ with this integral affine structure \cite{Castano04}. Consider a small enough neighborhood $V$ of $\Ll_{\Delta}$ so that the action coordinates on any open and simply connected set of $\Ll_{\Delta}(V)$ define an embedding. Picking a Lagrangian section for $\Ll_{reg}$ we can now glue it with the admissible section on the germ in a unique way to give the required SYZ fibration.

For a germ $\delta$ denote the corresponding fibration constructed this way by $\Ll_\delta$. By Lemma \ref{tmGermCalsReg} and by construction this is an injection.  We now show this is a surjection. Namely, let $\Ll:M\to B$ be a complete nice semi-toric SYZ fibration with section and let $\delta$ be its germ. We need to show that $\Ll$ is symplectomorphic to $\Ll'=\Ll_{\delta}$. By assumption, there is an affine diffeomorphism between $B_0$ and $B'_0$. Moreover, the monodromy, which is determined by $f_\mu(\Delta)$, uniquely determines how the integral affine structure extends across $\Delta_\R$. The claim now follows by Lemma \ref{tmGermCalsReg}. 
\end{proof}

\begin{tm}\label{tmUniqueComp}
Let $\Ll_i=M_i\to B$ for $i=1,2$ be a pair of complete semi-toric SYZ fibrations such that $f_{\mu_1}(\Delta_1)$ is related to $f_{\mu_2}(\Delta_2)$ by an element of $SL(2,\R)$. The they are equivariantly symplectomorphic.
\end{tm}
\begin{rem}
Note that equivariant symplectomorphism is weaker than a symplectomorphism which intertwines the torus fibrations. Indeed the former involves only the components of $\Ll$ which generate periodic actions. A similar observation to Theorem \ref{tmUniqueComp} in the $2$ dimensional case is made in \cite[Corollary 5.4]{Symington2002}, though the equivariance isn't stated. 
\end{rem}
\begin{proof}
First observe that by the previous Theorem, the equivariant diffeomorphism type of a complete semi-toric SYZ fibration is determined by the combinatorial structure of $f_\mu(\Delta)$. If the condition of the current Theorem is satisfied we can find an equivariant diffeomorphism which in addition intertwines the moment maps. Moreover, for the preimage under $\mu$ of a compact set $K\subset\R^2$, we may assume that the $\Ll_i$ are intertwined outside of a compact set. The assumption on $f_{\mu_i}(\Delta_i)$ implies that $\psi^*\omega_2-\omega_1$ is exact. Moreover, denote by $V\subset M$ the ends where  $\psi^*\omega_2-\omega_1$ is assumed identically $0$. Then we may further assume that $\omega_1$ and $\psi^*\omega_2$ are cohomologous relative to $V$. Indeed, this is nothing but the intertwining of the action coordinates near the ends. We can therefore pick a primitive $\sigma$ such that for each $K\subset \R^2$ we have that for which $\sigma=0$ on the complement of a compact set in $\mu^{-1}(K)$. Moreover, by averaging, $\sigma$ can be taken to be $T^2$ invariant. Let $\omega_t =\omega_1+td\sigma$ and let $X_t$ be the $\omega_t$ dual of $\sigma$. Then $X_t$ is invariant, preserves the moment map $\mu$, and vanishes where $\sigma$ does. In particular, the flow exists for all times.   Moser's trick now implies the claim.
\end{proof}
\subsection{Tame SYZ fibrations}
Before stating the following definition we recap the following point from the previous subsection. Given an edge $e$ of $\Delta$, we can make a choice of equivariant normal coordinates as in Theorem \ref{tmLocForm} identifying $e$ with a subset of $\mathfrak{g}$ and a neighborhood $V$ of $e$ with the product of $e$ and the standard disk in $\R^2$. In these coordinates we obtain a smooth function $H_{e,\Ll}$ featuring in eq \eqref{eqLocalSemiGlobal}. Moreover, the germ of $H_{e,\Ll}$ modulo flat functions is independent of the choice of equivariant normal coordinates. Note that $e$ is identified with a  a subset of $\mathfrak{g}$ up to integral affine transformations. Thus we can talk about the $C^k$ norms of $H_{e,\Ll}$ along a half infinite edge $e$ being bounded or not without introducing any choices. 

\begin{df}\label{dfAdmSYZ}
We say that $\Ll$ is \textit{tame} if it is complete, nice, and, for any $k$, $dH_{e,\Ll}$ is uniformly bounded in $C^k$ along $e$ for any half infinite edge $e$
\end{df}
\begin{tm}[\textbf{The uniform structure on $B$}]\label{tmBaseCoord}
Let $\Ll$ be tame. There exists a global coordinate system $(x_1,x_2, x_3):B\to\R^3$ with the following properties. For $i=1,2$ denote by $I_i$  action coordinates associated with some basis for the periods in $\Gamma\subset\Lambda^*$. Fix a simply connected open and dense set $V\subset B_{reg}$ and choose a complementary action coordinate $I_3$ on $V$.
\begin{enumerate}
\item For any $\delta>0$ there is a bi-Lipschitz equivalence between the coordinates $x_1,x_2, x_3$ and $I_1,I_2, I_3$ on $V\setminus B_{\delta}(\Delta)$ with Lipschitz constants depending only on $\delta$ and the chosen basis.
\item the functions $I_1,I_{2}$ are Lipschitz on $B$ in the coordinates $\{x_i\}$.
\item For small enough $\delta>0$, for any edge $e$ and any $x\in B_{\delta}(e)\setminus\Delta$,  we have
\[
\left|\ln \frac1cd(x,\Delta)\right|<|\nabla I_{n}(x)|<|\ln cd(x,\Delta)|.
\]

\end{enumerate}
Any such coordinate system is determined uniquely up to bi-Lipschitz equivalence.
\end{tm}
\begin{rem}
We may take $x_i=I_i$ for $i=1,2$. We may also assume $\Delta$ is contained in the plane $\{x_3=0\}$.
\end{rem}
\begin{proof}
For each half infinite edge $e$ of $\Delta$ let $V_e=D^2\times e$. Let $h_e: V_e\to\R$ whose germ at $e\times\{(0,0)\}$ is $H_{e,\Ll}$. Then $|dh_e(t,0,0)|<C$ for $t\in e$. Thus there is an open neighborhood of $e\times\{(0,0)\}$ where the same inequality holds. By appropriately deforming $h$ outside of an even smaller neighborhood we may assume without loss of generality that the same inequality holds on all of $V_e$. Moreover, we may assume that $h$ satisfies uniform estimates on all derivatives on all of $V_e$. By Theorem \ref{tmExFibSem} there exists an SYZ fibration for which the sets $V_e$ describe the local normal coordinates. By taking the disk $D^2$ to be small enough  we may assume that the corresponding action coordinates are embeddings on simply connected sets. Thus, after completing as in Theorem \ref{tmBijecCompDel}, this is the unique up to symplectomorphism complete nice SYZ fibration having with the same germ as the one we started with.

To define the coordinate $x_3$ observe first that the difference between two successive branches of $dI_3$ is a bounded function of $(I_1,I_2,I_3)$. We may thus define a new function $J_3$ by interpolating two branches on a neighborhood of some cuts in such a way that the system $\{I_1,I_2,I_3\}$ is bi-Lipschitz equivalent to the system $\{I_1,I_2,J_3\}$ on any simply connected subset of $B_{reg}$. Fix a constant $C$ such that $|dH_{\Delta,\Ll}|\ll C$. Then by taking $D^2$ to be the disk of radius $e^{-C}$ we can interpolate $CJ_3$ with the normal coordinate $b_1$ in such a way that the requirements are satisfied.
\end{proof}

\begin{df}
For $\Ll$ tame, call a local normal form  $\psi$ on a neighborhood $V$ of $\crit(\Ll)$ $\Ll$-tame if, with respect a metric as in Theorem \ref{tmBaseCoord}, $V$ projects under $\Ll$ onto a uniform tubular neighborhood of $\Delta$ and the functions $q_1,q_2,q_3$ defined by $\psi$ as in equation \eqref{eqLocEdgeNormForm} factor as $\tilde{q}\circ\Ll$ for $\tilde{q}$ Lipschitz with respect to said metric. Given a Lagrangian section $\sigma$ which does not meet $\crit(\Ll)$ we say that the local normal form is $(\Ll,\sigma)$-tame if, further, $\sigma$ is uniformly disjoint from $\crit(\Ll)$ in the coordinates defined by $\psi$.
\end{df}
\begin{df}
Let $(X,g_X),(Y,g_Y)$ be Riemannian manifolds. A submersion $f:X\to Y$ is \emph{quasi-Riemannian} there is a constant   $C>1$ such that for each $p\in Y$ and each $q\in f^{-1}(p)$ denoting by $g_q$ metric on $T_pY$ induced by considering the quotient metric on $T_qX/\ker df$ and identifying $T_qX/\ker df$ with $T_pY$, we have for each $v\in T_pY$
\begin{equation}
\frac1{C} g_q(v,v)<g_Y(v,v)<Cg_q(v,v).
\end{equation}
\end{df}
\begin{tm}[\textbf{The uniform structure on $M$}]\label{tmCanonicalMetric}
Suppose $\Ll$ is tame. For each section $\sigma$ of $\Ll$ which does not meet $\crit(\Ll)$ there is an $\omega$-compatible almost complex structure $J_\sigma$ such that the associated metric $g_{J_\sigma}$ has a finite cover by uniformly bi-Lipschitz coordinate systems for any $\delta>0$ as follows.
\begin{enumerate}
\item Pre-images under $\Ll$ of simply connected open subsets of $B\setminus B_\delta(\Delta)$ with action angle coordinates for which $\sigma$ is the $0$-section.
    \item
     $\Ll^{-1}(B_\delta(\Delta))\setminus B_{\delta}(\crit(\Ll))$ with coordinates given on base as in \ref{tmBaseCoord} and on the fibers coordinates $t\in [-T,T]^n$ giving the time $t$ Hamiltonian flow of the base coordinates with initial condition $\sigma$.
 \item
    $B_{\delta}(\crit(\Ll))$ with $(\sigma,\Ll)$-tame local normal coordinates.
 \end{enumerate}
 Any two almost complex structures $J_1$ and $J_2$ satisfying this condition give rise to equivalent metrics.
 With respect to such a metric, we have $\Ll$ is a quasi-Riemannian submersion on the complement of $B_{\delta}(\crit(\Ll))$ and $\sigma$ is a uniformly Lipschitz embedding.
\end{tm}
\begin{proof}
We start by taking an open cover by coordinate systems as in the statement interpreting $B_\delta(\crit(\Ll))$ as the ball with respect to some $(\Ll,\sigma)$-tame local coordinates.

We remark that to obtain a $\sigma$-tame local normal form we can start with an arbitrary tame local normal form, which exists by the tameness assumption and Theorem \ref{tmBaseCoord}, and modify by flowing $\sigma$ to the section $x_1=x_2=\delta,\theta=0$. Here we are referring to the normal coordinates introduced in Theorem \ref{tmLocNorFo}.

Observe also that for a given $\delta$ which is small enough there is a fixed $T$ such that the sets of the last two types are a cover of $\Ll^{-1}(B_\delta(\Delta))$. To see this note that $|dh|$ as in \eqref{eqLocalSemiGlobal} is roughly the time it takes for a point moving out of the normal neighborhood to hit it again under the Hamiltonian flow of the local normal coordinate on the base. See \cite{Ngoc03}. Tameness of $\Ll$ thus guarantees uniform boundedness of this time for fixed $\delta$.

We need to show that on the overlaps the coordinates are bi-Lipschitz equivalent. We comment on this for the overlap between the second and third type of coordinate system, the others being the same. Observe first that in both systems, we have that $\Ll$ is a uniformly quasi-Riemannian submersion \textit{on the overlap} and $\sigma$ is a uniformly Lipschitz embedding. The map from Poisson coordinates $(I_1,\dots, I_n,t_1,\dots t_n)$ to local normal coordinates $(x,y,r,\theta)$ involves $\Ll,\sigma,$ and the differential of the Hamiltonian flow of the functions $q_1,q_2,r$. On the overlaps, the gradients of these functions are bounded above and below. Moreover, their Hessians are uniformly bounded since they are $r$ independent. From this it is straightforward to deduce the desired bi-Lipschitz equivalence.

Using these coordinates, we may pick uniformly bounded compatible almost complex structures on each coordinate neighborhood. On the overlaps, the various almost complex structures induce equivalent metrics. By Lemma \ref{lmEquivCont} they may be glued together while preserving the equivalence.

For the second statement note that both the Poisson coordinates corresponding to the coordinate system of Theorem \ref{tmBaseCoord} and a tame local normal form are determined uniquely up to bi-Lipschitz equivalence given those coordinates and a Lagrangian section.

Finally the claims about $\sigma$ and $\Ll$ are immediate in each of these coordinate systems, so the claim follows by the bi-Lipschitz equivalence of the coordinate systems on the overlaps.
\end{proof}
\begin{df}
An almost complex structure as in Theorem \ref{tmCanonicalMetric} is called $(\mathcal{L},\sigma)$-adapted. 
\end{df}
\begin{tm}
The metric associated with any  $(\mathcal{L},\sigma)$-adapted 
almost complex structure  is uniformly isoperimetric in the sense of Definition \ref{dfGoBound}.
\end{tm}
\begin{proof}
On the complement of $B_{\delta/2}(\crit(\Ll))$ the metric is locally uniformly equivalent to the Euclidean metric on $\R^n\times\mathbb{T}^n$. On
\[
B_{\delta/2}(\Delta)(\crit(\Ll))
\]
the metric is up to bi-Lipschitz equivalence translation invariant outside of a compact set. The claim follows by the remark after Definition \ref{dfGoBound}.
\end{proof}
The following Lemma will be used in Appendix \ref{AppA} in the course of the proof of Theorem \ref{tmExLGPot}.
\begin{lm}\label{tmGoodCoord}
There exists a Lipschitz function $\theta:M\to S^1$ with the following properties.

\begin{enumerate}
\item $\theta$ is $G$-invariant.
\item $d\theta$ generates $H^1(M;\Z)$.
\item For any $\delta>0$ the restriction of $\Ll\times\theta$ to $M\setminus B_{\delta}(\crit(\Ll))$ is a quasi-Riemannian submersion.
\end{enumerate}

\end{lm}
\begin{proof}
On $M\setminus \Ll^{-1}(B_{\delta}(\Delta))$ we pick a global angle function $\theta_3$ to $S^1$ as in the proof of Theorem so that $\sigma$ is mapped to $1$. Note that as $\delta$ goes to $0$ the differential $d\theta_3$ goes to $0$ except at the critical points where it converges to $\infty$. We renormalize as follows. On $U_1=\Ll^{-1}(B_{\delta}(\Delta))\setminus B_{\delta}(\crit(\Ll))$ consider the function $\tilde{\theta}_3$ obtained from the vector field $f(b)\partial_{x_3}$ normalized so as to map $U_1/G$ onto the upper half circle. Let $U_2:=(B_{\delta}(Crit\Ll)$. Then $U_2\cap U_1$ has two connected components. Thus we can extend $\tilde{\theta}_3$ smoothly to $U_2$ so that it maps to the lower half circle.  In this procedure we have no control from below on the gradient of $\tilde{\theta_3}$ in the region $B_{\delta}(\crit(\Ll))$, but outside that region $\tilde{\theta}_3$ is a quasi-Riemannian submersion. We can now glue together $\theta_3$ and $\tilde{\theta_3}$ to obtain a function as required.
\end{proof}
\subsection{Landau-Ginzburg potentials and walls}
Recall Definition \ref{dfLgPotential} of a Landau Ginzburg potential for $\Ll$ and a tautologically unobstructed Lagrangian. We will prove in Appendix \ref{AppA} the following theorem
\begin{tm}\label{tmExLGPot}
If $\Ll$ is tame then for each Lagrangian section $\sigma$ of $\Ll$ there is a Landau Ginzburg potential $(\pi,J)$ such that $J$ is $(\Ll,\sigma)$-adapted. Moreover, there is a canonical contractible set of such.
\end{tm}
\begin{rem}
If we drop the requirement that $J$ is $(\Ll,\sigma)$-adapted, it is not hard to come up with a Landau-Ginzburg potential. The difficulty in proving Theorem \ref{tmExLGPot} is in constructing a $(\pi,J)$ such that $J$ is geometrically bounded and tamed by $\omega$.
\end{rem}
The significance of an LG potential is given by the following theorem.
\begin{tm}\label{tmLGSphereDiskMax}
Fix a Landau Ginzburg potential $(J,\pi)$. Then any $J$-holomorphic sphere is contained in $\pi^{-1}(1)$. Furthermore, suppose $L\subset M$ is a tautologically unobstructed Lagrangian sub-manifold. Then there are no non-constant $J$ holomorphic disks with boundary on $L$.
\end{tm}
The proof of Theorem \ref{tmLGSphereDiskMax} relies on the following theorem.
\begin{tm}\label{tmSpheresMax}
Let $(M,J)$ be an almost complex manifold, let $V\subset \C^{n}$ be an open tubular neighborhood of a properly embedded complex hypersurface and let $f:M\to\C^n$ be a smooth map such that the restriction of  $f$ to $U:=f^{-1}(V)$ is $J$-holomorphic. Then  any closed $J$-holomorphic curve $u$ in $M$ meeting a point $p\in U$ satisfies $u\subset f^{-1}(p)\subset V$. Moreover, if $u$ is a  $J$-holomorphic map from a compact surface with boundary such that $u$ intersect $U$ but $\partial u\subset M\setminus U$, then $f(\partial u)$  represents a non-trivial class in $H_1(\C^n\setminus V;\Z)$.
\end{tm}

\begin{proof}[Proof of Theorem \ref{tmLGSphereDiskMax}]
By Theorem~\ref{tmSpheresMax} every $J$-holomorphic sphere is contained in a fiber of $\pi$. The regular fibers of $\pi$ have the homotopy type of tori and are thus a-spherical. The first part of the claim follows.

Suppose $u:(D,\partial D)\to (M,L)$ is $J$-holomorphic. The function $\pi$ can be considered as a function $f:M\to\C$ by adding either $\{0\}$ or $\{\infty\}$ to $\C^*$ in such a way that $f_*:H_1(L;\Z)\to H_1(\C\setminus 1;\Z)$ is trivial. Then the image of $u$ must intersect $f^{-1}(1)$ if $u$ is non-constant. Indeed, otherwise, if $\partial u$ is contractible in $L$, we have $\int u^*\omega=0$. This follows since $M\setminus f^{-1}(1)$ is topologically a-spherical being a torus fibration over $B_{reg}$ which has the homotopy type of a punctured Riemann surface. If $\partial u$ is non-contractible in $L$, then it is also non-contractible in $M\setminus f^{-1}(1)$. Either way, $u$ meets $f^{-1}(1)$. On the other hand $\partial u\subset L$, so by assumption, $f(\partial u)$ is homologically trivial in $\C\setminus\{1\}$. Moreover there is a neighborhood $V$ of $1$ such that $\partial u\subset M\setminus\pi^{-1}(V)$. The claim follows by Theorem~\ref{tmSpheresMax}.
\end{proof}
\begin{proof}[Proof of Theorem \ref{tmSpheresMax}]
The case $n>1$ can be reduced to the case $n=1$. Indeed, the complex analytic hypersurface in $\C^n$ is the zero locus of a global analytic function $g$ \cite{GH}, so the hypothesis with $n=1$ will be satisfied by the function $g\circ f$. So assume from now on $n=1$.  Pick an exact form $\omega_V$ with support in $V$, positive at $p$ and non-negative everywhere with respect to the standard complex structure on $\C$.
This can obviously be done in $\C$ since every $2$-form is closed and therefore exact. 
The pullback $f^*\omega_V$ is again an exact form. Thus $\int u^*f^*\omega_V=0$. But $f\circ u$ meets $p$ and is holomorphic on the pre-image of $V$, so unless $d(f\circ u)\equiv 0$ we have that $u^*f^*\omega_V$ is everywhere non-negative and positive somewhere. This contradiction proves the first part of the claim. For the second part, note that the same argument as above implies $\int u^*f^*\omega_V\neq 0$. Let $\alpha$ such that $d\alpha=\omega_V$. By assumption, $d\alpha=0$ on $\C^n\setminus V$ so
\[
\int u^*f^*\omega_V=\int f(\partial u)^*\alpha.
\]
The latter is only non-zero if $\partial u$ is a non-trivial class in $H_1(\C^n\setminus V;\Z)$.
\end{proof}

\begin{rem}
Observe that a Landau-Ginzburg potential is a flexible structure. Namely, given one such, we are free to deform $\pi$ away from $\crit(\Ll)$ so long as the fibers remain symplectic.  We can then extend $J$ from a neighborhood of $\crit(\Ll)$ by arbitrarily extending to the vertical bundle determined by the symplectic fibration, and then uniquely extend by the condition of $J$-holomorphicity and horizontality. Note also that carefully examining the proof of Theorem~\ref{tmLGSphereDiskMax} it becomes evident that one only needs to require $J$-holomorphicity of $\pi$ for the pre-image of a neighborhood of $0$.
\end{rem}

\subsection{Admissible Lagrangians}

Consider the Lagrangian torus fibration $\mathcal{L}:M\to B$ and the section $\sigma$ of $\Ll$. Let $b$ be a regular value of $\Ll$ and let $p=\sigma(b)$. Write $L_b:=\Ll^{-1}(b)$. There is a natural isomorphism $\iota:T_pL_b=T^*_bB$. This arises from the short exact sequence
\[
0\to T^*_bB\to T^*_pM\to T^*_pL \to 0,
\]
and the identification of $T_pL_b$ with the kernel of $T^*M\to T^*L$ via $\omega$.\textit{ Note that conormality also makes sense over singular values of $\Ll$ since $\sigma$ is disjoint of the critical points of $\Ll$.}

Given a locally affine submanifold $C\subset B_{reg}$, the conormal $C^\perp_p\subset T_pL_b$ of $C$ at $p$ is defined to be the annihilator of $T_bC$ under $\iota$. The image of the right inverse $\sigma_*:T_bB\to T_pM$ is isotropic. Therefore $Im\sigma+C^\perp_p\subset T_pM$ is Lagrangian.  If $C$ has rational slopes, $C^\perp_p$ defines a sub-torus of $L_b$ going through $p$ of dimension complementary to $C$. This sub-torus is also referred to as the conormal of $C$. We thus obtain a Lagrangian submanifold of $M$ fibering over $C$ with fiber over $c\in C$ the conormal sub-torus through $\sigma(c)$.

For $M$ connected, with a fixed almost complex structure $J$, a Hamiltonian $H$ is said to be \textit{Lyapunov} if there are constants $c,\lambda$ such that for any pair of points $x_0,x_1\in M$ and any $t>0$ we have
\[
\frac1cd(x_0,x_1)e^{-\lambda t}d(\psi_H^t(x_0),\psi_H^t(x_1))\leq cd(x_0,x_1)e^{\lambda t},
\]
the distance measured with respect to $J$.
The Lyapunov property is invariant under equivalence of metrics.

\begin{df}
A properly embedded Lagrangian submanifold $L\subset M$ is $(\mathcal{L},\sigma)$-adapted if after applying a Hamiltonian isotopy $\psi$ generated by a Lipschitz Lyapunov Hamiltonian, $\psi(L)$ is a as follows. There is a compact subset $K\subset M$ and an $N\subset B$ such that in affine coordinates $N\cap B_{reg}$ is locally affine with rational slopes and such that $\psi(L)\setminus K$ is the conormal of $\sigma(N)$. 
\end{df}

\begin{df}\label{dfAdmLag}
Fix a Landau-Ginzburg potential $(\pi,J)$. A Lagrangian submanifold is said to be admissible if it is
\begin{enumerate}
\item uniformly disjoint of $\pi^{-1}(0)$,
\item $(\Ll,\sigma)$-adapted,
\item spin, and,
\item  $arg(\Omega|_{TL})$ is exact.
\end{enumerate}
\end{df}
These give a grading on Hamiltonian chords, and orientation of the moduli spaces of discs with boundary on the Lagrangian, and, as we shall see, compactness of the moduli spaces.
\textit{Henceforth we fix an LG potential $(\pi,J)=(\pi_{LG},J_{LG})$ and consider admissibility with respect to it.}
\begin{exs}
\begin{enumerate}
    \item For any regular value $b\in B$, $\mathcal{L}^{-1}(b)$ is $\mathcal{L}$-admissible.
    \item
        Lagrangian sections $\sigma:B\to X$ which do not meet $\crit(\Ll)$ can be made admissible by a Hamiltonian isotopy.
    \item
        Conormals which are constructed out of a section  $\sigma$ as before and a submanifold of $B$ which is contained in a single affine chart and have linear slopes near infinity.
\end{enumerate}
\end{exs}

\begin{df}\label{dfLagGeoTame}
Let $J$ be uniformly isoperimetric. A properly embedded Lagrangian sub-manifold $L$ is said to be \textit{uniformly isoperimetric} if there are constants $\delta$ and $c$ such that any loop in $L$ of diameter less than $c\delta$ is contractible in $L$ and for any chord $\gamma:[0,1]\to M$ with endpoints on $L$ and of length $\leq \delta$ there is a disc $D\subset M$ and a path $\tilde{\gamma}:[0,1]\to L$ of length $\leq c\delta$ such that $\partial D=\gamma\cup\tilde{\gamma}$ and such that $Area(D)\leq c\ell^2(\gamma)$.

$L$ is said to be \textit{Lipschitz} if there exist constants $r$ and $C$ such that denoting by $d_L$ the induced distance on $L$, the following two hold.
\begin{enumerate}
\item
For $x,y\in L$ such that $d_M(x,y)<r$ we have $d_L(x,y)<Cd_M(x,y)$. In other words, on balls $B_r(x)\cap L$, the two distance functions are equivalent uniformly in $x\in L$.
\item
For any $x\in L$ we have that $B_r(x)\cap L$ is contractible in $L$.
\end{enumerate}
\end{df}
\begin{lm}
If $L$ is Lipschitz and $M$ is uniformly isoperimetric with respect to some $J$ then $L$ is also uniformly isoperimetric with respect to $J$.
\end{lm}
\begin{proof}
Given a small enough chord $\gamma$ in $M$ with endpoints in $L$, the Lipschitz condition guarantees it can be completed to a closed loop by concatenating a chord $\tilde{\gamma}$ contained in $L$ and such that $\ell(\tilde{\gamma})\leq C\ell(\gamma)$. Making $\ell(\gamma)$ even smaller, we can fill $\gamma\cup\tilde{\gamma}$ by a disc of area
\[
\leq c\ell(\gamma\cup\tilde{\gamma})^2\leq c(1+C)^2\ell(\gamma)^2.
\]
\end{proof}
\begin{lm}
An admissible Lagrangian $L$ is Lipschitz and therefore uniformly isoperimetric.
\end{lm}
\begin{proof}
First observe that the both the property of being Lipschitz and of being admissible are invariant under equivalences of metrics and are thus preserved under the flow of a Lyapunov Hamiltonian. Thus we may assume that outside of a compact set $L$ is the conormal with respect to $\sigma$ of a locally affine submanifold of $B_{reg}$. Since $\sigma$ is uniformly disjoint of $\crit(\Ll)$ we have we can find a finite cover of $B_{reg}$ so that with respect to any metric as in Theorem \ref{tmCanonicalMetric} $L$ is covered by components equivalent outside of a compact set to an affine subspace of $U_i\times [-T_i,T_i]^n$.

\end{proof}
\section{Wrapped Floer cohomology}\label{SecWrapped}
\subsection{Admissible Floer data and $C^0$ estimates}
\textit{Henceforth fix an admissible Lagrangian section $\sigma$ and a $(\sigma,\Ll)$-tame LG potential $(\pi_{LG},J_{LG})$.} We allow  $(\pi_{LG},J_{LG})$ to be time dependent.
\begin{df}\label{dfAdmFloer}
Let $L_0$ and $L_1$ be admissible Lagrangians. A time dependent Hamiltonian $H$ is said to be \emph{admissible} if $H$ is proper, bounded below, and, Lipschitz and Lyapunov. It is said to be \emph{$(L_0,L_1)$-admissible} if there is a compact set $K$ and an $\epsilon>0$ such that, denoting by $\psi_H$ the time-1 flow of $H$, we have $d(L_1\setminus K,\psi_H(L_0))>\epsilon$. A time dependent almost complex structure $J_t$ is said to be admissible if $J_i=J_{LG}$ for $i=0,1$ and $J_t$ is $(\Ll,\sigma)$-adapted.
A pair $(H,J)$ of time dependent Floer data is called $(L_0,L_1)$-admissible if $H$ and $J$ are each admissible and $J$ is compatible with $\omega$ in the neighborhood of any Hamiltonian chord running from $L_0$ to $L_1$.
\end{df}
\begin{lm}\label{lmAbundance}
\begin{enumerate}
\item
For any Lipschitz Hamiltonian $H_0$ which is proper and bounded below, and any pair of  admissible Lagrangians $L_0,L_1$, there is an $(L_0,L_1)$-admissible $H$ such that $H>H_0$.
\item If $H$ is admissible for $(L_0,L_1)$ there is an $\epsilon>0$ for such that any Hamiltonian $H'$ satisfying $\|dH-dH'\|<\epsilon$ outside of a compact set is also admissible.
\end{enumerate}
\end{lm}
\begin{proof}
\begin{enumerate}
\item
We first point out that the claim is trivial if one of the Lagrangians is closed. Indeed, in that case, if $H_0$ is admissible, so is $aH+c$ for any $a>0$ and $c\in\mathbb{R}$.  

For non-closed Lagrangians we prove the claim in only for the case $L_0$ and $L_1$ are Lagrangian sections which is the only studied in this paper. Moreover, we restrict to the case $L_0=L_1$, the adjustment to the case where $L_0$ is the image of $L_1$ under the flow of a Lipschitz Lyapunov Hamiltonian being trivial. Hamiltonians which factor through $\Ll$ give rise rotations in the torus fibers. Let $x_1,x_2,x_3$ as in theorem \ref{tmBaseCoord}. 
Let $h:\R_+\to\R_+$ be a Lipschitz function which is linear of slope $n\in\mathbb{N}$ at infinity and equals $x^2$ near $0$. The flows of the functions $a\sum_{i=1}^3 h(|x_i|)$ displace $\sigma$ outside of a compact set by a uniform distance whenever $a\not\in\Z$. Note that this utilizes the fact that $\sigma\cap \crit(\Ll)=\emptyset$. Thus $h\circ\Ll$ is admissible. Since $H_0$ is Lipschitz with respect to the uniform structure induced by the integral affine coordinates, if we choose $n$ large enough we have $H_0<h\circ\Ll+C$ for $C$ some constant. So, we take $H=h\circ\Ll+C$. 


\item Immediate from the definition.
\end{enumerate}
\end{proof}

Let $(H,J)$ be $(L_0,L_1)$-admissible and suppose all the time-1 Hamiltonian chords of $H$ from $L_0$ to $L_1$ are non-degenerate. Denote by $\Omega(L_0,L_1)$ the space of paths from $L_0$ to $L_1$. Given Hamiltonian chords $\gamma_-,\gamma_+$ and a homotopy class $A$ of paths in $\Omega(L_0,L_1)$, let
\[
\M((H,J),\gamma_-,\gamma_+,A),
\]
be the set of solutions $u$ to Floer's equation
\[
\partial_su+J(\partial_tu-X_H)=0,
\]
such that $[u]=A$ and $\lim_{t\to\pm\infty}u_t=\gamma_\pm$. Write $E(u):=\frac12\int_{-\infty}^{\infty}|\partial_su|^2ds.$
\begin{lm}\label{lmFlTrajDiamEst}
 There exists a compact set $K=K(E)$ depending additionally on $H,J,L_0,L_1$, such that $u(\R\times[0,1])\subset K$ for any Floer solution $u$ with $E(u)<E$.
\end{lm}
\begin{proof}
This is a particular instance of the global estimate Theorem \ref{tmGlobalEst} in the appendix below. For sake of exposition we spell out the proof for this particular case. We first rule out the possibility of a sequence $t_i\to\infty$ such that $u_{t_i}\to\infty$. Let $A=[t_0,t_1]\times [0,1]$ and suppose $u(A)\subset M\setminus K'$ where $K'$ is the set such that $d(\psi_H(L_1\setminus K'),L_0)>\epsilon>0$. By Lemma \ref{lmLyapunovEst} the Lyapunov condition produces a $\delta>0$ such that $E(u_t)\geq\delta$ for any $t\in[t_0,t_1]$. Since $E(u)=\int E(u_t)dt$ it follows that $t_1-t_0\leq E(u)/\delta$.  On the other hand, by Cauchy Schwartz we have the inequality
\[
d(u_{t_0},u_{t_1})^2\leq (t_1-t_0){E(u;[t_0,t_1]\times [0,1])}.
\]
Combining these inequalities we see that we cannot have  a sequence $t_i\to\infty$ such that $u_{t_i}\to\infty$. That is, there is a compact set $\tilde{K}(E)$ such that $u_t$ meets $\tilde{K}(E)$ for every $t\in\R$. It remains to estimate the diameter of $u_{t}$ for any $t$. Here we apply the domain local estimate \ref{tmDiamEst} which is based on the monotonicity estimate of \cite{Sikorav94} with respect to any admissible $J$ and in particular satisfies the monotonicity inequality. Namely, we cover $u_t$ by a pair of half discs to obtain for the constant $R=R(E)$ of \ref{tmDiamEst} that $u_t$ is contained in the compact set $K(E)=B_{2R}(\tilde{K}(E))$.
\end{proof}

\subsection{Regularity}
Given an admissible Floer datum introduce the following moduli spaces.
For a homology class $A\in H_2(M;\Z)$ and a time dependent admissible almost complex structure $J_t$ let
\[
\M^*(\{J_t\},A)=\{(t,u)\in(0,1)\times  C^{\infty}(S^2, M)|\overline{\partial}_{J_t} u=0,[u]=A, \mbox{u is simple}\}
\]
the moduli space of simple $J$-holomorphic spheres.
For a pair of Hamiltonian chords $\gamma_1,\gamma_2$ denote by $C(L_0,L_1,\gamma_0,\gamma_1)$ the space of smooth maps $u:(-\infty,\infty)\times[0,1]$ with boundaries in $L_0, L_1$ respectively, and asymptotics $\gamma_0$ and $\gamma_1$ respectively. For a homotopy class $A$ of maps in $C(L_0,L_1,\gamma_0,\gamma_1)$, let
\[
\M((H,J),\gamma_0,\gamma_1,A):=\{u\in C(L_0,L_1,\gamma_0,\gamma_1)|[u]=A,(du-X_{H_t})^{0,1}=0\}.
\]
A pair $(H,J)$ of time dependent admissible Hamiltonian and almost complex structures are said to be regular if the following conditions hold.
\begin{enumerate}
\item The Hamiltonian chords going from $L_0$ to $L_1$ are all non-degenerate.
\item For all $A\in H_2(M;\Z)$, the moduli space $\M^*(\{J_t\},A)$ is smooth of the expected dimension.
\item For any $\gamma_0,\gamma_1$ of index difference $1$ or $2$, the moduli spaces $\M((H,J),\gamma_0,\gamma_1,A)$ are smooth of the expected dimension.
\item There is no $(t,u)\in \M(\{J_t\},A)$ with  and $v\in \M((H,J),\gamma_0,\gamma_1,A)$ and $s\in\R$ such that $A\neq 0$ and $v(s,t)\in u(S^2)$.
\end{enumerate}
\begin{tm}\label{tmGenReg}
The space of regular admissible Floer data is of second category in the space of all admissible Floer data.
\begin{proof}
We emphasize that our regularity assumption does not involve any statement concerning disks. Thus the statement of the present theorem is proven in, e.g., \cite{HoferSalamon} albeit without
the admissibility condition. We comment on the adaptation necessary due to this additional condition. The requirement of $\Ll$-adaptedness is open and thus has no affect in this regard. It remains to discuss the condition $J_i=J_{LG}$ for $i=0,1$. First note that the generic smoothness of the smooth part of the moduli spaces of Floer trajectories is unaffected by the restriction on $J_i$. Indeed, transversality can be achieved by perturbing $J_t$ for $t$ in any compact subset of $(0,1)$. The generic smoothness of the moduli spaces $\M^*(\{J_t\},A)$ is slightly different than usual. Namely, we do not assume that $J_0,J_1$ are regular for spheres, and correspondingly require smoothness of the space $\M^*(\{J_t\},A)$ only for $t\in(0,1)$. Sphere bubbling at any interior point of a strip is prevented by the usual argument. Namely, the proof that the evaluation map
\[
ev:\M^*(\{J_t\},A)\times\M((H,J),\gamma_0,\gamma_1,A)\times S^2\times \R\times(0,1)\to M\times M\times (0,1),
\]
is transverse to the identity goes through. Sphere bubbling at the boundary is prevented by  Theorem~\ref{tmSpheresMax}. Namely, all spheres are contained in $\pi^{-1}(0)$ which by definition does not meet an admissible $L_i$. Note we only assume $J=J_{LG}$ at the endpoints, so as to allow perturbations, which is why we don't invoke Theorem~\ref{tmSpheresMax} to rule out sphere bubbling at interior points. 
\end{proof}
\end{tm}
\subsection{The Floer complex}
Let $(H,J)$ be an admissible Floer datum for $L_0,L_1$. We define the Floer complex as follows. Denote by $\mathcal{P}$ the set of Hamiltonian chords going from $L_0$ to $L_1$. Our Lagrangians are graded, so denote by $\mathcal{P}^i\subset\mathcal{P}$ the chords of index $i$. Let $\Lambda_{nov}$ be the universal Novikov field over $\K$. Namely,
\[
\Lambda_{nov}=\left\{\sum_{i=0}^{\infty}a_it^\lambda_i:a_i\in \K,\lambda_i\in\R,\lim_{i\to\infty}\lambda_i=\infty\right\}.
\]
Define
\[
CF^i(L_0,L_1)=\oplus_{\gamma\in \mathcal{P}_i} \Lambda_{nov}\langle\gamma\rangle.
\]
Note that by admissibility and non-degeneracy, the set of chords is always finite.
For an element
\[
u\in  \M((H,J),\gamma_0,\gamma_1,A),
\]
define the topological energy
\[
E_{top}(u):=\int u^*\omega +\int_tH_t(\gamma_1(t))-H_t(\gamma_0(t))dt.
\]
The topological energy depends on $A$ only, so we will also write $E_{top}(A)$.
Define the Floer differential by
\[
\langle \mu^1\gamma_0,\gamma_1\rangle:=\sum_{A}t^{E_{top}(A)}\#\M(\gamma_0,\gamma_1,A)/\R
\]
\begin{tm}
For any regular admissible $(H,J)$ the Floer differential is well defined and satisfies $\mu^1\circ\mu^1=0$.
\end{tm}
\begin{proof}
First note that by Lemma \ref{lmFlTrajDiamEst}, we have an energy diameter estimate and thus Gromov-Floer compactness for moduli spaces of Floer trajectories in a given homotopy class. Now, since by Theorem \ref{tmLGSphereDiskMax} and admissibility we have that $L_0,L_1$ bound no $J_0,J_1$-holomorphic discs respectively, the identity $\mu^1\circ\mu^1=0$ is standard.
\end{proof}
\subsection{Continuation maps}
Let $(H^i,J^i)$ be $(L_0,L_1)$-admissible and regular for $i=0,1$. Suppose $H^1-H^0$ is bounded from below. A family $(H^s,J^s)$ of admissible Floer data is admissible if $\partial_sH^s$ is Lipschitz and bounded from below.

If $\gamma_i$ is a Hamiltonian chord of $H^i$ for $i=0,1$, denote by
\[
\M(\{(H^s,J^s)\},\gamma_0,\gamma_1,A)
\]
the moduli space of Floer solutions, and by
\[
\M^*(\{J^s_t,A\})=\{t,s,u|\overline{\partial}_{J_{t,s}}u=0,[u]=A,\mbox{u is simple}\}.
\]
We say that the homotopy $(H^s,J^s)$ is regular if
\begin{enumerate}
\item For all $A\in H_2(M;\Z)$, the moduli space $\M^*(\{J_t^s\},A)$ is smooth of the expected dimension.
\item For any $\gamma_0,\gamma_1$ of index difference $0$ or $1$, the moduli spaces $\M((H,J),\gamma_0,\gamma_1,A)$ are smooth of the expected dimension.
\item There is no $(s,t,u)\in \M(\{J^s_t\},A)$ and $v\in \M(\{(H^s,J^s\}),\gamma_0,\gamma_1,A)$ such that $v(s,t)\in u(S^2)$.
\item There are no trivial solutions, i.e., $s$-independent solutions, to Floer's equation.
\end{enumerate}
\begin{tm}
The space of regular admissible homotopies is of second category in the space of all admissible homotopies.
\end{tm}
\begin{proof}
This is also proven, without the admissibility condition, in \cite{HoferSalamon}. Namely, only difference to the discussion in the proof of Theorem~\ref{tmGenReg} is that the dimension of  $\M^*(\{J_t^s\},A)$ is raised by one because of the additional $s$ dependence. This is offset by the requirement that a sphere bubble at the point $(s,t)$ has to be $J_{s,t}$ holomorphic, so that sphere bubbling remains codimension $2$.
\end{proof}

Define the map $\kappa:CF^*(H^0,J^0)\to CF^*(H^1,J^1)$ by
\begin{equation}\label{eqContinuation}
\langle \kappa\gamma_0,\gamma_1\rangle:=\sum_{A}t^{E_{top}(A)}\#\M(\{(H^s,J^s),\gamma_0,\gamma_1,A)
\end{equation}
Here $E_{top}(A)$ is defined just as before.
\begin{tm}\label{tmContinuation0}
The map $\kappa$ is well defined given $(H^s,J^s)$ and is a chain map over the Novikov field.
\end{tm}
\begin{proof} The bound on $\partial_sH_s$ implies an a-priori estimate on the geometric energy on a  homotopy class of Floer solutions with fixed asymptotic conditions. The $C^0$ estimate is given in Theorem \ref{tmGlobalEst}. Sphere bubbling in the interior is ruled out by regularity. Disk and sphere bubbling on the boundary are ruled out by Theorem~\ref{tmLGSphereDiskMax}. The claim is now standard.
\end{proof}
The following theorem is, in light of existing literature, only a slight variation on the previous claims and is thus stated without proof.
\begin{tm}\label{tmContinuation1}
\begin{enumerate}
\item Different choices of homotopy give rise to chain homotopic definitions of $\kappa$. In particular, the map on homology is independent of the choice.
\item $\kappa$ is functorial at the homology level with respect to concatenations.
\item If $|H^0-H^1|$ is bounded, $\kappa$ induces an isomorphism on homology.
\end{enumerate}
\end{tm}
\subsection{Wrapped Floer cohomology}
For a pair of admissible Lagrangians $L_1,L_2$ denote by $\mathcal{F}_{\mathcal{L}}(L_1,L_2)$ the set of $(L_0,L_1)$ admissible data. The abundance of such data is guaranteed by Lemma \ref{lmAbundance}.  Moreover, it is an easy consequence of  Lemma \ref{lmAbundance} that  $\mathcal{F}_{\mathcal{L}}(L_1,L_2)$ admits a cofinal sequence with respect to the relation
\begin{equation}\label{eqRelation}
(H_1,J_1)\leq (H_2,J_2)\iff H_2-H_1>C>-\infty
\end{equation}
for some constant $C$. Given a pair $(H_1,J_1)\leq (H_2,J_2)\in  \mathcal{F}_{\mathcal{L}}(L_1,L_2)$ Theorems \ref{tmContinuation0} and \ref{tmContinuation1} guarantee the existence of a canonical map
\[
HF^*(L_1,L_2;H_1,J_1)\to HF^*(L_1,L_2;H_2,J_2).
\] 
We can thus make the following definition.
\begin{df}
We define
\[
HW^*_{J_{LG}}(L_1,L_2):=\varinjlim_{(H,J)\in\mathcal{F}_{\mathcal{L}}(L_1,L_2)}HF^*(L_1,L_2;H,J).
\]
We shall omit $J_{LG}$ from the notation henceforth, but note that the definition possibly depends on the choice of Landau-Ginzburg potential.
\end{df}

\subsection{Composition and the unit}
  \begin{df}\label{dfAdmFlSur}
 A Floer datum $F=(\{J_z\},\mathfrak{H})$ is called admissible if
 \begin{enumerate}
 \item
 $J_z$ are all $(\Ll,\sigma)$-adapted,
 \item denoting by $U_{LG}$ an open neighborhood of $\pi_{LG}^{-1}(0)$, for any $z\in\partial\Sigma$ we have $J_z|_{U_{LG}}=J_{LG}$,
 \item $\mathfrak{H}$ takes values in admissible Hamiltonians,
 \item $\mathfrak{H}$ satisfies the positivity condition \eqref{EqMonHom} of Appendix \ref{AppLipFloer},
 \item  and the Floer data are of the form $H_idt$ on the strip-like ends for $H_i$ which are $(L_i,L_{i+1})$-admissible.
 \end{enumerate}
 \end{df}

 Consider a triple of admissible $L_1,L_2,L_3$ and correspondingly admissible and regular $(H_{12},J_{12}),(H_{23},J_{23}),$ and $ (H_{13},J_{13})$ such that $H_{13}\geq 2\max\{H_{12},H_{23}\}$. Assume further that no endpoint of an $H_{12}$ chord is the starting point of any $H_{2,3}$ chord. Denote by $D_{1,2}$ the disk with one positive puncture and two negative ones.

\begin{tm}
For a generic choice of admissible interpolating datum  $(\mathfrak{H},J)$ on the disc $D_{1,2}$, the corresponding Floer moduli spaces give rise to an operation
\[
\mu^2:CF^*(L_1,L_2;H_{12},J_{12})\otimes CF^*(L_2,L_3;H_{23},J_{23})\to CF^*(L_1,L_3;H_{13},J_{13}).
\]
The induced operation on cohomology is independent of the choice of interpolating datum, is associative and commutes with continuation maps.
\end{tm}
\begin{proof}
The definition of regularity for a  Floer datum on $D_{1,2}$ is an obvious variation on the previous case as is the proof that a generic Floer datum is regular. The $C^0$ estimate is given in Theorem \ref{tmGlobalEst} and, for relations such as associativity, in the discussion following it. 

\end{proof}
Similarly, let $D_{1,0}$ have one positive puncture. A Floer datum $(\mathfrak{H},J)$ is admissible if it satisfies the same condition as in Definition \ref{dfAdmFlSur}. We then have the following theorem whose proof requires no new ideas and is thus omitted.
\begin{tm}
For generic admissible choices we get a map $e_L:\Lambda\to HF^0(L,L;H)$.
$e_L$ commutes with continuation maps. The induced map $e_L:\Lambda\to HF^*(L,L)$ gives rise to a unit by $1_{L}:=e_L(1)$.
\end{tm}
The set of admissible Lagrangians with $Hom$ spaces given by wrapped Floer cohomology thus defines a category $\mathcal{WH}_{\Ll}$ which we refer to as \textit{the wrapped Donaldson-Fukaya category.}

\subsection{The PSS homomorphism}
Let $L$ be an admissible Lagrangian and let $f:L\to\R$ be a proper Morse function. Let $g$ be a metric on $L$ such that $(f,g)$ is Morse-Smale. Let $(H,J)$ be an $(L,L)$-admissible Floer datum. Denote by $(CM^*(f,g),d_M;\Lambda)$ the Morse complex. We define a chain map
\[
f_{PSS}:CM^*(f,g)\to CF^*(H,J),
\]
as follows. Let $D_{0,1}$ be the disk with one positive strip-like end. Let $F$ be an admissible Floer datum on $D_{0,1}$ with $H,J$ on the strip-like end. On $D_{0,1}$ there is a distinguished marked point obtained as follows. First note that $D_{0,1}$ can be compactified to a disc $\overline{D}_{0,1}$ with a marked point $q$ so that the cylindrical end is defined on a punctured neighborhood of $q$. The cylindrical end at the puncture induces a tangent direction $v$ at $q$ by considering the tangent to the ray $(0,\infty)\times\{1/2\}.$ Call a biholomorphism of $\overline{D}_{0,1}$ with the upper hemisphere in the Riemann sphere \emph{admissible} if $q$ is mapped to $0$ and and $v$ is mapped  to the imaginary axis. Different admissible biholomorphism are related by scaling which does not affect the point at $\infty$. Thus there is a unique point on $\partial D_{0,1}$ which is mapped to $\infty$ under any admissible biholomorphism. Denote this point by $p_d$. Given a critical point $p\in CM^*(f,g)$ and a chord $x\in CF^*(h,J)$ let $\mathcal{M}(f,g,H,J,x,p)$ the space of configurations consisting of pairs $(u,\gamma)$ where $u$ is a Floer solution on $D_{0,1}$ with the Floer datum $F$ and $\gamma:(-\infty,0]\to L$ is a half gradient line for $(f,g)$. For  a configuration $(u,\gamma)$ define the energy by $E_{top}(u,\gamma):=E_{top}(u)$. The virtual dimension of $\mathcal{M}(f,g,H,J,x,p)$ is the difference $i_{Morse}(p)-i_{CZ}(x)$. Thus, for generic admissible choices, this is a zero dimensional smooth oriented manifold. Define the degree $0$ map
\[
\langle f_{PSS}(p),x\rangle:=\sum_{(u,\gamma)\in\M(f,g,H,J,x,p)}t^{E_{top}(u)}.
\]
The following theorem is standard if one takes into account that there are no $J$-holomorphic disks on an admissible Lagrangian and is thus stated without proof.
\begin{tm}
$f_{PSS}$ is a chain map. The cohomology level map commutes with the continuation maps in both Floer and Morse cohomology. In particular we have a canonical map
\[
f_{PSS}:H^*(L;\Lambda)\to HW^*_{J_{L_G}}(L,L).
\]
Moreover, $f_{PSS}$ is a unital algebra homomorphism.
\end{tm}

\begin{tm}
If $L$ is a compact admissible Lagrangian, $f_{PSS}$ is an isomorphism of $\Lambda$-algebras.
\end{tm}
\begin{proof}
First we prove that for a compact Lagrangian, the continuation maps occurring in the definition of $HW^*_{J_{L_G}}(L,L)$ are all isomorphisms. For this note that the Floer homology $HF^*_{J_{L_G}}(L,L)$ only depends on the slope at infinity. In defining the continuation maps between different slopes we can choose Hamiltonians which coincide on arbitrarily large subsets.  By compactness, changing the slope at infinity does not introduce any new chords. Thus the continuation map is arbitrarily close to the identity and thus is the identity. 

Since $L$ is compact, we can find compactly supported Hamiltonians $H$ which such that the time $1$ flow of $L$ intersects $L$ transversely. Using such compactly supported Hamiltonians we can define the inverse to the PSS homomorphism. The continuation maps from compactly supported Hamiltonians to those which have positive slope at infinity is the an isomorphism by the same argument as in the previous paragraph. The claim follows.
\end{proof}
\subsection{The closed string sector}
We define the set of \emph{closed-string admissible} Hamiltonians $\mathcal{H}_{\mathcal{L},closed}$ to be those $H$ which are admissible as per Definition \ref{dfAdmFloer} and in addition satisfy that there exists an $\epsilon>0$ and a compact set $K$ such that 
\[
d(p,\psi_H(p))>\epsilon
\]
whenever $p\not\in K$. The distance here is taken with respect to the metric induced by the fixed almost complex structure $J_{LG}$.  

We denote by $\mathcal{F}_{\mathcal{L},closed}$ the set consisting of pairs $(H,J)$ where $H\in \mathcal{H}_{\mathcal{L},closed}$ has Morse-Bott non-degenerate flow and $J$ is an almost complex structure which is metrically equivalent to $J_{LG}$ the $(H,J)$ and is regular for the definition of Floer cohomology.  Any element of $\mathcal{F}_{\mathcal{L},closed}$ is dissipative in the sense of \cite{Groman15}. Thus given an admissible Hamiltonian $H$, we define the Floer homology $HF^*(H)$ in the usual way. We omit $J$ from the notation since for different choices we get isomorphic Floer homologies. 
The set $\mathcal{F}_{\mathcal{L},closed}$ is a bidirected system with respect to the relation $\leq$ described in \eqref{eqRelation}. Given a pair $(H_1,J_1)\leq (H_2,J_2)\in \mathcal{F}_{\mathcal{L},closed}$ there is a canonical continuation map $HF^*(H_1)\to HF^*(H_2)$. 

We thus define
\[
SH^*(M;\mathcal{L}):=\varinjlim_{H\in\mathcal{F}_{\mathcal{L},closed}}{HF}^*(H)
\]

The set $\mathcal{H}_{\mathcal{L}}$ is a monoidal indexing set in the sense of \cite{Groman15}. This means that for any $H_1,H_2\in \mathcal{H}_{\mathcal{L}}$ we can find an $H_3\in \mathcal{H}_{\mathcal{L}}$ so that for each $x\in M$ we have $H_3\geq 2\max\{H_1(x),H_2(x)\}$.  When this holds for a triple $H_1,H_2,H_3$ there is a canonical operation $HF^*(H_1)\otimes HF^*(H_2)\to HF^*(H_3)$ which commutes with the continuation maps. Thus we get an algebra structure on $SH^*(M;\mathcal{L})$ over the Novikov field. Using standard constructions $SH^*(M;\mathcal{L})$  is in fact a unital $BV$ algebra over the Novikov field. 

On the other hand, taking the inverse limit of the system $\mathcal{F}_{\mathcal{L},closed}$, it is shown in \cite{Groman15} that there is a natural isomorphism
\[
H^*(M;\Lambda)\to\varprojlim_{H\in\mathcal{F}_{\mathcal{L},closed}}HF^*(H)
\]
induced by the PSS map. Composing with the natural map from the inverse to the direct limit we obtain 
a unital map
\[
H^*(M;\Lambda)\to SH^*(M;\mathcal{L}).
\]

\subsection{The closed open map}
For any admissible Lagrangian $L$, there is a unital map of algebras
\[
\mathcal{CO}:SH^*(M;\Ll)\to HW^*(L,L;\Ll),
\]
arising as follows. Let $(H_1,J_1)$ be admissible and let $(H_2,J_2)$ be $(L,L)$-admissible. Then we define an operation
\[
\mathcal{CO}:CF^*(H_1,J_1)\to CF^*(L,L;H_2,J_2),
\]
by choosing an admissible interpolating form on $T^1_1$ with $H_1$ on the interior puncture and $H_2$ on the boundary puncture.
\begin{tm}
For generic choice of interpolating data this gives rise to degree $0$ chain map. At the cohomology level it commutes with the continuation maps and the product $\mu^2$, and is unital.
\end{tm}
Thus the induced map $\CO:SH^*(M\Ll)\to HW^*(L,L;\Ll)$ is a unital $\Lambda$-algebra morphism.
\subsection{Proof of Theorem~\ref{tmc}}
\begin{proof}
The precise definition of good behavior at infinity is given in Definition \ref{dfAdmSYZ}. The set $\mathcal{H}_\Ll$ is the set of Hamiltonians $H:M\to \R$ which can be squeezed between a pair of functions of the form $h\circ\cL$ where $h$ is proper, bounded below, and  Lipschitz with respect to any of the equivalent metrics on the base $B$ described in Theorem \ref{tmBaseCoord}. The construction of $\Pi$ is given in Appendix \ref{AppA}. Admissible Lagrangians are given in Definition \ref{dfAdmLag}. Admissibility of Floer data is described in Definition of \ref{dfAdmFloer}. The construction of the wrapped Donaldson-Fukaya category has been carried out in \S\ref{SecWrapped}.
\end{proof}

\section{The Floer cohomology of a Lagrangian section}\label{SecComp}
Let $\sigma:B\to M$ be an admissible Lagrangian section of $\mathcal{L}$. Our aim in this section is to compute $HW^*(\sigma,\sigma)$. Let $\theta_n:M\to S^1$ be a global angle coordinate  representing the 1-form given in Theorem \ref{tmBasicTop}\ref{tmBasicTop3}. Assume without loss of generality that $\theta_n=Const$ along $\sigma$. For any chord $\gamma:[0,1]\to M$ starting and ending on $\sigma$, let $|\gamma|=\int_{\gamma}d\theta_n$.
\begin{lm}
$|\cdot|$ induces on $HW^*(\sigma,\sigma)$ the structure of a $\Z$-graded algebra.
\end{lm}
\begin{proof}
We have $d\theta_n\equiv0$ along $\sigma$. So, $|\gamma|\in\Z$. In fact, $|\gamma|$ is the number of times $\theta_n\circ\gamma$ winds around $S^1$.
In particular, for any admissible $H$, $|\cdot|$ defines a grading on $CF^*(\sigma,\sigma,H)$. For topological reasons, given any surface asymptotic to chords on $\sigma$, the total grading on the inputs must match that on the outputs. In particular, the differential, continuation maps and product respect the grading.
\end{proof}
We denote by $HW^{0,0}(\sigma,\sigma)\subset HW^0(\sigma,\sigma)$ the subalgebra generated by the elements $x$ such that $|x|=0$, i.e., the contractible elements. We similarly define the sub-algebras $HW^{0,\geq 0}(\sigma,\sigma)$ and $HW^{0,\leq 0}(\sigma,\sigma)$. In the following, let $G$ be the $n-1$-torus acting on $M$ and preserving $\Omega$.

Theorem~\ref{tmRoughform} in the introduction is a consequence of the following lemma. We abbreviate $\Lambda:=\Lambda_{nov}$.
\begin{lm}\label{lmRoughform}
We have
\begin{enumerate}
\item $HW^{0,0}(\sigma,\sigma)=\Lambda[H_1(G;\Z)]$.
\item There is an $x\in HW^0(\sigma,\sigma)$ ($y\in HW^0(\sigma,\sigma))$ with $|x|=1$ ($|y|=-1$) such that $HW^{0,\geq 0}\simeq HW^{0,0}(\sigma,\sigma)\otimes \Lambda[x]$ ($HW^{0,\leq 0}\simeq HW^{0,0}(\sigma,\sigma)\otimes \Lambda[y]$).
\end{enumerate}
\end{lm}
Before proving the Lemma~\ref{lmRoughform} we show how it implies Theorem~\ref{tmRoughform}.
\begin{proof}[Proof of Theorem \ref{tmRoughform}]
Since $HW^0(\sigma,\sigma)$ is a graded algebra with respect to $|\cdot|$, we have that there is a surjective map
\[
f:HW^{0,\leq 0}(\sigma,\sigma)\otimes_{HW^{0,0}(\sigma,\sigma)}HW^{0,\geq 0}(\sigma,\sigma)\to HW^0(\sigma,\sigma).
\]
Let $g:=xy\in HW^{0,0}(\sigma,\sigma)$, then $\ker f$ is generated as an ideal by $x\otimes y-g$. Indeed, by Lemma~\ref{lmRoughform}, $x$ and $y$ algebraically generate $HW^{0,\geq0}$ and $HW^{0,\leq 0}$ respectively over $HW^{0,0}$. Thus, since $HW^0$ is graded, we must have $\ker f\subset HW^{0,0}[x\otimes y]$. Moreover, we have $f(x\otimes y-g)=0$. Any element of $a\in HW^{0,0}[x\otimes y]$can be written as the sum of an element $b$ in the ideal generated by $x\otimes y-g$ and an element $c\in HW^{0,0}$. This holds in particular for $a\in\ker f$ but this time we can take $c\in HW^{0,0}\cap\ker f$. But $f$ is tautologically an isomorphism when restricted to $HW^{0,0}\subset HW^{0,\leq 0}\otimes_{HW^{0,0}}HW^{0,\geq 0}$. So, $c=0$. Part \ref{tmRoughformb} of the Theorem \ref{tmRoughform} now follows from Lemma \ref{lmRoughform}. Part~\ref{tmRoughformc} will follow in the course of the proof of Lemma \ref{lmRoughform} by the fact that we compute the Floer homology using Hamiltonians with chords of Maslov degree $0$ only.
\end{proof}

Before proceeding to the proof of Lemma~\ref{lmRoughform} we introduce the family of $\sigma$-admissible Hamiltonians we will be using. Denote by $H_{F_i}$ the components of the moment map $\mu$. Let $H_B=x_n$ as in Theorem \ref{tmBaseCoord}. We assume the discriminant locus is contained in the level set $\{x_n=0\}$. Denote by $\mathcal{H}_{\mathcal{B}}$ the set of Hamiltonians of the form
\[
H=\sum_if_i\circ H_{F_i}+g\circ H_B,
\]
where $f_i$ and $g$ are proper, convex, bounded below and linear near infinity with non-integer slope. We denote by $\mathcal{H}_{\mathcal{B}\pm}\subset \mathcal{H}_{\mathcal{B}}$ the subset consisting of those for which the unique minimum of $g$ is obtained at some $t>0 (t<0)$.
Note also that the sets $\mathcal{H}_{\mathcal{B}\pm}$ are each cofinal in the set of all $\Ll$-admissible Hamiltonians. Moreover, there is a cofinal sequence which is $\sigma$-admissible. We assume for simplicity that $\pi_{LG}$ is of the form $\pi=x_n+i\theta_n$ where $\theta_n:M\to S^1$ is the $G$-invariant function as in Lemma \ref{tmGoodCoord}. Let $U_{\pm}:=\{\pm\re\pi>0\}$. 

\begin{lm}\label{lmExplicitFormUplus}
Fix an $\epsilon>0$ and let $U_{\pm,\epsilon}:=\{\pm\re\pi\geq\epsilon\}$.
There are proper Lipschitz functions $f_{\pm,\epsilon}:\R^{n-1}\to\R$ so that $U_{+,\epsilon}$ is integral affine isomorphic to the region 
\[
V_{f_+,\epsilon}:=\{(x_1,\dots,x_n)|x_n\geq f_+(x_1,\dots,x_{n-1})\}
\]
with an analogous statement holding for $U_-$. In particular, $\mathcal{L}^{-1}(U_{\pm,\epsilon})$ is symplectomorphic to $V_{f_{\pm,\epsilon}}\times \mathbb{T}^n\subset T^*\mathbb{T}^n$. 
\end{lm}
\begin{proof}
This follows from Theorem \ref{dfAdmSYZ}. Namely, we can pick as fundamental domain $V$ the complement of product of $\Delta$ with the negative half line. Then $U_{+,\epsilon}\subset V\setminus B_{\epsilon}(\Delta)$ where the ball is taken with respect to the coordinates $x_1,\dots,x_n$ as in Theorem \ref{dfAdmSYZ}.  In particular Theorem \ref{dfAdmSYZ} says that on $U_{+,\epsilon}$ we have a bi-Lipschitz map between the coordinates $x_1,\dots,x_n$ and the action coordinates $I_1,\dots I_n$. Moreover, we may take $x_i=I_i$ for $1\leq 1\leq n-1$ since these are global integral affine functions. It follows that on $V$ we can write $x_n=f(I_1,\dots I_n)$ with $\frac {\partial{x_n}}{\partial I_n}>0$, and in fact the derivative is uniformly bounded away from $0$. Thus along the hypersurface $f(I_1,\dots, I_n)=\epsilon$ we can find a Lipschitz function $f_{+,\epsilon}$ so that $I_n=f_{+,\epsilon}(I_1,\dots ,I_{n-1})$. The claim follows.
\end{proof}

The following is an immediate consequence and is used repeatedly throughout the section. 
\begin{cy}\label{cyExplicitFormUplus}
There is a constant $C>0$ such that for any real number $R$ there exists a $c$ such that for any point $p$ in the region $\{x_n>c\}\cap \{\sum_{i=1}^{n-1}x_i^2\leq R\}$  the ball $B_R(p)\subset M$ is symplectomorphic to the ball $B_{CR}(x)\subset T^*\mathbb{T}^n$ for any point $x$. Here the metric on $T^*\mathbb{T}^n$ is the standard one induced from the flat metric on $\mathbb{T}^n$.
\end{cy}

\begin{proof}[Proof of Lemma \ref{lmRoughform}]
We show the claim for $HW^{0,\geq 0}$, the other half being the same. Pick a cofinal sequence of Hamiltonians $H_j\in \mathcal{H}_{B+}$. We take $g_j$ such that near $H_B=0$ we have that $g'_j$ is sufficiently small so that under the Hamiltonian flow $\psi_j$ of $H_j$, the Lagrangian $\psi_j(\sigma)$ stays out of a fixed uniform neighborhood of $\pi^{-1}(0)$ and thus remains admissible. Let $J_j$ and $J'_j$ be admissible time dependent almost complex structures. Denote by $\psi^H_t$ the time $t$ flow of the Hamiltonian $H$. For any $j$ let $A_j=CF^*(\sigma,\sigma;H_j,J_j)$ and let $B_j=CF^*(\sigma,\psi^{H_j}_1(\sigma),0,J'_j)$. Then there is a quasi-isomorphism, in fact, a canonical isomorphism of complexes, $f_j:A_j\to B_j$. Indeed, let $\tilde{J}_{j}$ be the time dependent almost complex structure defined by $\tilde{J}_{j,t}=\psi_t^{H_j,*}J_j$.  There is a tautological isomorphism between $A_j$ and the complex $C_j= CF^*(\sigma,\psi^{H_j}_1(\sigma),0,J_{j,t})$. The desired quasi-isomorphism is then obtained by composing with the continuation map for going from $\tilde{J}_{j}$ to $J'_j$. The latter continuation map may appear to be ill defined at first sight since, strictly speaking,  $\tilde{J}_{j,t}$ is not admissible since $\pi$ is not necessarily $\tilde{J}_{j,1}$-holomorphic. However, $\pi\circ\psi_1^{H_j}$ is $\tilde{J}_{j,1}$-holomorphic. Moreover, using the hamiltonian isotopy there is a path $(\pi_s,J_s)$ of Landau-Ginzburg potentials for which $L$ is admissible and which connects both Landau-Ginzburg potentials. This is sufficient for defining continuation maps. By the same token, the maps $f_j$ commute up to homotopy with the continuation maps
\[
\kappa_{ij}:CF^*(\sigma,\sigma,H_i)\to CF^*(\sigma,\sigma,H_j).
\]
Finally, the $f_i$ intertwine the product up to homotopy, where the product on the right is defined by counting $J'_j$-holomorphic triangles. At last observe that since all the chords are concentrated in degree $0$, continuation maps are isomorphisms and independent of any choices already at the chain level.

We now pick an almost complex structure $J$ such that $\pi$ is $J_t$ holomorphic \textit{for all $t\in[0,1]$}. Assume at first that such a $J_t$ can be chosen with sufficient genericity for the definition of the Floer differential and product. Since $\pi$ is $J$-holomorphic, we have that all the holomorphic polygons satisfy a maximum principle with respect to $\re\pi$. On the other hand, we have that all the relevant intersection points are contained in the $\Ll$-invariant region $U_+:=\{\re\pi>0\}$. Thus all the relevant Floer solutions are contained in the region $U_+$. Since $U_+$ is simply connected, and there are no singularities of $\mathcal{L}$ in $U_+$, we can use action angle coordinates to embed $U_+$ into $T^*G\times T^*S^1$ in such a way that the functions $H_{F_i}$ become the standard action coordinates on $T^*G$. Write this as $\phi:U_+\to T^*G\times T^*S^1$. Let $V=\phi(U_+)$. The function $H_B\circ\phi^{-1}$ extends from $V$ to a function $H_{n}$ on all of $T^*G\times T^*S^1$ satisfying $\partial_{n}H_{n}<0$ on the complement of $V$. Indeed, we have that $\partial_{n}H_{n}<0$ on the $\partial V$. Similarly, $\sigma$ extends to a section $\sigma'$ of the standard torus fibration on $T^*G\times T^*S^1$. Finally $\pi, J$ extend so that $\pi$ is $J$-holomorphic. On $T^*G\times T^*S^1$ we have a grading , still denoted by $|\cdot|$, given by the winding around the last factor. This intertwines under $\phi$ with the grading $|\cdot|$. The sequence of function $H'_j$ thus computes $CF^{0,\geq 0}(\sigma',\sigma')$. Moreover, all the relevant solutions are contained in $V$. Thus we have set up a graded isomorphism of $\Lambda$-algebras $HW^{0,\geq 0}(\sigma,\sigma)=HW^{0,\geq 0}(\sigma',\sigma')$. But $\sigma'$ is the cotangent fiber in the cotangent bundle of $G$.  By invariance of the wrapped Floer cohomology under Hamiltonian isotopy which is generated by linear at infinity Hamiltonian isotopy, we find that
\[
HW^{0,0}(\sigma',\sigma')=\Lambda[H_1(G;\Z)],
\]
and for an appropriate generator $x$,
\begin{equation}\label{eqPosHW0}
HW^{0,\geq0}(\sigma',\sigma')=HW^{0,0}(\sigma',\sigma')\otimes\Lambda[x].
\end{equation}
One way to prove this is to rely on a well known computation of the wrapped Floer cohomology of the cotangent fiber $\phi$ in $T^*S^1$ according to which
\begin{equation}\label{eqPosHW1}
HW^0(\phi)=\Lambda[x,x^{-1}],
\end{equation}
and
\begin{equation}\label{eqPosHW2}
HW^{0,\geq 0}(\phi)=\Lambda[x].
\end{equation}
The claim \eqref{eqPosHW0} follows from \eqref{eqPosHW1} and \eqref{eqPosHW2} by the Kunneth formula.

It remains to remove the assumption of regularity of $J$. For this note that for any fixed $E$ and $\delta>0$ and fixed $j_1,\dots,j_n$ there is an $\epsilon>0$ such that if $\|J'-J\|<\epsilon$ in $C^k$ for some sufficiently large $k$, then the image of any $J'$-holomorphic polygon with boundaries in $\psi_j(\sigma)$ is contained in an $\epsilon$ neighborhood of a $J$-holomorphic polygon with the same boundary conditions. Using this, the argument above works for Floer homology $HF^*(\cdot;\Lambda^E)$ over the ring $\Lambda^E=\Lambda_{\geq 0}/\Lambda_{\geq E}$. Moreover, the isomorphisms commute with the natural isomorphisms $HF^*(\cdot;\Lambda^{E'})\to HF^*(\cdot;\Lambda^E)$ for $E'>E$. Thus we get the isomorphism claim for $HF^*(\cdot,\Lambda)=\varprojlim HF^*(\cdot,\Lambda^E)\otimes\Lambda$.
The claim follows.
\end{proof}

\subsection{Symplectic cohomology}
Our next goal is to work out the Laurent polynomial $g$. We will use the following
\begin{tm}\label{tmFloerInt}
Let $\sigma:\R^n\to M$ be an admissible Lagrangian section of $\mathcal{L}$. Then the map
\[
\mathcal{CO}:SH^0(M;\mathcal{L})\to HW^0(\sigma,\sigma;\mathcal{L}),
\]
is an isomorphism of rings.
\end{tm}
To prove Theorem \ref{tmFloerInt} and to compute $SH^0$ we will again use the set of Hamiltonians $\mathcal{H}_{\mathcal{B}}$ as above. By definition, $H=\sum_if_i\circ H_{F_i}+g\circ H_B$. We first work out what are the contractible orbits of $H$. For this recall the cohomology $H^1(M;\Z)$ is generated by the $1$-form $d\theta_n$. Now, $d\theta_n(X_{H_B})\geq 0$ with equality only at the points where $X_{H_B}=0$. On the other hand $d\theta_n(X_{H_{F_i}})=0$.  Therefore, given that the functions $H_{F_i}, H_B$  commute with one another, the set of contractible $1$-periodic orbits of $H$ is a union $\mathcal{P}_1\cup\mathcal{P}_2$ where $\mathcal{P}_1=\mathcal{P}_1(H)$ is the set of periodic orbits of $H_F$ within the critical locus of $H_B$ and $\mathcal{P}_2=\mathcal{P}_2(H)$ is the set of periodic orbits in the fiber over the unique minimum of $g$. The non-contractible periodic orbits will be analyzed later.

The proof of Theorem \ref{tmFloerInt} will rely on the following lemma.
\begin{lm}\label{lmFloerInt}
\begin{enumerate}
\item
The periodic orbits contributing to $CF^0$ are all in $\mathcal{P}_2$. 
\item
For any $H\in \mathcal{H}_{B_+}$ The Floer differential on $CF^{0,\geq 0}(M;H)$ vanishes.
\end{enumerate}

\end{lm}
We prove Theorem~\ref{tmFloerInt} given Lemma~\ref{lmFloerInt}.
\begin{proof}[Proof of Theorem~\ref{tmFloerInt}]
First note that there is a bijection between the periodic tori in $\mathcal{P}_2$ and the Hamiltonian chords in the computation of the Lagrangian Floer homology of $\sigma$ where each chord maps to the periodic torus on which it lies. Consider the bases for $CF^0(L,L;H)$ and $CF^0(H)$ given by the chords and the generator of $0$th homology of the corresponding critical tori. Then the leading order term in the closed open map with respect to these bases is the identity. To see this note that by energy considerations as in Lemma \ref{lmTLDG} the lowest energy contributions to the closed open map are contained in a small neighborhood of the periodic tori. Such a neighborhood is symplectomorphic to a neighborhood inside the cotangent bundle of a torus. There, by exactness, the only Floer solutions contributing to the closed open map are the local ones. Moreover the closed open map is an isomorphism by the generation criterion.

It follows that at the chain level, $\mathcal{CO}$ is an isomorphism of $\Lambda$-vector spaces. Since the differential in both complexes is $0$, and there are no generators in degree $-1$, it follows that this induces an isomorphism on homology.
\end{proof}
The rest of the section is devoted to the proof of Lemma~\ref{lmFloerInt}.

\begin{lm}\label{lmDegCont}
Fix an $H\in\mathcal{H}_\mathcal{B}$. The non-trivial orbits of $H$ which are in $\mathcal{P}_1(H)$ have Conley-Zehnder index $\geq n-1$. The fixed points of $H$ which are in $\mathcal{P}_1(H)$ have Conley-Zehnder  index $\geq 2$.
\end{lm}
\begin{proof}

We first consider the fixed points. For them, the Conley-Zhender index equals the Morse index of $H$. We may choose the functions $f_i$ so that their critical points occur away from the fixed points. Moreover,  this choice can be made so that $H=\sum f_i(H_{F_i})+g(H_B)$ has a non-degenerate critical point at each fixed point. Moreover, since $H_B$ is independent of $\{H_{F_i}\}$ away from the critical points, we may assume its critical points contain those of $\crit{\Ll}$ (and in fact coincide with them). There are local coordinates where $H_B$ is the real part of the function $z_1\cdot\dots\cdot z_n$. So, in case $n\geq 3$ we have in particular that the Hessian of $H_B$ vanishes at each fixed point.  Therefore, the Morse index of the fixed points is determined by the $H_{F_i}$. In equivariant normal  coordinates these are the function $|z_0|^2-|z_i|^2$ for $i=1,\dots, n-1$. Any function of the functions $|z_i|^2=x_i^2+y_i^2$ has even index. Since the fixed points are not minima, the claim follows for $n\geq 3$. When $n=2$, the quadratic part of $H$ is in local coordinates $a(|z_1|^2-|z_2|^2) +b\re(z_1z_2)$ for some constants $a,b$. One readily computes that each eigenvalue of the hessian of the latter function has even multiplicity for any choice of $a,b$.  

We now consider the contribution of the non trivial periodic orbits. Let $\alpha$ be a periodic orbit in an edge $e$. Let $\psi_t$ denote the flow of $H_F$ and $\phi_t$ the flow of $H_B$. Since the horizontal and vertical flows commute, the Conley Zehnder index of $\alpha$ is, after shift by $n$, the Robin-Salamon index of the concatenation of $d\psi_t$ and $d\phi_t d\psi_1$. The Robin-Salamon index of the second part is zero since it has no crossings. Thus we need to compute $i_{RS}(d\psi_t)$.
One can homotope the linearization $d\psi_t$ of the flow to a concatenation of paths of matrices $A_tB_1$ and $B_t$ where $B_t$ is in the representation of $G$ and $A$ is a matrix of the form
\[
\left(
  \begin{array}{ccc}
    1 & 0 & 0 \\
    at & 1 & 0 \\
    0 & 0 & \id_{2n-2} \\
  \end{array}
\right)
\]
We have that $i_{RS}(B)=0$ since $B_t$ is unitary of fixed complex determinant. Thus
\[
i_{RS}(\alpha)=i_{RS}(AB_1)=-\frac12
\]
The orbits are transversally non-degenerate (because of the presence of $H_B$). So, after adding a time dependent perturbation we get $i_{CZ}(\alpha)\in\{0,-1\}$.  After adding $n$, the claim follows.
\end{proof}
 We now turn to prove the vanishing of the differential.

\begin{lm}\label{tmSH00}
The differential vanishes on $CF^{0,0}(H;J)$.
\end{lm}
Lemma~\ref{tmSH00} will follow from the next two lemmas.
Fix a collection of non-integer slopes $a=(a_0,a_1,\dots,a_{n-1})$. Let $g$ and $f_i$ be convex functions with slope at infinity equal respectively to $a_0$ and $a_i$. Suppose the minimum of $g$ occurs at $s=0$.  For each $s\in\R$ let $g_s(t):=g(t-s).$ Let $H_{a,s}:=g_s\circ H_B+\sum f_i\circ H_{F_i}$. The periodic orbits come in families, each family being a torus. Moreover, for $s>0$ these tori are labeled by their homology classes in $H_1(U_+;\Z)$, where $U_+=\{\re\pi>0\}$. Let $\gamma$ be the non-contractible periodic orbit of $H_B$ which generates $H_1(U_+;\Z)$ and winds positively with respect to $\theta_n$. Then we can label the periodic orbits by $H_1(G;\Z)\oplus \Z\langle\gamma\rangle$. Note that the splitting is not canonical since  $H_B$ can altered by adding a Lipschitz function of the coordinates on the fiber $F$. Denote by $\mathcal{P}_{0,a,s}$ the set of tori of contractible periodic orbits of $H_{a,s}$. For $s,s'>0$ we identify $\mathcal{P}_{0,a,s}=\mathcal{P}_{0,a,s'}$ in the obvious way.
\begin{lm}\label{lmValPresSH0}
For any $s_0,s_1>0$ the natural continuation map
\[
f_{s_0,s_1}:CF^{*,0}(H_{a,s_0})\to CF^{*,0}(H_{a,s_1})
\]
is a valuation preserving isomorphism of complexes.
\end{lm}
\begin{rem}
Note the continuation from $H_{a,s_0}$ to $H_{a,s_1}$ is \textit{not} strictly monotone. That is, we do not have the pointwise relation $H_{a,s_0}(x)\leq H_{a,s_1}(x)$ for all $x\in M,$ as opposed to the relation \eqref{eqRelation} which does hold. We emphasize that in this paper we are concerned with computing Floer theory over the Novikov \emph{field}, which is why the relation \eqref{eqRelation} is the relevant one. However, the proof of Lemma \ref{tmSH00} requires to momentarily study Floer homology over the Novikov ring. In general, continuation maps are defined over the Novikov ring only if there is strict monotonicity. Thus the claim of Lemma \ref{lmValPresSH0} is rather non-trivial.  
\end{rem}
\begin{proof}
Each family of periodic orbits is contained in a Lagrangian torus. In particular, the action
\[
\mathcal{A}_{H}(\gamma):=-\int_D\omega-\int_tH\circ\gamma
\]
is independent of the choice of periodic orbit within the family. We claim first the identification $\iota:\mathcal{P}_{0,a,s_0}=\mathcal{P}_{0,a,s_1}$ preserves the action. Indeed, let $\gamma\in \mathcal{P}_{0,a,s_0}$. Then $H_{a,s_0}|_\gamma$ is constant and preserved by $\iota$. Let $D$ be a disc filling $\gamma$. To obtain a disc filling $\iota(\gamma)$ we can use a homotopy which has $0$ flux since $\gamma$ is a contractible orbit. Thus
\[
\mathcal{A}_{H_{a,s_1}}(\iota\gamma)=\mathcal{A}_{H_{a,s_0}}(\gamma).
\]
On the other hand, when $|s_0-s_1|$ is small, we claim the continuation map of \eqref{eqContinuation} can be written as $\iota+O(t^\epsilon)$\footnote{Here, $O(t^{\epsilon})$ denotes an element of the Novikov ring whose valuation is at least $\epsilon$.} for some $\epsilon>0$. To see this, let $\rho:\R\to [s_0,s_1]$ be a monotone increasing function whose derivative is compactly supported.  Consider a homotopy $s\mapsto H_s:=H_{a,\rho(s)}$. This is a non-monotone homotopy.   However, the derivative $|\partial_s H_s|$ can be estimated by $\sim  |s_0-s_1|$. From this we deduce that for any solution $u$ contributing to the continuation map we have $E^{top}(u)>-|O(s_0-s_1)|$. We use this now to show that for each $s>0$ there is a $\delta>0$ such that if $s_0,s_1>s$ and $|s_0-s_1|<\epsilon$ there is no continuation trajectory with negative topological energy. First observe that $\omega$ is exact in the region $U_+=\{x_n>0\}$ and in fact $U_+$ is symplectomorphic to an open subset in $T^*T^n$. Fixing such a symplectomorphism we can thus take $\delta>0$ to be the minimal action difference of the periodic tori of $H_{a,s}$ with the action defined using the Liouville form on $T^*T^n$. Note it makes no difference which $s$ we choose by the flux consideration above. Moreover, for fixed $a$ there is a finite number of such tori, so $\delta$ can indeed be taken to be positive. Thus if we take $|s_0-s_1|<\delta$, a continuation trajectory with negative topological energy cannot be contained in $U_+$. By Gromov compactness, for fixed $s_0,s_1$ there is a minimum, still denoted by $\delta$, so that any continuation trajectory from $H_{a,s_0}$ to $H_{a,s_1}$ which reaches outside of $U_+$ must have geometric energy more than $\delta$. Moreover, by Corollary \ref{cyExplicitFormUplus}, this delta can be taken to be the same for all $s_0,s_1>s$. 

To complete the argument, note that $E^{geo} \leq E^{top}+O(s_0-s_1)$. So for $|s_0-s_1|<\delta$ there are no negative energy solutions. Moreover, the contribution from $0$ energy solutions is exactly the map $\iota$.  It follows that $|s_0-s_1|<\delta$ the continuation map is $\iota+O(t^\epsilon)$. In particular valuation preserving and invertible. For $|s_0-s_1|$ arbitrary we can subdivide the interval $[s_0,s_1]$ into segments of length $<\epsilon$ and write a continuation map which is a composition of a large but finite number of valuation preserving isomorphisms. 

\end{proof}

\begin{lm}\label{lmNoWallModE}
For any $E$ there is an $s_0$ such that for any $s>s_0$, the differential in $CF^*(H_{a,s})$ vanishes modulo $E.$ Namely, $\val(d\gamma)\leq \val(\gamma)-E.$
\end{lm}
\begin{proof}
First observe that there is a constant $c$ such that, in the terminology of definition \ref{dfCadmFl} bellow, for any fixed admissible $J$, the Floer datum $(J,H_{a,s})$ is $c$-admissible on $U_+:=H_B^{-1}(0,\infty)$ and $s$ is large enough. It follows by the target-local, domain-global estimate of Lemma \ref{lmTLDG} that for large enough $s_0$, any Floer trajectory of energy $E$ is contained in  $U_+$. But distinct periodic orbits represent distinct homotopy classes in $U_+$. Thus, the only Floer trajectories are those starting and ending from the same periodic orbit. The restriction of $\omega$ to $H_B^{-1}(0,\infty)$ is exact. So the only Floer trajectories starting from and ending on the same periodic orbit are the trivial ones. The claim now follows by standard Morse-Bott considerations.
\end{proof}
\begin{proof}[Proof of Lemma~\ref{tmSH00}]
This is an immediate consequence of the previous two lemmas.
\end{proof}
The vanishing of the differential on $CF^{0,>0}$ is easier.
\begin{lm}\label{lmSH0geq0}
For $H\in\mathcal{H}_{\mathcal{B}+}$ ($H\in\mathcal{H}_{\mathcal{B}-}$)  the differential vanishes on $CF^{0,>0}$ ($CF^{0,<0}$).
\end{lm}
\begin{proof}
For $\epsilon>0$ let $g_\epsilon:\R_+\to\R_+$ be a proper convex function which is equal to $0$ on $[0,\epsilon]$ and to $\id$ near infinity. Write
\[
H_{a,\epsilon,s}:=g_{\epsilon}\circ g_s\circ H_B+\sum f_i\circ H_{F_i}.
\]
The function $H_{a,\epsilon,s}$ is Morse-Bott non-degenerate for non-contractible periodic orbits. For any Floer trajectory $u$, the projection $\pi\circ u$ is holomorphic for the region $\{\re\pi\in[0,\epsilon]\}$. Thus no solution with asymptotes in $\{\re\pi>0\}$ can intersect it. So all Floer solutions are contained in $\{\re\pi>0\}$. It follows that the differential on $CF^*(H_{a,\epsilon,s})$ vanishes. To conclude we need to show that the continuation map $CF^*(H_{a,\epsilon,s})\to CF^*(H_{a,s})$ is an isomorphism of complexes. For this note that the continuation map induces an isomorphism on homology, since the asymptotics at infinity of both Hamiltonians is the same. Since both chain complexes have the same dimension as vector spaces over the Novikov ring, the vanishing of the differential on one complex thus implies it for the other complex.
\end{proof}
\begin{proof}[Proof of Lemma~\ref{lmFloerInt}]
The first part is Lemma~\ref{lmDegCont}. The second part is Lemmas~\ref{tmSH00} and~\ref{lmSH0geq0}.
\end{proof}

As a consequence of the preceding discussion, we have that for Hamiltonians in $H\in\mathcal{H}_{\mathcal{B}_+}$ with slope at infinity given by $a$, $CF^0(H)=HF^0(H)$ has a basis naturally labeled by a subset of $H_1(U_+;\Z)$ which we identify as above with $H_1(G;\Z)\oplus \Z\langle\gamma\rangle$. The following theorem claims that the continuation maps respect this labelling up to multiplication by a diagonal matrix with values in $\Lambda_{nov}$. Denote by $H_1^{\pm}(U_{\pm};\Z)$ the elements on which $\theta_3$ is non-negative and non-positive respectively.
\begin{tm}\label{lmChainSH}
There is a natural isomorphism
\[
SH^{0,0}(M;\Ll)=\Lambda_{nov}[H_1(G;\Z)],
\]
and an isomorphism
\[
SH^{0,\pm}(M;\Ll)\simeq\Lambda_{nov}[H_1^{\pm}(U_{\pm};\Z)]\simeq SH^{0,0}(M;\Ll)\otimes\Lambda_{nov}[u^{\pm\gamma}]
\]
which is natural up to scaling of the basis elements by scalars  in $\Lambda_{nov}$. Moreover, there is a cofinal collection of Hamiltonians $\mathcal{H}'_{\mathcal{B_\pm}}\subset\mathcal{H}_{\mathcal{B_\pm}}$ such that for any $\alpha\neq 0\in\Z\langle\gamma\rangle,\beta\in H_1(G;\Z)$, and $H\in\mathcal{H}'_{\mathcal{B_\pm}}$ with sufficiently large slopes (depending on $\alpha,\beta$ only) the class in $SH^{0,\pm}$ corresponding to $(\alpha,\beta)$ is given by an element $c\delta\in CF^*(H)$ where $c\in\Lambda_{nov}$ and $\gamma$ is a periodic orbit of $H$ representing $(\alpha,\beta)$ in $U_{\pm}$.
\end{tm}
\begin{proof}
Relying on Lemma \ref{lmValPresSH0}, consider for a Hamiltonian $H$ of slope $a$ the valued vector space
\[
\widetilde{CF}^{0,0}(H):=\varinjlim_{s\to\infty}CF^{0,0}(H_s).
\]
Then $\widetilde{CF}^{0,0}(H)$ has a natural basis labeled by elements of $H_1(G;\Z)$. For a pair $H_1\leq H_2$ the induced continuation map
\[
\widetilde{CF}^{0,0}(H_1)\to\widetilde{CF}^{0,0}(H_2)
\]
respects this labeling by the argument of \ref{lmNoWallModE} Namely, otherwise we would have a continuation trajectory of bounded energy reaching the critical locus of $\Ll$ and connecting orbits that arbitrarily far from the $\crit\Ll$. The first natural isomorphism proclaimed by the lemma is induced from the maps $CF^{0,0}(H)\to \widetilde{CF}^{0,0}(H)$.

For the rest of the claim consider Hamiltonians as in the proof of Lemma \ref{lmSH0geq0} and argue that there are no Floer trajectories connecting non-contractible periodic orbits and reaching the critical locus.


\end{proof}

\subsection{The Laurent polynomial}\label{subsecLaurent}
In the rest of this section we prove Theorem \ref{tmLaurent0}.

Recall the definition of the universal Novikov ring
\[
\Lambda_{nov}:=\left\{\sum_{i=0}^{\infty}a_it^{ r_i}:a_i\in \C, r_i\in\R,\lim_{i\to\infty} r_i=\infty\right\}.
\]
It carries a valuation defined by
\[
\val\left(\sum_{i=0}^{\infty}a_it^{ z_i}\right)=\min_{\{i:a_i\neq 0\}} r_i.
\]
A function $f:\R^n\to\Lambda_{nov}$ is said to be monomial if $f(x_1,\dots x_n)=ct^{\sum u_ix_i}$ for some $c\in\Lambda_{nov},u_i\in\Z$. We write this monomial as $c_uz^u$. For a precompact open set $U\subset\R^n$   define the valuation  on the space of monomials by
\[
\val_U\left(cz^u\right)=\inf_{x\in U}\left(\val{c}+\sum u_ix_i\right).
\]
This is the same as the logarithm of the $\sup$-norm with respect to the norm $e^{-\val}$.
A sum of monomials
\[
\sum_{u\in\Z^n}c_uz^u,
\]
is said to converge on $U$ if $\val_U{c_uz^u}\to\infty$. Note that a sum of monomials which converges on $U$ defines a function $U\to\Lambda_{nov}$. A function $f$ is said to be analytic on an open set $U$ if it is locally a convergent sum of monomials. The key property of analytic functions for us is
\begin{lm}
A formal sum of monomials converging on an open set $U$ and which vanishes identically as a function on an open set $V\subset U$ is the zero sum. In particular, the local expansion of an analytic function is unique. Moreover, the coefficients of the expansion are fixed on every component of the domain of definition.
\end{lm}
\begin{proof}
 Indeed, the leading term with respect to $\val_V$ is a polynomial. By choosing a generic point in $V$ and restricting to a smaller subset, the leading term is a monomial for which the claim is obvious.
\end{proof}

So far we have endowed subsets of $\mathbb{R}^n$ with the structure of a ringed space. For $U\subset\mathbb{R}^n$ let us denote the structure sheaf by $\mathcal{O}_{U}$. It is clear that this construction is functorial with respect to integral affine isomorphisms and thus can be stated as an assignment of a sheaf $\mathcal{O}_U$ to abstract integral affine manifolds $U$. In this formulation, after picking a basepoint $b\in U$, elements $u$ of the integral lattice in $T^*_bV$ give rise to local monomials $z^{u}$ which on picking integral affine coordinates centered at $u$ are given by $(x_1,\dots,x_n)\mapsto t^{\sum u_ix_i}$. 

By definition, the analytic functions of $V$ are the global sections of $\mathcal{O}_V$. In particular,  for any simply connected open set $V\subset B_{reg}$ the integral affine structure induces a ringed structure on $V$. Let $\mathcal{A}_V$ be the space of analytic functions on $V$. For any $b\in V$ we have an identification of the period lattice in $T^*V$ with $H_1(L_b;\mathbb{Z})$. We thus get a collection of monomial functions described explicitly as follows. For any $b\in V$ and any $\alpha\in H_1(L_b;\Z)$ let $f_{b,\alpha,V}$ be the affine linear function whose value at $x\in V$ is $\int_I\omega$ where $I$ is the cylinder traced by moving a representative of $\alpha$ along any path connecting $b$ and $x$. Denote the corresponding monomial by $ z_{b,\alpha,V}:=t^{f_{b,\alpha,V}}\in\mathcal{A}_V$. Note that $z_{b+c,\alpha,V}=t^{\langle\alpha,c\rangle}z_{b,\alpha,V}$.
In the following write $V_\pm\subset B_{reg}$ for the subsets $\Ll(U_\pm)$ where as before $U_+=\{H_B>0\}\subset M$ and $U_-=\{H_B<0\}\subset M$ respectively.

We note that for any base point $b$, the rings $\mathcal{A}_V$ are certain completions of the group algebra over the Novikov field of $H_1(L_b;\mathbb{Z})$. Picking a basis for  $H_1(L_b;\mathbb{Z})$ identifies them with a certain completion of the ring of Laurent polynomials over the Novikov field. 

In the following, for a Lagrangian $L$ we denote by $1_L$ the unit of $HW^*(L,L)$ and by $\mathcal{CO}_L$ the closed open map $SH^*(M)\to HW^*(L,L)$. For each $s\in SH^0(M)$ denote by $h_{\pm}(s)$ the unique (a-priori non-continuous) function $\R^n\to\Lambda_{nov}$ satisfying
\begin{equation}\label{eqCOunit}
\mathcal{CO}_{L_b}(s)=h_{\pm}(s)(b)1_{L_b},\quad\forall b\in V_{\pm}.
\end{equation}
Clearly, the map $s\mapsto h_{\pm}(s)$ defines a $\Lambda_{nov}$ algebra homomorphisms from $SH^0(M)$ to the functions $\R^3\to\Lambda_{nov}$.
\begin{tm}\label{tmCOAn}
For any $s$ the functions $h_{\pm}(s)$ are analytic. That is, $h_{\pm}\in \mathcal{A}_{V_\pm}.$
\end{tm}
\begin{proof}
Let $H\in \mathcal{H}_{\mathcal{L}}$ such that $s$ is represented by an element of $HF^0(H)$, still denoted by $s$, and assume for simplicity that $s$ is represented by a single periodic orbit $\gamma$  of $H$. The more general case follows by linearity. Fix some $b\in B$.  Choose an $H_b\geq H$ which has non-degenerate chords for $L_b$. For simplicity assume $CF^0(L,L;H_b)$ is generated by a single chord $x_e$ which is a critical point of $H_b$. This can be achieved by appropriate choice of $H_b$. Then the unit of $HF^*(L_b,L_b;H_b)$, is represented uniquely by the chain
\[
\sum_{u_e\in\mathcal{M}_e} t^{u^*_e\omega+ H_{b,t}(x_e)}x_e,
\]
where $\M_e$ is the moduli space of Floer solutions on $D_{1,0}$ with boundary on $L$ and asymptotic output on $x_e$. 
For a relative homotopy class $A\in\pi_2(M,L)$ let
\[
\M_e(A):=\{u\in\M_e|[u]=A\}.
\]
We can then rewrite the expression for the unit as
\begin{gather}
\notag 1_{L_b}=t^{H_b(x_e)}\sum_{A\in\pi_2(M,L)}\#\M_e(A)t^{\omega(A)}x_e\\
=:t^{{H_b}(x_e)} g_{e,H_b}x_e.\notag
\end{gather}

On the other hand, the closed open map is defined by closed open trajectories $u_{co}$ weighted by
\[
t^{E_{top}(u_{co})}=t^{-\int_0^1 H_t(\gamma(t))+ H_{b}(x_e)+u^*_{co}\omega}.
\]

Denote by $\pi(M,L;\gamma_0)$ the homotopy classes of once punctured discs with boundary on $L$ and puncture asymptotic to $\gamma_0$.  For any $A\in\pi(M,L,\gamma_0)$ denote by $\M_{CO}(A;H,H_b)$ the moduli space of closed open trajectories $u$ such that $[u]=A$. We omit $H$ and $H_b$ from the notation when there is no ambiguity. Let
\[
 g_{CO}(\gamma):=\sum_{A\in\pi(M,L;\gamma_0)}\#\M_{CO}(A)t^{\omega(A)}x_e.
\]
Then
\[
\mathcal{CO}(\gamma;H,H_b)=t^{-\int_0^1H_t(\gamma(t))dt+H_b(x_e)} g_{CO}(\gamma)x_e
\]
Thus,
\[
h_+(s)=t^{-\int_0^1H_t(\gamma(t))} g_{e,H_b}^{-1} g_{CO}(\gamma).
\]
Note that the numbers $ g_e, g_{CO}$ each depend on numerous choices. However, the ratio weighted appropriately is well defined independently of any choices (aside from the fixing of Landau-Ginzburg potential and the class $s\in SH^0(M;\Ll)$). Indeed, it can be expressed in terms of Floer theoretic invariants which depend only possibly on $J_{LG}$. Define an equivalence relation on $\pi(M,L_b;\gamma_0)$ as follows. For $b'\in V_+$ identify $\pi(M,L_b;\gamma_0)=\pi(M,L_{b'};\gamma_0)$ by concatenating cylinders contained in $V_+$. Then  $A\sim B$ iff $\partial A=\partial B$ and $\omega(A)=\omega(B)$ for all $b'\in V_+$. Then we can write
\begin{equation}\label{eqDefn_A}
 g_{e,H_b}^{-1} g_{CO,H,H_b}(\gamma)=\sum_{A\in\pi(M,L_b,\gamma_0)/\sim}n_{A,b}t^{\omega(A)}.
\end{equation}
We need to show that this expansion holds on a fixed neighborhood of $b$. That is, there is a neighborhood of $b$ independent of $A$ on which $n_{A,b}$ is independent of $b$. Assume the $(H,J)$ chosen for the Hamiltonian Floer homology is regular. For a fixed generic choice of Floer data on $D_{1,1}$ defining $\mathcal{CO}$ and coinciding with $(H,J)$ near the interior puncture, there is a neighborhood of $U_A$ of $b$ such that all the moduli spaces going into the computation of $n_A$ are smooth. In particular $n_A$ is constant over that neighborhood. On the other hand, the numbers $n_A$ are independent of any of these choices. To see this, let $n_A(F)$ correspond to the choice of Floer datum $F$. After slightly perturbing $b\in U_+$ so that $n_A(F_1)$ and $n_A(F_2)$ remain well defined, we may assume that $A$ is the unique class with value $\omega(A)$. So $n_A$ is the coefficient of $t^{\omega(A)+\mathcal{A}_H([\gamma:u_0])}$ in $h_{\pm}(s)$ which is independent of any choices. Since $n_A$ is well defined and locally constant, it is constant. We have thus presented $h_\pm(s)$ locally as a convergent sum of monomials.

\end{proof}
\begin{rem}\label{remIntNA}
Note that the numbers $n_A$ defined in equation \eqref{eqDefn_A} are in fact integers. For this observe first that the moduli spaces for the unit and for the open closed maps have fundamental cycles defined over the integers. Second the leading term in $g_{e,H_b}$ has $1$ as its coefficient, corresponding to the trivial disk mapping to $x_e$.  Thus the inverse of $g_{e,H_b}$ is a Taylor series with integer coefficients.
\end{rem}

Our aim now is to analyze the behavior of the invariants $n_A$ occurring in the analytic expansion of $h_\pm(s)$ as introduced in equation \eqref{eqDefn_A}. 

\begin{lm}\label{lmSame}
Let $x\in SH^{0,+}$ correspond under the isomorphism of Lemma \ref{lmChainSH} to an element $[\gamma]\in H_1(U_+;\Z)$.  Fix some base point $b_+\in V_+$. Then
\[
h_+(x)=t^{E}z_{[\gamma],b_+,V_+},
\]
For some real number $E$ depending on $x$. A corresponding claim holds for $SH^{0,-}$, $V_-$ and $h_-$.
\end{lm}
\begin{proof}
We continue with the notation of the proof of Theorem \ref{tmCOAn}.  Fix a class $A\in\pi(M,L_b,\gamma)$. $A$ determines a choice of class in $\pi(M,L_{b'}(\gamma))$ for each $b'\in V_+$ which we still denote by $A$. For simplicity take $x$ to be a non-contractible class, the contractible case requires some easy adjustment which we leave for the reader when we invoke Theorem \ref{lmChainSH}. Assume we have chosen a Hamiltonian $H\in \mathcal{H}'_{\mathcal{B}_+}$ as introduced in Lemma \ref{lmChainSH} and such that $\gamma$ occurs on $L_b$. Associated with $A$ is a coefficient $n_A$ in the expansion of $h_+(x)$ which is obtained by a sum of weighted configurations containing among other things a closed-open trajectory $u_{co}$ with $\gamma$ on the interior puncture and with boundary on $L_{b}$. The topological energy of $u_{co}$ is given by
\[
E_{top}(u_{co})=-\int_0^1 H_t(\gamma(t))+\int_0^1 H_{b,t}(x_e)+u^*_{co}\omega.
\]
Moreover, if we take $H_{b,t}>H_{t}$ we can choose non-negative interpolating data on $D_{1,1}$ so that we have
\[
\int_{D_{1,1}}\|du_{co}^{0,1}\|^2\leq E_{top}(u_{co}).
\]
We can further choose $|H_{b,t}-H_{t}|$ to be arbitrarily small. So the geometric energy is essentially controlled by the  term $u^*_{co}\omega$. Since $\partial (u^*_{co})\subset L_{b}$, we can deform $[u_{co}]$ into a connected sum of the loop $\gamma$ with some element $B\in\pi(M,L_{b})$ without changing the integral of $\omega$. Thus the geometric energy is essentially controlled by $\omega(B)$. If we apply a continuation map to another Hamiltonian $H'\in \mathcal{H}_{\mathcal{B_+}}
$ then by Lemma \ref{lmChainSH} $x$ is represented by $t^E\gamma'$ where $E$ is some real number and $\gamma'$ is a periodic orbit occurring on $L_{b'}$ for some other $b'\in V_+$ and such that $\gamma'$ corresponds to $\gamma$ under the identification of $L_b\sim L_{b'}$ by a path in $V_+$. The coefficient corresponding to $n_A$ can thus be represented as a count of configurations containing an open closed trajectory $u'_{co}$ occurring on $L_{b'}$ and homotopic to the concatenation of $\gamma'$ and a class $B'$. The class $B'$ is obtained from $B$ by concatenating the cylinder obtained by transporting $\gamma$ to $\gamma'$ along a path in $V_+$.  We can move $b\to\infty$ inside $V_+$ while keeping the coordinates $(H_{F_1},H_{F_2})$ constant. For such paths, the term $\omega(B')$ which controls the geometric energy remains constant. Finally, for any $R$, and for $b$ such that $d(b,W)>2R$  we may choose $H$ to so that it satisfies a bound of the form $d(\psi^H_1(x),x)>\epsilon$ for all $x\in \Ll^{-1}(B_R(b)\setminus B_{R/2}(b))$ while keeping $H$ uniformly Lipschitz and Lyapunov and without introducing new periodic orbits. For $R$ large enough, this guarantees by Lemma \ref{lmTLDG} that
\begin{equation}\label{eqDaumUCOEst}
u'_{co}\subset \Ll^{-1}(B_R(b)).
\end{equation}
We deduce from this, first, that $B$ must be constant. In particular,
\[
h_+(\alpha)=c z_{\gamma,b,V_+}.
\]
Second, we use this to show that $c=t^{-\int_0^1H_t(\gamma(t))}$. For this, embed $\Ll^{-1}(B_R(b))$ into $T^*G$ using action angle coordinates, extend the Floer data to all of $T^*G$ as in the proof of Lemma \ref{lmRoughform} while preserving all geometric bounds on the Floer data and without introducing new orbits. By \eqref{eqDaumUCOEst} it now suffices to show the corresponding claim about $c$ for $T^*G$. For this observe that the $\mathcal{CO}:SH^0(T^*G)\to HW^0(G)$, where $G$ is embedded as the $0$ section, an unital algebra map and that working  with $\Z$ coefficients (which we can, since the zero section is exact) $x$ is invertible in $SH^0(T^*G)$ over $\Z$.
\end{proof}

We now consider $y\in SH^{0,<0}$ and study the expansion of $h_+(y)$. We fix an $H\in\mathcal{H}_{\mathcal{B_-}}$ with large enough slope so that $y$ is  represented by the periodic orbit $\gamma$ for some $[\gamma_0]\in H_1(U_-;\Z)$ occurring at the fiber over some $b\in V_-$. Let $W=B_{reg}\setminus (V_+\cup V_-)$. Observe that  the map $f_{\mu}:B\to \mathfrak{g}^*$ restricts to an open and dense embedding of $W$ into  $\mathfrak{g}^*$. For each $\alpha\in\Delta^*$ let $W_\alpha$ be the corresponding component of $W\setminus\Delta$. Write $V_\alpha:=V_-\cup W_{\alpha}\cup V_+$.  Fix a point $b_-\in V_+$. Each $\alpha$ singles out a distinguished element $A_\alpha$ of $\pi(M,L_b;\gamma)$ by the requirement that it can be represented by a cylinder contained in $U_\alpha:=\Ll^{-1}(V_\alpha)$. Note that $\pi(M,L_b;\gamma)$ is a torsor over $\pi_2(M,L_b)$. Each $\alpha$ also singles out a semi-group $C_\alpha\subset\pi_2(M,L_b)$ consisting of elements $D$ for which, under the identification $\pi(M,L_{b'})=\pi_2(M,L_b)$ for any $b,b'\in V_\alpha$, we have that, if $D$ is non-trivial, $f_{b,\partial D,V_\alpha}(w)\to\infty$ as the distance of the point $w$ from the closed subset $\partial W_\alpha\subset W$ goes to $\infty$. Here we endow $W$ with a metric by identifying it with $\mathfrak{g}^*$ using the map $f_{\mu}$ and choosing any inner product on $\mathfrak{g}^*$. Note that $C_\alpha$ is invariant under the action of $\pi_2(M)$ on $\pi_2(M,L_b)$.
\begin{lm}\label{lmPositivCone}
Write $h_+(y)=\sum_{A\in \pi(M,L_b;\gamma_0)}c_At^{\omega(A)}z_{b,\partial A,V_+}.$
If $c_A\neq 0$ then
\[
A\in \cap_{\alpha\in\Delta^*}A_\alpha+C_\alpha.
\]
\end{lm}
\begin{proof}
The argument is very similar to that of Lemma \ref{lmSame}. This time, for each $b\in V_+$ near the wall $W_\alpha$, we choose an $H$ with representative $\gamma$ of $x$ which lies on a Lagrangian $L_{b_-}$ for some $b-\in V_-$ such that $b_-$ and $b$ have the same $(H_{F_1},H_{F_2})$ coordinates and $d(b_-,W)$ is uniformly bounded.  For  any closed open trajectory $u_{co}$, we can then deform the class  $[u_{co}]$ to the connected sum of a cylinder $C$ obtained by transporting $\gamma$ from $b_-$ to $b$ and some class $B\in\pi_2(M,L)$. As before, for fixed $A$, the class $B$ is fixed as we move $b$ around in $V_+$. If we keep $b$ a fixed distance to the wall $W$, the topological energy will be controlled  by $\omega(B)$.  Thus $B$ must have the property that $\omega(B)\to\infty$ as $d(b,\partial W)\to\infty$ or otherwise be trivial. The claim now follows as in Lemma \ref{lmSame}.
\end{proof}
\begin{lm}\label{lmGlobalConst}
For any $\alpha$ identify the classes $\pi(M,L_b,\gamma_0)$ on both sides of the wall by transport through $V_\alpha$. Then, under this identification, the coefficient $c_{A_\alpha}$ defined in Lemma \ref{lmPositivCone} is globally constant on $V_+\cup V_-$.
\end{lm}

\begin{proof}

For any $b\in W_\alpha$, we may choose $H$ so that $\gamma$ is so close to $L_b$ that the topological energy associated with the $B$ corresponding to $A_\alpha$ is less than the energy of the smallest disk with boundary on $b$. Thus the coefficient $c_{A_\alpha}$ is well defined on $W_\alpha$ and is locally constant on a neighborhood thereof. Since $c_{A_\alpha}$ is also locally constant on the complement of $W_\alpha$, it follows that it is globally constant on $V_\alpha$.
\end{proof}

\begin{cy}\label{cyNAAlpha1}
 For any $\alpha\in \Delta^*$ we have $c_{A_\alpha}=t^{E}$ for some number $E$ independent of $\alpha$.
\end{cy}
\begin{proof}
Combine Lemmas \ref{lmSame} and \ref{lmGlobalConst}
\end{proof}

We proceed compute the intersection
\[
\cap_{\alpha\in\Delta^*}(A_\alpha+C_\alpha^+).
\]
We have a map $\partial: \pi_2(M,L)\to H_1(M;\Z)$ defined by taking the boundary. We have that the image of $\partial$ is $\Gamma\subset H_1(M;\Z)$. For the torus $G$ we have the isomorphism $i:\Gamma\to \mathfrak{g}_\Z$. Denote by  $\mathfrak{g}_\alpha\subset\mathfrak{g}$ the cone generated over $\N\cup\{0\}$ by the vectors from $\alpha$ to its adjacent vertices.
\begin{lm}\label{lmIntCones}
There is an $A\in \pi(M,L_b;\gamma_0)$ such that
\[
\cap_{\alpha\in\Delta^*}(A_\alpha+C_\alpha)=A+(i\circ\partial)^{-1}\Delta^*.
\]
\end{lm}
In other words, the support of $h_+(y)$ is a shift of $\Delta^*$ under obvious identifications.

First we prove
\begin{lm}
We have $A_\alpha+C_\alpha=A+(i\circ\partial)^{-1}(\alpha+\mathfrak{g}_\alpha)$.
\end{lm}
\begin{proof}
Unwinding definitions, the function $f_{b,\partial D,V_\alpha}(w)$ is, up to a constant, just the evaluation of $f_\mu(w)$ on  $i\circ\partial D$.
So, by Lemma \ref{lmDualGraph} we have $i\circ\partial C_\alpha=\mathfrak{g}_\alpha$. Moreover, Lemma \ref{lmDualGraph} implies that for an appropriate $A$ independent of $\alpha$, we have  $i\circ(\partial A_\alpha-\partial A)=\alpha$. To see this note that $A_\alpha$ and $A_\beta$ are related by the monodromy. The claim of the lemma now follows by the invariance of $C_\alpha$ under the action of $\ker\partial =\pi_2(M)$.
\end{proof}

\begin{proof}[Proof of Lemma \ref{lmIntCones}]
Let $\Delta^*_0$ be the vertices $\alpha$ for which the cone emanating from $\alpha$ is strictly convex. Then by convexity we have
\[
\Delta^*=\cap_{\alpha\in\Delta^*_0}(\alpha+C_\alpha).
\]
Intersecting further by the cones which are non strictly convex has no effect. Indeed for interior vertices the cone is the entire plane. For vertices in the interior of the edges, the cone is a half plane bordered by the given edge.
\end{proof}

Before proceeding to the proof we recap what we have established so far. 
\begin{itemize}
\item We have constructed, using the closed open map, a pair of injective homomorphisms $h_{\pm}$ from the unknown algebra $SH^0(M;\mathcal{L})$ to the algebras $\mathcal{A}_{V_{\pm}}$. The latter algebras are known. They are each isomorphic to some completion of the algebra of Laurent polynomials over the Novikov ring. A choice of isomorphism is induced by picking a basis for $H_1(L_{b_{\pm}};\mathbb{Z})$ for $b_{\pm}\in V_{\pm}$ any basepoint.
\item While we have not directly computed the homomorphisms $h_{\pm}$ on all of $SH^0(M;\mathcal{L})$ we have computed their restrictions to $SH^{0,\pm}$ respectively and found these to be the inclusion of the polynomial algebra into the Laurent polynomials. In particular, given generators $x,y$ for $SH^{0,+}$ and $SH^{0,-}$ respectively over $SH^{0,0}$, the key to finding the Laurent polynomial $xy$ is to compute either $h_-(x)$ or $h_+(y)$. Indeed,  we will be able to determine $xy$ from the equation $h_+(xy)= h_+(x)h_+(y). $
\item
Within $H_1(L_b;\mathbb{Z})$ there is a distinguished rank $2$ submodule corresponding to elements which are contractible in $M$. This submodule is canonically identified with the integral lattice in $\mathfrak{g}$. Under this identification, we find that the newton polygon of the polynomial $h_-(x)$ is, up to a shift, given by the dual graph $\Delta^*$.  To understand this note that the vertices of $\Delta^*$ are in bijection with components of the wall $W$. Moreover, for each component $W_\alpha$ of the wall we obtain an identification 
\[
i_{\alpha}:H_1(L_{b_-};\mathbb{Z})=H_1(L_{b_+};\mathbb{Z})
\]
by parallel transport across $W_\alpha$. This induces for each $\alpha$ a separate identification between the algebras $\mathcal{A}_{V_{+}}$ and $\mathcal{A}_{V_{-}}$ and in particular, a separate way to compare $h_+(x)$ and $h_-(x)$. From this we deduce in the proof below the claim concerning the Newton polygon:  On the one hand, for each component of the wall we get a non-zero monomial in the expansion of $h_-(x)$ by constancy of the leading term across each wall. See Lemmas \ref{lmSame} and \ref{lmGlobalConst} for the precise argument. On the other hand  Lemma \ref{lmIntCones} establishes that there can be no other monomials in the expansion. Namely, for any element $\gamma\in H_1(L_{b};\mathbb{Z})$ which  is contractible in $M$ and does not lie on  the given polygon, we can make the distance of $b$ to the discriminant locus arbitrarily large while the flux of $\gamma$ is kept constant. But by energy diameter considerations there can be no such contribution to the closed open map.

\end{itemize}

\begin{proof}[Proof of Theorem~\ref{tmLaurent0}]
Pick any generator $x$ of $SH^{0,+}$ over $SH^{0,0}$. Using a Hamiltonian $H_+\in\mathcal{H}_{+}$, there is a periodic orbit $\gamma$ which generates $H_1(M;\Z)$ such that $x$ is represented by $t^{E_+}\langle\gamma\rangle$ for some constant $E_+$. 
By Lemma \ref{lmSame}, there is a $b_+\in V_+$ for which without loss of generality
\[
h_+(x)=t^{E_+} z_{b_+,[\gamma],V_+}.
\]
Pick $y$ as follows. Consider $\Delta^*$ with a choice of a distinguished vertex $\alpha_0$ corresponding to a component $W_{\alpha_0}$ of $W$. Pick a $b_-\in V_-$ and a Hamiltonian $H_-$ generating a periodic orbit $\delta$ on $b_-$ such that $[\delta]=-[\gamma]$ as classes in $H_1(\Ll^{-1}(V_{W_{\alpha_0}});\Z)$. Let $y$ be the corresponding element in $SH^-$ such that
\[
h_-(y)=t^{E_-} z_{b_-,[\delta],V_-}.
\]
By the last three Lemmas,
\[
h_+(y)=t^{E_-}\sum_{\alpha\in\Delta^*}\sum_{A\in\pi(M,L_{b_+},\delta)|[\partial A]=\alpha+[\delta]\in H_1(L_{b_+};\Z)}n_At^{\omega(A)},
\]
where the $n_A$ are introduced in equation \eqref{eqDefn_A}, and are integers. See remark \ref{remIntNA}. We can write $A=A_{W_{\alpha}}\# B_A$ for some $B_A\in H_2(M;\Z)$. Then $\omega(B_A)\geq0$ by the proof of Lemma \ref{lmGlobalConst}. Moreover, We have
\[
t^{\omega(A)}=t^{\omega(B_A)} z_{b_-,[\delta],V_{\alpha}},
\]
and,
\[
 z_{b_-,[\delta],V_{\alpha}}=t^{f(\alpha)} z_{b_-,[\delta]+\alpha,V_{\alpha_0}},
\]
where $f(\alpha)=\langle \alpha,v_{b,\alpha}\rangle$. Here, $v_{b,\alpha}\in \mathfrak{g}^*$ is any vector from the projection of $b_-$  onto $\mathfrak{g}^*$ to the edge $e$ of $\Delta$ annihilated by $\alpha$.

Thus we can rewrite
\[
h_+(y)=t^{E_-}\sum_{\alpha\in\Delta^*}t^{f(\alpha)}\sum_{B\in H_2(M;\Z)|\omega(B)\geq0}n_{A_{W_\alpha}\#B}t^{\omega(B)} z_{b_-,[\delta]+\alpha,V_{\alpha_0}}.
\]
After appropriately rescaling $x$ and $y$, we get
\[
h_+(y)=\sum_{\alpha\in\Delta^*}t^{f(\alpha)}\sum_{B\in H_2(M;\Z)|\omega(B)\geq0}n_{A_{W_\alpha}\#B}t^{\omega(B)} z_{b_+,\alpha,V_{\alpha_0}}.
\]
By Lemma \ref{lmSame} we can find generators $u_1^{\pm},u_2^{\pm}$ of $SH^{0,0}$ such that $h_+(u_i)= z_{b_+,\gamma_i,V_+}$ for $\gamma_i$ generating the kernel of $H_1(L_b;\Z)\to H_1(L_b;M)$. Thus, we have
\[
h_+(xy)=\sum_{\alpha\in\Delta^*}t^{f(\alpha)}\sum_{B\in H_2(M;\Z)|\omega(B)\geq0}n_{A_{W_\alpha}\#B}t^{\omega(B)}h_+(u^{\alpha}).
\]
Note that by Corollary \ref{cyNAAlpha1} the coefficients $n_{A_{W_\alpha}}$ are equal to $1$.
Note that $f(\alpha)$ is well defined up multiplying by a function of the form $t^{\langle\alpha,v\rangle}$ for an arbitrary vector $v\in\mathfrak{g}^*$. This ambiguity can be absorbed by a rescaling of $u_i^{\pm}$.
Since $h_+$ is injective, the claim follows.
\end{proof}
\subsection{Worked example: the focus-focus singularity}\label{SubSecWorkedExample}

We illustrate the proof of Theorem \ref{tmLaurent0} in the case of $n=2$ with a single critical point of focus-focus type. This case has been treated by Pascaleff
\cite{Pasc1} using a different method.

\includegraphics{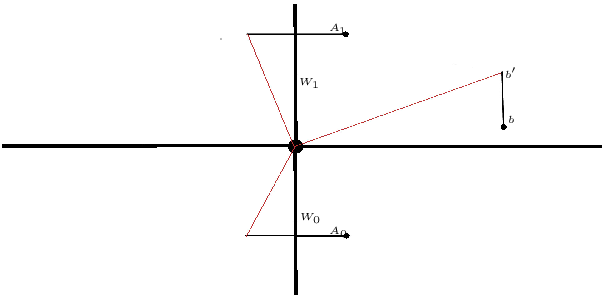}\label{FigWall}

In this case, $B=\R^2$ as depicted in Figure\ref{FigWall}. There is a single critical point. The wall consists of two components $W_0$ and $W_{1}$ emanating from the critical value $b_0$ and transverse to the horizontal lines. Note that the walls are not necessarily affine straight lines. The horizontal axis is monodromy invariant and the lines parallel to it are affine. In fact the horizontal lines are characterized as those along which the flux of the homology classes of $L$ that are contractible in $M$ are constant. The flux of a non-contractible generator locally defines a transverse affine coordinate, but it is not globally defined due to the monodromy around $b_0$.

According to Theorems \ref{lmRoughform} and \ref{tmFloerInt} we have the following description of $SH^0(M)$. The part generated by contractible orbits, $SH^{0,0}$, is the ring of Laurent polynomials over $\Lambda$ with a basis of monomials $u^i,i\in\Z$.  We then have generators $x,y$ corresponding to the two generators of the homology of $M$. For elementary topological reasons it follows that $SH^0=\Lambda[x,y,u,u^{-1}]/xy-g(u)$ where $g$ is some Laurent polynomial. Our goal is to show that, for appropriately chosen generators, we have
\[
g=1+u.
\]
We will deduce this relation from the homomorphisms
\[
h_{\pm}:SH^0(M;\mathcal{L})\to\mathcal{A}_{V_\pm},
\]
defined by equation \eqref{eqCOunit}. The Floer homology is defined with respect to a $J$ for which $b$ is not on a wall. With this choice, the closed open map is defined without choosing a bounding cochain.

On each side of the wall we have the affine coordinates $(b_1,b_2)$ and the corresponding generators $(z_1,z_2)=(t^{b_1},t^{b_2})$ of $\mathcal{A}_{V_\pm}$. The coordinate $b_2$ is well defined globally. We choose $b_1$ so that it is continuous across the upper wall $W_1$. The affine continuation of $b_1$ across the lower wall gives the coordinate $b_1+b_2$ corresponding to the wall crossing transformation $z_1\mapsto z_1z_2$. Pick a non-contractible periodic orbit $\gamma$ whose flux on $V_+$ measures the coordinate $z_1$ up to a constant and denote the corresponding element of $SH^{0,+}$ by $x$. Suppose $x$ is supported on a fiber $L_b$ and that $x$ via the closed open map is evaluated on a nearby fiber $L_{b'}$. Note that the open closed trajectory can be deformed to a connected sum of the cylinder  transporting the homology class of $\gamma$ from $b$ to $b'$ and some disc with boundary on $L_b'$. We represent the cylinder by a line from $b$ to $b'$ and the disc by a line from $b'$ to the origin. Now observe that the symplectic energy is roughly the sum of $d(b,b')$ and the flux of the disc. Therefore, keeping the distance between $b$ and $b'$ fixed while moving $b'$ horizontally to infinity the energy stays constant. Thus, by the energy-diameter estimate, Lemma \ref{lmTLDG}, there can only be the solution corresponding to a segment from $b$ and $b'$. Such a solution necessarily exists by comparison to the case of $T\T^{2*}$. Thus up to a constant, the element $x$ evaluates on $V_+$ as the coordinate $z_1$. Similarly $y$ can be taken to evaluate as the coordinate $z_1^{-1}$ on $V_-$.

Now evaluate $x$ on a $b'\in V_-$. Because of wall crossing, we expect there to be additional closed open solutions. This time, if we wish to keep $d(b,b')$ fixed, we can only move vertically. Consider $b,b'$ in the upper half plane. In the notation of the previous section, $A_1$ in figure \ref{FigWall} corresponds to the segment from $b$ to $b'$. By moving vertically upwards and applying the same energy-diameter argument, we see that the only possible contributions to $x$ at $b$ come from gluing a non-negative multiple of a disk to $A_1$. Let $C_1\subset\pi(M,L_{b'})$ be the cone of discs on which $\omega$ is non-negative. The preceding discussion can be summarized by the statement that the wall $W_1$ imposes the constraint that $x$ act by a polynomial corresponding to an element of $A_1+C_1$. In the coordinates $z_1,z_2$ this is the same as the polynomials generated by $z_1z_2^i,i\geq 0$. Repeating the same reasoning for the lower half plane we find that it must be in $A_0+C_0=A_0-C_1$. Since $A_0$ crosses the negative $y$ axis, there is a monodromy correction. The monomial corresponding to it is $z_1z_2$. Thus we obtain the additional constraint that $x$ evaluates to a polynomial generated by $z_1z_2^{1-i},i\leq 0$. Combining the two conditions, we find that $x$ on $V_-$ is generated by $z_1,z_1z_2$. Moreover, by moving $b,b'$ to $\infty$ (resp. $-\infty$) in the vertical direction, we see that $A_1$ (resp. $A_0$) has a unique local solution, since an arbitrarily large neighborhood is symplectomorphic to an arbitrarily large neighborhood of the zero section in $T^*\T^2$.  Thus = $xy=1+z_2$ on $V_-$. Let $u$ be given by the periodic orbit measuring the flux $z_2$. Then $u$ evaluates to $z_2$ on $V_-$ since $\gamma$ can be taken to occur in $V_-$. We thus obtain the desired relation $xy=1+u$.

\appendix
\section{Tame LG potentials}\label{AppA}
The aim of this appendix is to prove Theorem \ref{tmExLGPot}
\begin{df}
A Landau-Ginzburg function is a function of the form $\pi=e^{H+i\theta}:M\to \C$ together with a diffeomorphism  $\psi_{LG}$ from a  neighborhood $U_0$ of $\crit(\Ll)$ to a  neighborhood $U_1$ of the skeleton in $X_M$ from Theorem \ref{tmNormalTCY} such that
\begin{enumerate}
\item $H=h\circ\Ll$ and is independent of the functions $H_{F_i}$ away from $\Ll^{-1}(\Delta)\subset B$.
\item $\theta:M\to S^1$ is invariant under the action of $G$ and satisfies that $d\theta$ generates $H^1(M;\Z)$ and that $\theta(\Ll^{-1}(\Delta))=Const$.
\item We have $\pi_M|_{U_0}=\pi_X\circ\psi$ and such that $\psi^*J_X$ is tamed by $\omega$.
\end{enumerate}
\end{df}
\begin{lm}\label{lmTameComp}
Let $V$ be a symplectic vector space and let $W\subset V$ be a symplectic hyperplane. Then the set of all complex structures on $V$ which are tamed by $\omega$ and which preserve $W$ is contractible.
\end{lm}
\begin{proof}
Let $J$ preserve $W$ and let $x\neq0\in W^\omega$. Then $Jx$ splits as $Jx=y_h+y_v$ where $y_h\in W^\omega$ and $y_v\in W$. We claim that $J$ is tamed by $\omega$ if and only if for all $\theta\in[0,2\pi]$ we have
\begin{gather}\label{eqTameComp}
\omega (w,\cos\theta y_v-\sin\theta Jy_v)^2<4\omega(w,Jw)\omega(x,y_h), \\
\omega(w,Jw)>0,\notag\\
\omega(x,y_h)>0,\notag
\end{gather}
for all $w\in W$. To see this, note that tameness is equivalent to the inequality
\begin{equation}\label{eqTameCont11}
\omega(y'_v,w)<\omega(w,Jw)+\omega(x',y'_h)
\end{equation}
for all $x'\in W^\omega$ and $w\in W$. Note that the inequality implies that the right hand side is always positive. Therefore, the left hand side may also be assumed positive by replacing $x'$ with $-x'$ if necessary. Note that $y'_h,y'_v$ are functions of $x'$. Fixing $x'=x$, considering the degree of homogeneity  on each side of \eqref{eqTameCont11} as a function of $w$ we obtain the equivalent inequalities in
\eqref{eqTameComp} with $\theta=0$. Allowing $x'$ to vary is equivalent to taking $x'=\cos\theta x+\sin\theta y_h$for $\theta\in[0,2\pi]$ which gives \eqref{eqTameComp} by a straightforward computation.
Thus, fixing an $x\neq0$, an almost complex structure as required is determined by a choice of tame almost complex structures on $W,W^\omega$ respectively, and adding a mixed term determined by the projection $y_v$ of $Jx$ to $W$ in such a way that \eqref{eqTameComp} is satisfied. The latter is invariant under $y_v\mapsto ty_v$ for $t\in[0,1]$. The condition is thus contractible.
\end{proof}
\begin{lm}
For any Landau Ginzburg function $\pi$ there is an $\omega$-tame almost complex structure $J$ such that $\pi$ is $J$-holomorphic.
\end{lm}
\begin{proof}
We use the map $\psi_{LG}$ to pull back a on $\omega$-tame complex structure on a neighborhood $U$ of $\crit(\Ll)$ so that $\pi$ is $J$-holomorphic on $U$. Now observe that on  $M\setminus \crit(\Ll)$ we have that $\pi$ is a symplectic fibration. Indeed, in local action angle coordinates we have \[
\omega =\sum_{i=1}^{n-1}dH_{F_i}\wedge d\theta_i+dH_n\wedge d\theta_n.
\]
Since $dH$ is independent of the $dH_{F_i}$ and $\theta$ is invariant under the flows of $X_{H_{F_i}}=\frac{\partial}{\partial\theta_i}$ the claim follows.

The set of all $\omega$-tame fiber-wise almost complex structures is non-empty and contractible. By Lemma~\ref{lmTameComp} each one has a contractible set of extensions to an almost complex structure for which $\pi$ is $J$-holomorphic. We can thus extend $J$ from $U$ to all of $M$ so that $\pi$ remains $J$-holomorphic.
\end{proof}

\begin{lm}\label{lmTameStNoFo} Consider the Gross fibration $\Ll_{X_M}$ of Theorem \ref{tmNormalTCY}. The symplectic form $\omega$ and the fibration $\Ll_{X_M}$ on an open normal neighborhood $U_X$ of $Crit(\Ll_{X_M})\subset X_M$ can be deformed to a symplectic form $\omega_X$ and Lagrangian fibration $\Ll_X$ on $U_X$ taming the standard complex structure $J$ on $U_X$ and satisfying the following.
\begin{enumerate}
\item There is a compact set $K\subset \overline{U}_X$ for which $(U_X\setminus K,\omega_X)$ is the product symplectic form of a half infinite cylinder with a disjoint union of symplectic $4$-balls.
\item    $\Ll_X$, $J$ and the LG potential $\pi_X$ are all invariant under the symplectic action of the half infinite cylinder of the previous part on $U_X\setminus K$.
\end{enumerate}
The complex structure $J$ may be deformed near the boundary of $U_X$ to an $\omega_X$-tame almost complex structure $J'$ which near $\partial U_X$ is compatible with $\omega_X$ and such that $\pi$ is still $J'$-holomorphic and $J'$ is still invariant under the action of the half cylinder.
\end{lm}
\begin{proof}
Let $K\subset X_M$ be a precompact open set containing the all the compact components obtained by the intersection of any two distinct toric divisors. Denote the standard K\"ahler form restricted to $K$ by $\omega_0$. We obtain $\omega_X$ by extending $\omega_0$ to a neighborhood of the pre-image of the entire 1-skeleton as follows. Fix a half infinite edge $e$ of the skeleton of $X$ and let $p\in \Ll^{-1}(e)\cap K$. Pick a transverse circle action at $e$ generated by a Hamiltonian $H_F$ and let $c=H_F(p)$. Construct a normal form on a neighborhood $U_p$ of $p$. This give rise to a symplectomorphism
\[
\psi_{symp}:U_p\to S\times (-\epsilon,\epsilon)\times S^1,
\]
such that $\Ll|_{U_p}$ is a product of the standard focus-focus singularity on $S$ with the projection to $(-\epsilon,\epsilon)$. Let $X_{symp}$ be the vector field generating translations of the interval.  Note that the flow of $X_{symp}$ preserves $\Ll_X$ and we  have $dH_F(X_{symp})=1$.


Consider the vector field $JX_F$. It generates the real part of the $\C^*$ action whose imaginary part is $X_F$. It therefore commutes with the other actions, and moreover, it preserves the global function $\pi_X$. Thus $JX_F$ preserves the $G$ action. Moreover, we have that $dH_F(JX_F)$ is constant along any orbit of $X_F$. After multiplying by a constant  (which we assume to be $1$), it satisfies $dH_F(JX_F)=1$ along the orbit through $p$. Let $U_X$ be the non-negatve time  flow of $S$ under $X_F$ and $JX_F$. Then we have a diffeomorphism
\[
\psi_{hol}:U_X\to S\times\R_+\times S^1.
\]
Note that along the $X_F$ orbit through the critical point $p$ we have that $JX_F$ and $X_{symp}$ coincide. Thus, taking $S$ and $\epsilon$ small enough we have that $\psi_{hol}$ and $\psi_{symp}$ are arbitrarily close in $C^1$.

Let $g:\R\to\R$ be a smooth function such that $g'$ vanishes on the complement of $(c-\epsilon/2, c+\epsilon/2)$, $g=0$ near $-\infty$ and $g=1$ near $+\infty$. Consider the vector field
\[
X_{hyb}:=(1-g\circ H_F)X_{symp}+(g\circ H_F)JX_F.
\]
Then $X_{hyb}$ preserves the $G$ action. Indeed, both $X_{symp}$ and $JX_F$ preserve the $G$ action. Since $g\circ H_F$ also satisfies this property, so does $X_{hyb}$.

Let
\[
\psi:U_p\to S\times (-\epsilon,\epsilon)\times S^1
\]
be the diffeomorphism induced by the flows of $X_{hyb}$ and $X_F$. Then $\psi$ is also $C^1$ arbitrarily close to $\psi_{symp}$. Define a symplectic form on $U_X$ as
\[
\omega_X=\psi^*p_1^*\omega_0|_S+\psi^*(ds\wedge d\theta).
\]
Then $\omega_X$ smoothly extends $\omega_0$. Moreover, it tames $J$ along $s\in (c-\epsilon,c+\epsilon)$ since it is $C^1$ close to $\omega_0$ there. Since $X_{hyb}$ preserves both $J$ and $\omega_X$ this tameness holds everywhere. Extend $\Ll$ smoothly to $U_X$ by requiring it to be invariant under the action of $X_{hyb}$. Note that $\Ll$ and $\pi_X$ are invariant under $G$. Thus $\pi_X$ is an LG potential with respect to $J$.

Since the space of tame almost complex structures deformation retracts to that of compatible ones and since $\pi_X$ is a symplectic fibration away from the critical points, and, moreover, all our constructions are translation invariant at infinity, the final part of the claim is clear.



\end{proof}

In the following, fix $U_X$ and $\omega_X$ as in Lemma \ref{lmTameStNoFo} once and for all for each toric Calabi-Yau 3-fold $X$. Let $\Ll:M\to\R^3$ be a tame SYZ fibration and fix a Lagrangian section $\sigma$ which is disjoint of the critical points of $\Ll$.

\begin{df}\label{dfLGIneq}

We say that an LG function $\pi=e^{H+i\theta}$ is \textit{admissible} if the following are satisfied.
\begin{enumerate}
\item\label{dfLGIneqA} $\psi_{LG}$ is a bi-Lipschitz equivalence between a $(\sigma,\Ll)$-tame local normal neighborhood of $\crit(\Ll)$ and the metric on $U_{X_M}$ determined by $\omega_{X_M}$ and $J_{X_M}$ of Lemma \ref{lmTameStNoFo}.  Moreover, $\psi_{LG}^*J_{X_M}$ is uniformly tamed by $\omega_M$.
\item \label{dfLGIneqB}  There is an $\epsilon>0$ such that with respect to some
    $(\sigma,\Ll)$-adapted metric, for any $x\in M$ we have
\begin{equation}\label{eqAngle1}
\frac{\langle \nabla H,\nabla H_{F_i}\rangle}{\|\nabla H\|\|\nabla H_{F_i}\|}<1-\epsilon.
\end{equation}
\item \label{dfLGIneqC}For a $(\sigma,\Ll)$-adapted metric, $\pi$ is a quasi-conformal map. This means there is a constant $C$ such that for any $p\in M\setminus\crit(\Ll)$ and any pair of mutually orthogonal vectors $v,w\in\ker d\pi^{\perp}_p$ we have
    \[
    \frac1C<\frac{|d\pi(v)|}{|d\pi(w)|}< C.
    \]
\end{enumerate}
\end{df}
\begin{rem}
Note that if $\pi$ is an LG function satisfying the conditions with respect to one $(\Ll,\sigma)$-tame $J$ it satisfies them with respect to any such.
\end{rem}
\begin{lm}\label{lmAdmLGEx}
Suppose $\pi:M\to \C$ is an admissible LG function. Then there exists a $(\sigma,\Ll)$-tame almost complex structure $J$ on $M$ with respect to which $\pi$ is $J$-holomorphic.
\end{lm}
\begin{proof}

We show this first on $M\setminus \crit(\Ll)$. Suppose $\pi$ satisfies  conditions  \ref{dfLGIneqB} and \ref{dfLGIneqC} of Definition \ref{dfLGIneq} with respect to some compatible almost complex structure $J$ which is $(\sigma,\Ll)$-adapted. We construct a new almost complex structure $J'$ as follows. Consider the splitting $TM=V\oplus H$ with respect to the symplectic fibration $\pi :M\setminus \crit(\Ll)\to\C$. Let $p_H,p_V$ denote the horizontal and vertical projections respectively. Pick a global orthonormal frame $\{v_1,\dots,v_{n-1}\}$ of the tangent fibers to the torus orbits. Define $J'$ fiberwise by
\[
J'v_i:=p_VJv_i+\sum_{j<i}\omega(p_HJv_i,p_HJv_j)v_j,
\]
and extend uniquely by the condition $d\pi\circ J'=J'\circ d\pi$. Then $J'$ is $\omega$-compatible and $\pi$ is $J'$ holomorphic. Moreover, equation \eqref{eqAngle1} and the quasi-conformality guarantee that $J'$ is equivalent to $J$. Namely, equation \eqref{eqAngle1} bounds how far $H$ and $V$ are from being orthogonal. This in particular guarantees equivalence of the induced  metrics on $V$.  The quasi-conformality guarantees equivalence of the induced metrics on $H$. Taken together we get equivalence of the total metrics.

Observe now that if we find an $\omega$-tame almost complex structure $J''$ on a neighborhood of $\crit(\Ll)$ which is equivalent to $J$ and satisfies the inequality \eqref{eqAngle1} and the quasi-conformality condition we can interpolate it in the space of $\omega$-tame almost complex structures with $J'$  while preserving the equivalence with $J$. But these conditions are necessary. Indeed, suppose $J_0$ is equivalent to $J$ and $\pi$ is $J_0$ holomorphic. Since $\pi$ is equivariant with respect to the Hamiltonian flow of $H_{F_i}$ we have that $\nabla^{J_0} H\perp X_{ H_{F_i}}$ and therefore $\nabla^{J_0} H\perp J_0X_{ H_{F_i}}$. Since $J$ is equivalent to $J_0$, this implies equation \eqref{eqAngle1}. The quasi-conformality condition follows in the same way. Taking $J_0=\psi_{LG}^*J_X$ the claim follows by part \ref{dfLGIneqA} of Definition \ref{dfLGIneq} .
\end{proof}

\begin{proof}[Proof of Theorem \ref{tmExLGPot}]
We fix a tame local normal form and equivalence class of metrics as guaranteed by Theorem~\ref{tmCanonicalMetric}.
We use it to pull back $\pi$ and $J_U$ from $U_X$ as in Lemma \ref{lmTameStNoFo} to a small uniform neighborhood $U'$ of $\crit(\Ll)$. This defines $\psi_{LG}$. Note that it satisfies Condition \ref{dfLGIneqA} of Definition \ref{dfLGIneq}. Furthermore, since $\pi$ is $J_U$ holomorphic,  the other inequalities of Definition \ref{dfLGIneq}
are satisfied on the current domain of definition of $\pi$.

Our goal is to extend $\pi$ to all of $M$ in such a way that the inequalities of Definition \ref{dfLGIneq}
are satisfied. 
The function $H=\re\log\pi$ factors  as $h\circ\Ll$ for some function $h:\Ll(U')\to\R$. Identifying $B$ with $\R^n$ via coordinates as in Theorem \ref{tmBaseCoord} and relying on Theorem \ref{tmCanonicalMetric} we find that inequality \eqref{eqAngle1} translates into the inequality
\begin{equation}\label{eqAngle4}
\frac{\sum_{i=1}^{n-1}({\partial_ih)}^2}{(\partial_nh)^2}<C.
\end{equation}
By the above, this inequality holds on a neighborhood of $\Delta$.
Similarly, we have $dh\neq 0$ everywhere in a neighborhood of $\Delta$. To see this, observe that $\pi_U$ is given by \eqref{eqStLG} and that on the complement of the ends, $\Ll_U$ is arbitrarily close to the Gross fibration given explicitly in Theorem \ref{tmStGross} and that on the ends  Lemma \ref{lmTameStNoFo} guarantees translation invariance. This implies the stronger statement
\begin{equation}\label{eqNabBo}
\frac1c<|\nabla h|<c,
\end{equation}
for some $c>1$.

We claim that we can extend any function satisfying the inequalities \eqref{eqAngle4} and \eqref{eqNabBo} from a neighborhood of $\Delta$ to all of $\R^3$ such that both these inequalities are preserved. We may assume without loss of generality that there is a neighborhood $W_\Delta\subset\{x_n=0\}$ such that $h$ is initially defined on $W_\Delta\times (-\epsilon,\epsilon)$. Using an appropriate Urysohn function $f:\R^{n-1}\to\R$ with bounded derivative supported on a neighborhood of $\partial W_{\Delta}$ let $h=(1-f)\tilde{h}+fx_3$ on $\R^{n-1}\times(-\epsilon,\epsilon)$. The resulting function satisfies both inequalities and can now be extended in the same way to all of $\R^n$.

We now proceed to define $\theta$ so that the resulting map $\pi$ is a quasi-conformal submersion. For this consider the map $\Ll\times\theta_0:M\to B\times S^1$ given in Theorem \ref{tmGoodCoord} identified with $\R^n\times S^1$ so that in these coordinates $\Ll\times\theta_0$ is a quasi-Riemannian submersion away from a uniform tubular neighborhood $U_\delta$ of $\crit(\Ll)$ projecting to a uniform tubular neighborhood of $\Delta\times\{0\}$. By $G$-equivariance, we have that $\pi$ factors trough $\Ll\times\theta_0$, so we consider $h$ as a function on $B=\R^n$ and $\theta$ as a function of $B\times S^1$ with coordinates $x_1,\dots,x_n,t$. The quasi conformality in these coordinates translates, away from $\Ll\times\theta_0(U_\delta)$, into the conditions
\[
1/c<\frac{|d\theta(\partial_t)|}{|dh(\partial_{x_n})|}<c,
\]
and
\[
\frac{|d\theta(\partial_{x_n})|}{|dh(\partial_{x_n})|}<c.
\]
We have that $\theta$ is initially defined on a neighborhood of $\Delta\times\{0\}\in\R^n\times S^1$ and  that these inequalities are satisfied with uniform constants at least near the boundary of that neighborhood. It is straightforward by \eqref{eqNabBo} that $\theta$ can be appropriately extended to satisfy the inequality.

We have thus extended $\pi$ in an admissible manner. Lemma~\ref{lmAdmLGEx} now guarantees a $(\sigma,L)$-tame $J$ for which $\pi$ is $J$-holomorphic. It remains to show contractibility. The choices going in to the construction in Theorem \ref{lmTameStNoFo} are evidently in a contractible set. The same is true for the choice of a tame local normal form since any two such are related by a flat Lipschitz diffeomorphism in the base and a Lipschitz Hamiltonian isotopy in the fiber. The choice of admissible $\pi$ is evidently contractible. Finally, the extension of $J$ so that $\pi$ $J$-holomorphic is contractible by Theorem~\ref{lmTameComp}. The additional condition of tameness relative to $\Ll,\sigma$ maintains the contractibility.

\end{proof}

\section{Energy and $C^0$ estimates for Lipschitz Floer data}\label{AppLipFloer}
Let $(M,\omega)$ be a symplectic manifold. Fix an $\omega$-compatible almost complex structure $J_0$. For $z\in D:=B_1(0)\subset \C$ let $J_z$ be a tame almost complex structure such that for some $C>1$ independent of $z$ and for all tangent vectors $v$ we have 
$$ \frac1{C}<\frac{g_{J_z}(v,v)}{g_{J_0}(v,v)}<C$$.
When this holds we say $g_{J_z}$ is $C$-equivalent to $g_{J_0}$. Suppose $J_z$ depends smoothly on $z$. Let
\[
\mathfrak{H}:=F_zds+G_zdt
\]
be a $1$-form with values in smooth Hamiltonians. Suppose
\begin{equation}\label{EqMonHom}
\partial_sG_{s,t}-\partial_tF_{s,t}\geq 0.
\end{equation}
Suppose further  that for fixed $z=s+it$ the functions $F,H,\partial_tF,\partial_sG$ are all Lipschitz functions on $M$ with Lipschitz constants uniform in $z$ all with respect to $g_{J_0}$ and therefore, for any $z\in D$, with respect to $g_{J_z}$. Such Floer data are referred to as \textit{admissible}. We will use the term $C$-admissible when the Lipschitz constants are  bounded by $C$.

\begin{rem}
We comment that in the usual definition of the wrapped Fukaya category of a Liouville domain, the Floer data are \textit{not} Lipschitz with respect to the chosen almost complex structure. We hasten to assure the reader that this causes no incompatibility. Namely, using a zig-zag homotopy argument as in \cite[Theorem 4.6]{Groman15}, one can show that as long as one uses geometrically bounded almost complex structures, the choice of almost complex structure is inconsequential. The only thing that is of consequence is the asymptotic growth rate of the Hamiltonians. In this respect, the Lipschitz condition is merely a softening of the usual linear-at-infinity condition.
\end{rem}
For a $1$-form $\gamma$ on $\Sigma$ with values in $u^*TM$ write
\[
\gamma^{0,1}:=\frac12\left(\gamma+J\circ\gamma\circ j_\Sigma\right).
\]
A \textit{Floer solution} on $D$ with this data is a map $u:D\to M$ satisfying Floer's equation
\begin{equation}
(du-X_\mathfrak{H}(u))^{0,1}=0.
\end{equation}


The \textbf{geometric energy} of $u$ on a subset $S\subset D$ is defined as
\[
E_{\mathfrak{H},J}(u;S):=\frac1{2}\int_S\|du-X_\mathfrak{H}\|_J^2ds\wedge dt.
\]

Given an admissible Floer datum $(\mathfrak{H},J)$ define a $2$-form on $D\times M$ by $\omega_{\mathfrak{H}}:=\pi_2^*\omega+d\mathfrak{H}$. Given a Floer solution $u:D\to M$ denote by $\tilde{u}=id\times u:D\to D\times M$ its graph.

\begin{lm}\label{lmTopGeoEnEst}
For any $(\mathfrak{H},J)$-Floer solution $u:D\to M$ satisfying \eqref{EqMonHom} we have
\[
E(u;S)-\sup|\{F,G\}Area(D)|\leq\int_{D}\tilde{u}^*\omega_\mathfrak{H}.
\]
The expression on the right is called the topological energy and denoted by $E_{top}(u)$.
\end{lm}
\begin{proof}
Using the Floer equation and denoting by $d'$ the exterior derivative in the $M$ direction,
\begin{align}
\|du-X_\mathfrak{H}\|^2ds\wedge dt& = \omega(\partial_tu-X_G,X_F-\partial_su)ds\wedge dt\notag\\
&=u^*\omega+(d'G(\partial_su)-d'F(\partial_tu)-\omega(X_G,X_F))ds\wedge dt\notag\\
&=u^*\omega+d\mathfrak{H}-((\partial_sG-\partial_tF-\{F,G\})\circ u) ds\wedge dt\notag\\
&\leq u^*\omega+d\mathfrak{H}+|\{F,G\}|ds\wedge dt.\notag
\end{align}
\end{proof}

Denote the Floer datum $(\mathfrak{H},J)$ by $\mathfrak{F}$. Consider the almost complex structure
\[
J_{\mathfrak{F}}(z,x):=\left(\begin{matrix} j_{\Sigma}(z) & 0  \\ X_\mathfrak{H}(z,x)\circ j_\Sigma(z)-J(z,x)\circ X_\mathfrak{H}(z,x) & J(x) \end{matrix}\right)
\]
On $D\times M$. Then the graph of a solution to the Floer equation determined by $\mathfrak{F}$ is $J_{\mathfrak{F}}$ holomorphic. Fix an area form on $\omega_\Sigma$ on $\Sigma$ and for $a>0$ write
\[
\tilde{\omega}_a:=\omega_{\mathfrak{H}}+a\omega_\Sigma.
\]
\begin{lm}
For $a>\sup|\{G_z,F_z\}|$, the almost complex structure $J_{\mathfrak{F}}$ is tamed by the form $\tilde{\omega}_a$.
\end{lm}
\begin{proof}
 See Lemma 5.1 and remark 5.2 in \cite{Groman15}. Alternatively, examine eq. \eqref{eqGromovMetric} below and apply \eqref{EqMonHom}. Note that by the Lipschitz condition, we have that $\{G_z,F_z\}$ is indeed bounded.
\end{proof}
Fix a large enough $a$ and denote the corresponding metric by $g_{\mathfrak{F}}$.

\begin{lm}\label{lmTaming}
Fix a point $p_0\in M$. Denote by $g$ the product metric on $D\times M$. There are constants $c_1<1<c_2$ such that for any $R>0$ for any $p\in B_R(p_0)$ and any $v\in \R^2\oplus T_pM$ we have
\[
c_1\|v\|_g\leq \|v\|_{g_{\mathfrak{F}}}\leq c_2\sqrt{R}\|v\|_g.
\]
\end{lm}
\begin{proof}
Let $\alpha,\beta$ be the one forms on $M$ giving the $g_J$ inner product with $X_F$ and $X_G$ respectively. Then we have\footnote{The juxtaposition of two $1$-forms denotes their symmetric product.}
\begin{gather}\label{eqGromovMetric}
g_{J_{\mathfrak{F}}}=g_J-\alpha ds-\beta dt+g_J(X_F,X_G)dsdt+\\
\qquad\qquad\|X_F\|^2dt^2+\|X_G\|^2ds^2+\notag\\
\qquad\qquad(\partial_sG_{s,t}-\partial_tF_{s,t}-\{F_{s,t},G_{s,t}\}+a)(ds^2+dt^2)\notag.
\end{gather}
Let $g_0=(\partial_sG_{s,t}-\partial_tF_{s,t})(ds^2+dt^2)$ and $g_1=g_{\mathfrak{F}}-g_0$. Then $g_1$ is uniformly equivalent to the metric $g$ once $a$ is chosen large enough relative to the Lipschitz constants of $F$ and $G$. To see this let $g_2$ be obtained from $g_1$ by subtracting the cross term
\[
g_3=-\alpha ds-\beta dt+g_J(X_F,X_G)dsdt.
\]
Note that $\|X_F\|,\|X_G\|$ are $O(1)$. In particular $a$ can be taken large enough to dominate the $\{F,G\}$ term. So, the change from $g$ to $g_2$ amounts to rescaling the vectors $\partial_t,\partial_s$ by an amount which is bounded above and below. Making $a$ large enough, we also can bound the ratio $\frac{g_3(v,w)}{|v|_{g_1}|w|_{g_1}}$ away from $1$ , so $g$ is also uniformly equivalent to $g_1$.

On the other hand, let $A\in Hom(T_pv,T_pv)$ be the map
\[
\id + \frac1{\|\partial_t\|_{g_1}}g_0(\cdot,\partial_t)\otimes\partial_t+\frac1{\|\partial_s\|_{g_1}}g_0(\cdot,\partial_s)\otimes\partial_s.
\]
Then $g_{\mathfrak{F}}=g_1(\cdot,A\cdot)$. Since $A$ is $O(R)$, the claim follows.
\end{proof}

\begin{df}
We say that a metric $g$ is \textit{$(\delta,c)$-isoperimetric} at a point $p\in M$ if for any closed curve $\gamma$ contained in $B_{\delta}(p)$ there is a disc $D$ bordered by $\gamma$ and such that
\[
Area(D)\leq c\ell(\gamma)^2.
\]
Similarly, a point $p\in L$ is \textit{$(\delta,c)$-isoperimetric} with respect to $g$ if, in addition, for any chord $\gamma:[0,1]\to M$ with endpoints on $L$ and contained in $B_{\delta}(p)$ there is a half disc bordered by $\gamma$ and a path $\tilde{\gamma}\subset L$ and such that
\[
Area(D)\leq c\ell(\gamma)^2.
\]
\end{df}

\begin{lm}\label{lmIsopGrowth}
Suppose $g,L$ are $(\delta,c)$-isoperimetric at some $p\in M$ and let $g'$ be such that for some $a_1,a_2>1$ we have that
\[
\frac1{a_1}\|\cdot\|_g\leq \|\cdot\|_{g'}\leq a_2\|\cdot\|_g.
\]
Then $g',L$ are $(\delta/a_1,a_1^2a_2^2c)$-isoperimetric.
\end{lm}
\begin{proof}
This follows from the inequalities $\frac1{a_2}\ell_{g'}\leq\ell_{g}\leq a_1\ell_{g'}$ and $Area(g')\leq a_2^2 Area(g).$
\end{proof}

We first recall the monotonicity inequality for $J$-holomorphic curves with Lagrangian boundary. The following lemma and its proof are taken form \cite{Sikorav94}. We include them here to emphasize the dependence on parameters.
\begin{lm}[\textbf{Monotonicity}\cite{Sikorav94}]
Suppose $J$ is such that $g_J$ is $(\delta,c)$-isoperimetric at $p$. Then any $J$-holomorphic curve $u$ passing through $p$ and with boundary in $M\setminus B_{\delta}(p)$ satisfies
\[
E(u;u^{-1}(B_{\delta}(p)))\geq \frac{\delta^2}{2c}.
\]
If $p\in L$ and $J,L$ are $(\delta,C)$-isoperimetric, the same holds if we require instead $B_\delta(p)\cap u=\emptyset$ that $\partial u \cap B_{\delta}(p)\subset L$.
\end{lm}
\begin{proof}
For $t\in[0,\delta]$ let $A(t):=E(u;u^{-1}(u\cap B_t(p)))$. Then $A(t)$ is absolutely continuous and $A'(t)=\ell_t$ almost everywhere. So,
\[
(\sqrt{A})'=\frac{A'}{2\sqrt{A}}=\frac{\ell_t}{2\sqrt{A}}\geq\frac1{2\sqrt{c}}.
\]
Integrating from $0$ to $\delta$ and squaring both sides gives the claim.
\end{proof}

\begin{tm}[\textbf{Domain local estimate}]\label{tmDiamEst}
Let $L$ be a Lipschitz Lagrangian, let $K$ be  a compact set and let $\mathfrak{F}$ be an admissible Floer datum. For any $E>0$ there is an $R=R(E,K,L,\mathfrak{F})$ such that the following hold. Denote by $D_{1/2}\subset D$ the disc of radius $1/2$.
\begin{enumerate}
\item
Let $u$ be a solution to Floer's equation on $D$ and suppose
$E(u;D)\leq E$ and $u(D_{1/2})\cap K\neq\emptyset.$ Then
\[
u(D_{1/2})\subset B_R(K).
\]
\item
Denote by $D^+$ the intersection of the disc with the upper half plane. Let $u:(D^{+},\partial D^+)\to (M,L)$ be a solution to Floer's equation such that $E(u;D^+)\leq E$ and $u(D^+_{1/2})\cap K\neq\emptyset.$ Then
\[
u(D^+_{1/2})\subset B_R(K)
\]
\end{enumerate}
\end{tm}

\begin{proof}
Using the Gromov trick we consider $u$ as a $J$-holomorphic curve for which the monotonicity inequality applies. In the case of a general Floer datum, the constants of the monotonicity inequality are not uniformly bounded due to the presence of the derivatives $\partial_sG,\partial_tF$ in the expression  \eqref{eqGromovMetric} for the metric $g_{J_{\mathfrak{F}}}$. But since the geometry of the underlying manifold is uniformly bounded, Lemmas \ref{lmTaming} and \ref{lmIsopGrowth} imply that the monotonicity constant for $g_{J_{\mathfrak{F}}}$ grows at most as the distance. This means that the geometry of $g_{J_{\mathfrak{F}}}$ is i-bounded in the terminology of \cite{Groman15}. The proof is now the essentially the same as that of Theorem 4.10 in \cite{Groman15}. 
\end{proof}
We now turn to the global situation. Let $D_{k,l}$ be a disc with $k$ interior and $l$ boundary punctures \textit{and a fixed conformal structure}. Denote the set of punctures by $P=P_e\cup P_i$ decomposed into boundary and interior punctures respectively. Order $P_e$ counterclockwise and label the components $\partial_i$ of $\partial D_{i,k}\setminus P_e$ by their left boundary point. Further decompose $P_e,P_i$ into subsets $P_{e,\pm},P_{i,\pm}$ designating inputs and outputs. Near the $j$th puncture fix striplike coordinates $\epsilon_j:[0,1]\times(-\infty,0)\to D_{k,l}$ for inputs and $\epsilon_j:[0,1]\times(0,\infty)$ for outputs. Define cylindrical ends similarly. Our convention is $S^1=[0,1]/0\sim 1$. As before, fix a geometrically bounded $J_0$. Label the components $\partial_i$ of $\partial D_{i,k}\setminus P_e$ by Lagrangians $L_i$ which are Lipschitz with respect to $J_0$. Let $\mathfrak{H}$ be a $1$-form on $D_{k,l}$ with values in Lipschitz Hamiltonians. Assume there are Hamiltonians $H_i$ for $i\in P$ such that near the $i$th puncture, we have $\epsilon_i^*\mathfrak{H}=H_idt$. In particular, $\{\mathfrak{H},\mathfrak{H}\}$ is compactly supported on $D\setminus\{P_e\cup P_i\}$.

For Lagrangians $L_0, L_1$ we say that a Hamiltonian $H$ \textit{$\delta$-separates $L_0$ from $L_1$ near infinity} if there is a $\delta>0$ such that
\begin{equation}\label{DissipEst}
d(L_1\setminus K,\psi_{H_i}(L_0))>\delta.
\end{equation}
Similarly, we say that $H$ $\delta$-separates the diagonal if
\[
d(p,\psi_H(p))>\delta,
\]
for any $p\in M\setminus K$. Here $\psi_H$ is the time $1$ flow of $H$.

 The main ingredient needed to go from the local estimate to a global one is the following
\begin{lm}\label{lmRectEst}
Suppose $u:[a,b]\times[0,1]\to M$ is a solution to Floer's equation
\[
\partial_su+J(\partial_tu-X_H)
\]
for a Lyapunov Hamiltonian $H.$ Then
\[
d(u(a,\cdot),u(b,\cdot))d(u(\cdot,1),\psi_H(u(\cdot,0))<cE_{geo}(u;[a,b]\times[0,1]),
\]
for some constant $c$ depending on $H$.
\end{lm}
The proof relies on the following estimate
\begin{lm}\label{lmLyapunovEst}
Suppose $H$ is Lyapunov. Then there is a $C=C(H)>0$ such that for any path $\gamma:[0,1]\to M$ we have
\[
d(\gamma(1),\psi_H^1(\gamma(0)))^2< C\int_0^1\|\gamma'(t)-X_H\circ\gamma\|^2.
\]
\end{lm}
\begin{proof}[Proof of Lemma~\ref{lmRectEst}].
We start by observing that
\[
b-a\geq\frac{\int_0^1d(u(a,t),u(b,t))^2}{\int_a^b\|\partial_su\|^2}.
\]
which follows from Cauchy-Schwartz and the inequality
\[
d(u(a,t),u(b,t))\leq \ell(u([a,b]\times\{t\}))=\int_a^b\|\partial_su\|dt.
\]
Write $\delta:=d(u(\cdot,1),\psi_H(u(\cdot,0)))$. Then by Lemma \ref{lmLyapunovEst} we have
\[
\min_{s\in[a,b]}\|\dot{u}_s-X_H\|^2\geq c\delta.
\]
Thus
\[
E_{geo}(u)\geq c(b-a)\delta^2\geq c\delta^2\frac{d(u_a,u_b)^2}{E_{geo}}.
\]
The last inequality relies on the fact that $u$ solves Floer's equation. Rearranging and taking the inequality and taking square roots produces the desired estimate.
\end{proof}
\begin{proof}[Proof of Lemma~\ref{lmLyapunovEst}]
By Cauchy Schwartz
\[
\|\gamma'-X_H\circ\gamma\|_{L^1}\leq\|\gamma'-X_H\circ\gamma\|_{L^2}.
\]

Let $r$ such that $\gamma([0,1])$ is covered by charts
$\{U_i\subset M,\psi_i:B_{2r}(0)\to U_i\}$ which are bi-Lipschitz with Lipschitz constant $2$. Moreover the cover has Lebesgue number $r$. By compactness of $\gamma([0,1])$, there is a $K$ such that for any $i$ the vector field $d\psi_i^{-1}X_H$ considered as a map $B_{r_i}(0)\to \R^{2n}$ is $K$-Lipschitz.

Let
\[
g(t)=\|\gamma'(t)-X_H\circ\gamma(t)\|
\]
and let $f(t)=\int_0^tg(s)ds$. Let
\[
\Delta t\ll r\min\left\{\frac1{K},\frac1{\sup\|X_H\|},\frac1{\max g(t)}\right\}.
\]
 Suppose $N:=\frac1{\Delta t}$ is an integer. Suppose $\Delta t$ is made smaller still so that $f(t)$ has an approximation by a piecewise linear function $h(t)$ such that
\[
g(t)<h'(t),
\]
and such that $h$ is linear of slope $\epsilon_i$ on the intervals $[\frac{i}{N},\frac{i+1}{N}]$. Let $t_i=\frac{i}{N}$. Let $\gamma_i(t):=\psi_H^t(\gamma(t_i))$ and let $x_i=\gamma_i(1-i\Delta t)$. Then $\psi_H^t(\gamma(0))=x_0$ and $\gamma(1)=x_N$. Denoting $\Delta x_i:=d(x_{i},x_{i-1})$ for $i=1,\dots N$, we have the obvious estimate $d(x_0,x_N)\leq\sum\Delta x_i$. Let $y_i=\gamma(t_i)$ and $z_i=\psi_H^{\Delta t}(y_{i-1})$. Then by the Lyapunov condition,
\[
\Delta x_i\leq Ce^{ z (1-i\Delta t)}d(y_i,z_i).
\]
On the other hand we have a standard ODE estimate
\[
d(y_i,z_i)\leq \frac{\epsilon_i}{K}(e^{K\Delta t}-1).
\]
The last expression is $\leq 2\epsilon_i\Delta t$ since $\Delta t\ll \frac1K$.
So,
\[
d(x_0,x_N)\leq\sum_{i=1}^N 2C \epsilon_ie^{ z(1-i\Delta t)}\Delta t
\]
The last expression approximates the integral
\[
2C\int_0^1 h'(t)e^{ z(1-t)}dt.
\]
The Cauchy-Schwarz inequality 
\[
2C\int_0^1 h'(t)e^{ z(1-t)}dt\leq 2C\sqrt{\int_0^1(h'(t))^2dt}\sqrt{\int_0^1e^{ 2z(1-t)}dt}.
\]
now produces the required estimate.
\end{proof}
\begin{df}\label{dfCadmFl}
We call $\mathfrak{F}=(\mathfrak{H},J)$ \textit{admissible} if $\mathfrak{F}$ is Lipschitz and for each end $j$, we have $H_j$ is
\begin{enumerate}
    \item Lyapunov,
    \item separates the diagonal for any $j\in P_i$,
    \item separates $L_j$ from $L_{j+1}$ for any $j\in P_{e,-}$ and $L_{j+1}$ from $L_j$ for $j\in P_{e,+}$, and
    \item satisfies the monotonicity condition
\begin{equation}\label{EqMonoTonicity}
d\mathfrak{H}\geq 0.
\end{equation}
\end{enumerate}
For a real number $c>0$ we will say that The datum is $c$-admissible if all the constants such as, Lipschitz, Lyapunov,and separation data, are appropriately bounded in terms of $c$.
\end{df}
\begin{tm}\label{tmGlobalEst}[\textbf{Domain-global estimate}]
Fix an admissible Floer datum $\mathfrak{F}$ and a compact set $K$. Then there is a monotone function $R=R_{\mathfrak{F}}:\R_+\to\R_+$ such that any solution to Floer's equation on $D_{k,l}$ which intersects $K$ satisfies
\[
Diam(u) \leq R(E_{geo}(u))\leq R(E_{top}(u)+C),
\]
where $C$ is a constant depending on $\mathfrak{F}$.
\end{tm}
\begin{proof}
Let $A\subset D_{k,l}$ be open connected and precompact such that $D_{k,l}\setminus A$ consists of ends. Assume first that $u(A)\cap K\neq\emptyset$. Covering $A$ by a finite number of discs and half discs, the domain local estimate provides an $R_0=R_{A,\mathfrak{H}}$ such that $u(A)\subset B_{R_0}(K)$. By admissibility condition, Lemma~\ref{lmRectEst} produces an $R_1=R_1(\delta,E)$ such that for any end $e_i$ any $t\in\R_{\pm}$ we have that $u(\epsilon_i(\{t\}\times(0,1)))\cap B_{R_1}(B_{R_0}(K))\neq\emptyset$. Applying the domain local estimate again, we obtain an $R_2$ such that $u$ is entirely contained in $B_{R_2+R_1+R_0}(K)$. If $u$ intersects $K$ in a point $u(x)$ for some $x\not\in A$ then $x$ is on some end, and we may write $x=\epsilon_i(s_0,t_0)$. The domain local estimate bounds the diameter  of $u_{s_0}$. Thus we can apply Lemma~\ref{lmRectEst} to obtain an a-priori estimate the distance $d(K,u(A))$. We can now proceed as before.

For the right hand estimate, note that $\{\mathfrak{F},\mathfrak{F}\}$ is assumed to be supported on a compact subset $A\subset D_{k,l}$. The Floer data is assumed to be Lipschitz with some constant $c$.  Lemma \ref{lmTopGeoEnEst} implies that for any Floer solution $u$ we have $E_{geo}(u)\leq E_{top}(u)+c^2Area(A)$.
\end{proof}
\begin{rem}
We have formulated the admissibility condition and the global estimate for a punctured disc with fixed conformal structure. In general, e.g, the proof of associativity, we need to alow the conformal structure to vary. It is then required to choose Floer data for the entire moduli space. As is spelled out in detail in \cite{SeidelPL} this involves, among other things, choosing a thick thin decomposition near the boundary of the moduli space and assigning striplike and cylindrical coordinates on the thin part. The extension of the condition of admissibility to this setting is the obvious one. Namely, the Lipschitz estimates should be uniform on the moduli space, and \eqref{DissipEst} should hold for each thin component. The uniform global estimate for the entire family is then an obvious variant.
\end{rem}
Examining the proof of the global estimate we obtain the following very useful variant.
\begin{lm}\label{lmTLDG}[\textbf{target-local, domain-global estimate}]
There is a function $R=R(c,E)$ with the following significance. Let  $(\mathfrak{H},J_z)$ be a $c$-admissible Floer datum on a Riemann surface with cylindrical ends. Then any  Floer solution of energy at most $E$ has diameter at most $R$. Moreover, suppose the Floer datum $(\mathfrak{H},J_z)$ satisfies the inequality \eqref{EqMonoTonicity} everywhere, and $J_z$ induces a complete metric, but the other conditions for $c$-admissibility hold only on $B_R(p)$. Then any Floer solution $u$ meeting $p$ is contained in $B_R(p)$.
\end{lm}

\bibliographystyle{amsabbrvc}
\bibliography{ref}
\end{document}